\documentclass[reqno]{amsart}
\usepackage{amsmath, amsthm, amssymb, amstext}

\usepackage{hyperref,xcolor}
\hypersetup{pdfborder={0 0 0},colorlinks}
\usepackage{enumitem}
\allowdisplaybreaks
\usepackage{graphicx}
\usepackage{caption}
\usepackage{mathtools}
\usepackage{tikz}

\newtheorem{theorem}{Theorem}[section]
\newtheorem{remark}[theorem]{Remark}

\newtheorem{corollary}[theorem]{Corollary}
\newtheorem{definition}[theorem]{Definition}

\newtheorem{lemma}{Lemma}[section]

\DeclareMathOperator*{\ess}{\text{ess}}
\newcommand{\N}{\mathbb{N}}
\newcommand{\R}{\mathbb{R}}
\newcommand{\eps}{\varepsilon}
\newcommand{\ph}{\varphi}
\newcommand{\into}{\int_{\Omega}}
\newcommand{\intr}{\iint\limits_{\R^N\times\R^N}}
\newcommand{\inta}{\iint\limits_{\substack{x,y\in\R^N \\ \abs{x-y}\leq1}}}
\newcommand{\intb}{\iint\limits_{\substack{x,y\in\R^N \\ \abs{x-y}\geq1}}}
\newcommand{\pfrac}{\frac{|u(x)-u(y)|^p}{|x-y|^{N+sp}}}

\newcommand{\qform}{\mathcal{E}_{L_{\Delta_p}}}

\newcommand{\logP}{L_{\Delta_p}}
\newcommand{\nW}{\|u\|_{\WO}}
\newcommand{\nWs}{\|u_s\|_{\WO}}
\newcommand{\nWp}{\|\ph\|_{\WO}}
\newcommand{\nX}{\|u\|_{\XO}}
\newcommand{\WO}{W_0^{s,p}(\Omega)}
\newcommand{\XO}{X_{\logP}(\Omega)}
\newcommand{\nP}{\|u\|_{L^p(\Omega)}}
\renewcommand{\l}{\left}
\renewcommand{\r}{\right}

\def\abs#1{\left|{#1}\right|}

\numberwithin{theorem}{section}
\numberwithin{equation}{section}

\renewcommand{\thetheorem}{\arabic{section}.\arabic{theorem}}
\title[Logarithmic $p$-Laplacian]{Sharp embeddings and existence results for Logarithmic $p$-Laplacian equations with critical growth}

\author[R. Arora]{Rakesh Arora}
\address[R. Arora]{ Department of Mathematical Sciences, Indian Institute of Technology Varanasi (IIT-BHU), Uttar Pradesh 221005, India}
\email{rakesh.mat@iitbhu.ac.in, arora.npde@gmail.com}

\author[J. Giacomoni]{Jacques Giacomoni}
\address[J. Giacomoni]{
LMAP, UMR E2S-UPPA CNRS 5142, Ba\^timent IPRA, Avenue de l’Universit\'e F-64013 Pau, France}
\email{jacques.giacomoni@univ-pau.fr}

\author[H. Hajaiej]{Hichem Hajaiej}
\address[H. Hajaiej]{
Department of Mathematics, California State University, Los Angeles, CA 90032, USA}
\email{hichem.hajaiej@gmail.com }

\author[A. Vaishnavi]{Arshi Vaishnavi}
\address[A. Vaishnavi]{Department of Mathematical Sciences, Indian Institute of Technology Varanasi (IIT-BHU), Uttar Pradesh 221005, India}
\email{arshiv1998@gmail.com}


\begin{document}
\begin{abstract}
In this paper, we derive a new $p$-Logarithmic Sobolev inequality and optimal continuous and compact embeddings into Orlicz-type spaces of the function space associated with the logarithmic $p$-Laplacian. As an application of these results, we study a class of Dirichlet boundary value problems involving the logarithmic $p$-Laplacian and critical growth nonlinearities perturbed with superlinear-subcritical growth terms. By employing the method of the Nehari manifold, we prove the existence of a nontrivial weak solution. 

Lastly, we conduct an asymptotic analysis of a weighted nonlocal, nonlinear problem governed by the fractional $p$-Laplacian with superlinear or sublinear type non-linearity, demonstrating the convergence of least energy solutions to a non-trivial, non-negative least energy solution of a Brezis-Nirenberg type or logistic-type problem, respectively, involving the logarithmic $p$-Laplacian as the fractional parameter $s \to 0^+$.

The findings in this work serve as a nonlinear analogue of the results reported in \cite{Angeles-Saldana, Arora-Giacomoni-Vaishnavi, Santamaria-Saldana}, thereby extending their scope to a broader variational framework.
\end{abstract}
\maketitle
\tableofcontents
\section{Introduction}
The theory of nonlocal operators has witnessed remarkable progress in recent years, largely motivated by their wide-ranging applications in fields such as probability, phase transitions, fluid mechanics, and image processing (see, for instance, \cite{Caffarelli}, \cite{Nezza-Palatucci-Valdinoci}, \cite{Vazquez}, and the references therein). Among these operators, the fractional Laplacian and its nonlinear extension, the fractional $p$-Laplacian, play a particularly prominent role. For a smooth function $u$, the fractional $p$-Laplacian is defined as:
\[-\Delta_p^s u(x)= C_{N,s,p} \  \text{P.V.} \ \int_{\R^N} \frac{|u(x)-u(y)|^{p-2}(u(x)-u(y))}{|x-y|^{N+sp}} ~dy, \]
where $N\geq 1$, $p \in (1, \infty)$, $s \in (0, 1)$ and $C_{N,s,p}$ is a normalizing constant. Here, P.V. denotes the Cauchy principal value.
The fractional $p$-Laplacian can be regarded as a nonlocal analogue of the classical $p$-Laplace operator. Considerable attention has been devoted to studying the existence, uniqueness, and regularity of solutions to elliptic problems involving this operator. We refer the reader to \cite{Brasco-Parini}, \cite{Castro-Kuusi-Palatucci}, \cite{Franzina-Palatucci}, \cite{Lindgren} and \cite{Pezzo-Quaas}, among others, for detailed developments in this direction. 

The research related to nonlocal operators with zero-order kernels 
has received significant attention. The reason for this is twofold - involvement of interesting mathematical structures, see \cite{Chen-Weth}, \cite{Felsinger-Kassmann-Voigt}, \cite{Feulefack-Jarohs}, \cite{Foghem}, \cite{Frank-Konig-Tang}, \cite{Jarohs-Saldana-Weth}, \cite{Jarohs-Saldana-Weth-1}, \cite{Kassmann-Mimica} and \cite{Temgoua-Weth}, and a wide range of physical applications, see \cite{Pellacci-Verzini}, \cite{Sprekels-Valdinoci}, \cite{Sikic-Song-Vondracek} and the references therein. 

Recently, the logarithmic Laplacian $L_\Delta$, a singular integro-differential operator with zero-order kernels, introduced in the seminal work \cite{Chen-Weth}, is given by
\[
\begin{split}
    L_\Delta u(x) & =\frac{d}{ds}\bigg|_{s=0} (-\Delta)^s u\\&= c_N \int_{\mathcal{B}_1(x)} \frac{u(x)-u(y)}{\abs{x-y}^N} ~dy -c_N\int_{\R^N\setminus\mathcal{B}_1(x)} \frac{u(y)}{\abs{x-y}^N} ~dy + \rho_Nu(x),
    \end{split}
\]
where $\mathcal{B}_1(x) \subset \mathbb{R}^N$ denotes the unit ball centered at $x$ and
\[
c_N := \pi ^{\frac{-N}{2}}\Gamma(\frac{N}{2}), \quad \rho_N := 2\ln 2+ \psi(\frac{N}{2})-\gamma, \quad  \gamma := -\Gamma^{'}(1),\]
 $\gamma$ being the Euler-Mascheroni constant and $\psi := \frac{\Gamma^{'}}{\Gamma}$, the digamma function. 

The study of the logarithmic Laplacian has advanced considerably since its introduction in \cite{Chen-Weth}, where the authors developed a variational framework for analyzing Dirichlet problems on bounded domains. Moreover, they investigated the eigenvalue problem associated with the operator $L_\Delta$, and established a Faber–Krahn-type inequality, maximum principles, and continuity properties of weak solutions to Poisson problems. 

Building on the framework established in \cite{Chen-Weth}, the authors in \cite{Santamaria et. al, Santamaria-Saldana} and \cite{Arora-Giacomoni-Vaishnavi} derived a sharp logarithmic Sobolev inequality and optimal continuous and compact embeddings. Consequently, they proved the existence and uniqueness of least energy solutions to boundary value problems involving the logarithmic Laplacian. Additionally, they investigated the limiting behavior of solutions to boundary value nonlinear problems involving the fractional Laplacian of order $2s$ when the parameter $s$ tends to zero. Such investigations are relevant to population dynamics, optimal control, and image processing. A few contributions investigating these applications include \cite{Antil-Bartels}, \cite{Crismale et. al},  \cite{Pellacci-Verzini} and \cite{Sprekels-Valdinoci}. 

For other recent advancements in this direction, we refer the reader to, \cite{Lara-Saldana, Santamaria-Rios-Saldana, Pollastro-Soave} (for regularity results), \cite{Chen-Veron, Feulefack-Jarohs-Weth} (for bounds and asymptotics for eigenvalue problems), \cite{Chen-Zhou, Frank-Konig-Tang} (for semilinear equation in $\R^N$ and $\mathbb{S}^N$), \cite{Fernandez-Saldana} (for Yamabe-type problem), \cite{Chen-Veron-1} (for diffusion equations), \cite{Chen-Hauer-Weth} (for Caffarelli-Silvestre extension problem), \cite{Harrach-Lin-Weth} (for Calder\'on problem), \cite{Santamaria et. al} (for the error estimates and numerical implementation), \cite{Luca et. al} (geometric context of 0-fractional perimeter). 

Very recently, in \cite{Dyda-Jarohs-Sk} extended the operator $L_\Delta$ to its nonlinear counterpart- the logarithmic $p$-Laplacian $\logP$, which is defined as
\[
\begin{split}
    \logP u(x) &= \frac{d}{ds}\bigg{|}_{s=0} (-\Delta_p)^s u(x)\\
    &=C_{N,p} \int_{\mathbb{B}_1(x)} \frac{|u(x)-u(y)|^{p-2}(u(x)-u(y))}{|x-y|^{N}} ~dy\\& \quad+ C_{N,p} \int_{\R^N\setminus \mathbb{B}_1(x)} \frac{|u(x)-u(y)|^{p-2}(u(x)-u(y))-|u(x)|^{p-2} u(x)}{|x-y|^{N}} ~dy\\& \qquad \qquad + \rho_N \  |u(x)|^{p-2}u(x),
\end{split}
\]
where $N\geq 1$, $p \in (1, \infty)$, $s \in (0, \frac{1}{2})$, $u\in C_c^\alpha(\R^N), \ \alpha \in \R^+, \ x \in \R^N,$ $C_{N,p}$ and $ \rho_N$ are constants (defined below). In \cite{Dyda-Jarohs-Sk}, the authors demonstrated that the logarithmic $p$-Laplacian $\logP$ naturally arises as the first-order derivative of the fractional $p$-Laplacian and provided an explicit integral representation for this operator. They developed a variational framework to study Dirichlet problems involving $\logP$ on bounded domains and explored the relationship between the first Dirichlet eigenvalue and eigenfunction of the fractional $p$-Laplacian and those of the logarithmic $p$-Laplacian. In addition, they established a Faber–Krahn inequality for the first eigenvalue of $\logP$, proved maximum principles, and derived a boundary Hardy-type inequality for the associated energy space. The study of problems involving the logarithmic $p$-Laplace operator remains in an early stage of development, with many key questions about existence, uniqueness, properties of solutions, and asymptotic analysis still open. This work aims to contribute to closing some of these gaps. To the best of our knowledge, \cite{Dyda-Jarohs-Sk} is the only study in the literature that has explored the logarithmic $p$-Laplacian.

Inspired by the aforementioned works, we plan to study the following elliptic equation involving logarithmic $p$-Laplacian
\begin{align}
    \label{main_problem} 
    \tag{$L_{\Delta_p}$}
    \logP u = g(x,u)  + \mu |u|^{p-2} u \ln |u| 
 \ \ \text{in}  \ \Omega, \qquad u=0 \ \text{in} \ \R^N \setminus \Omega, 
\end{align}
where $\Omega \subset \mathbb{R}^N$ is a bounded open set, $g$ is a Carath\'eodory function with suitable growth conditions and $\mu$ is a parameter. Such kind of problems appear as a limiting case of the weighted nonlocal Dirichlet problem involving the fractional $p$-Laplacian
\begin{align}
    \label{frac_pblm}\tag{$L_{s,p}$}
    (-\Delta_p)^s u = a(s,x)|u|^{q(s)-2}u  \ \text{in}  \ \Omega, \qquad u=0 \ \text{in} \ \R^N \setminus \Omega,
\end{align} 
where $a$ and $q$ are $C^1$ functions such that $a(s,\cdot) \to 1$ and $q(s) \to p$ as $ s \to 0^+.$ Heuristically, by passing to the limit $s\to 0^+$ on both sides of equation \eqref{frac_pblm}, we get \[
(-\Delta_p)^s u \to |u|^{p-2} u \quad \text{and} \quad a(s,x) \abs{u}^{q(s)-2} u \to |u|^{p-2} u.\]
Since this gives no information on the limiting profile of problem \eqref{frac_pblm}, we have to rely on the first-order expansion (in $s$) on both sides of \eqref{frac_pblm}, which naturally leads us to the logarithmic Laplace operator $\logP$ and leads to the study of the following limiting problem 
\begin{align}
    \label{lim_pblm}\tag{$L_{\text{lim}}$} 
    \logP u = a'(0,x)|u|^{p-2} u  + q'(0) |u|^{p-2} u \ln |u| 
 \ \ \text{in}  \ \Omega, \qquad u=0 \ \text{in} \ \R^N \setminus \Omega, 
\end{align}
which is a particular case of the problem \eqref{main_problem} when $g(x,u) = a'(0,x) u$ and $\mu = q'(0).$ Moreover, depending on the sign of the parameter $\mu$, the problem \eqref{main_problem} can also be seen as the nonlocal counterpart of the famous Brezis-Nirenberg problem and logistic-type problem for nonlocal operators with zero-order kernels. 

To investigate the problem \eqref{main_problem}, we first setup the variational framework for the problem \eqref{main_problem}, but as per the existing literature (\cite[Corollary 6.3]{Foghem}), it is only known that the function space $\XO$ (associated with the logarithmic $p$-Laplacian) is compactly embedded into $L^p(\Omega)$, which is not enough to control the nonlinear growth terms. In order to handle the logarithmic nonlinear terms, we derive a new $p$-logarithmic Sobolev inequality (see Theorem \ref{log-Sob-ineq} below) for the functions in the energy space $\XO$. For $p=2$, this inequality is shown by Beckner in \cite[Theorem 3]{Beckner} for functions in the Schwarz space via the Fourier Transform and in \cite[Proposition 3.8]{Santamaria-Saldana} via the asymptotics of the best constant in the fractional Sobolev inequality. Such a type of inequality is also known as Pitt's type inequality in the literature. However, for the case $p \neq 2$, neither the Fourier transform technique is applicable nor the exact value of the best constant in the $p$-fractional Sobolev inequality is known, which makes it difficult to adopt the arguments as in \cite{Beckner} and \cite{Santamaria-Saldana}. To overcome this difficulty, we derive new asymptotics of the normalizing constant $C_{N,s,p}$ and the constants involved in the fractional Sobolev inequality in \cite[Corollary 4.2]{Frank-Seiringer} and \cite[Theorem 1]{Maz'ya-Shaposhnikova}, which play a crucial role in our analysis. To the best of our knowledge, no such inequality has been established in the past for the case $p \neq 2$ and for functions belonging to the space $\XO$.  

A detailed examination of this inequality leads us to investigate the embedding of the function space $\XO$ into an Orlicz-type space, which further allows us to handle the nonlinear term $g(x,u)$ in \eqref{main_problem}. Pursuing in this direction, we prove optimal continuous and compact embeddings of this function space into Orlicz-type spaces. The sharpness of these embeddings is demonstrated through explicit functions exhibiting nontrivial logarithmic scaling. These embedding results are of independent interest from a functional analytic perspective, as they provide a better understanding of the structure of the space $\XO$. Moreover, these results can also be viewed as a nonlinear extension of the embedding results established in our previous work \cite{Arora-Giacomoni-Vaishnavi}.

Next, as an application of the $p$-logarithmic Sobolev inequality and embedding results, we show the existence and uniqueness results for the problem \eqref{main_problem} depending upon the range of the parameter $\mu.$ 

For the case $\mu \in (0, \frac{p^2}{N})$, in view of the embeddings in Theorem \ref{thm:embd:results} and \eqref{eq:relation:embd}, the problem \eqref{main_problem} can be viewed as a Brezis-Nirenberg problem for nonlocal operators with zero-order kernels, where the logarithmic term has a critical growth and the nonlinearity $g(x,\cdot)$ acts as superlinear-subcritical growth perturbation. The Brezis–Nirenberg problem has been extensively studied for both the Laplacian \cite{Brezis-Nirenberg} and the fractional Laplacian \cite{Servadei-Valdinoci}. Although, as in \cite{Brezis-Nirenberg, Servadei-Valdinoci}, a similar geometric approach applies to our problem \eqref{main_problem}, the main challenge is to ensure that the limit of the minimizing sequence yields a nontrivial solution, making the use of standard variational methods particularly delicate. Additionally, due to the presence of the critical growth term in \eqref{main_problem}, the compactness properties of the embeddings of the energy space $\XO$ are not available. To tackle these difficulties, we employ the method of the Nehari manifold and exploit the $p$-logarithmic Sobolev inequality shown in Theorem \ref{log-Sob-ineq} and prove the existence of a least energy solution. 

For the case $\mu \in (-\infty,0)$, the corresponding energy functional turns out to be coercive. Thus, we get a global minimizer of the energy (which is the solution to the problem) via the variational method. By following the arguments as in \cite{Foldes-Saldana}, via convexity by paths, we also derive the uniqueness of the least energy solution. 
The problem \eqref{main_problem} with $\mu \in (-\infty,0),$ can be regarded as the elliptic counterpart of a logistic equation involving the logarithmic $p$-Laplacian. Such problems find their relevance in biological processes. Here, the function $g$ monitors the growth of $u$ and $\mu<0$, models the degradation of $u.$ The reader may consult \cite{Caffarelli-Dipierro-Valdinoci},  \cite{Costa-Drabek-Tehrani},  \cite{Iannizzotto-Mosconi-Papageorgiou} and the references therein for further context and discussion on logistic equations.

To the best of our knowledge, this work presents the first results corresponding to the existence of a solution for problems involving the logarithmic $p$-Laplacian with critical nonlinearity and its lower order perturbation. 

Lastly, we give a nonlinear analogue to existing linear theories in \cite{Angeles-Saldana, Arora-Giacomoni-Vaishnavi, Santamaria-Saldana} and examine the asymptotics of fractional $p$-Laplacian problems as the fractional parameter $s \to 0^+$. For this, we consider the weighted Dirichlet problem \eqref{frac_pblm} in a bounded set
 $\Omega\subseteq \R^N$ with Lipschitz boundary, $N\geq 1$, $p \in (1, \infty)$, $s \in (0, \frac{1}{2})$, the weight function $a(s,\cdot) \to 1$ and the exponent $q(s) \to p$ as $s \to 0^+$. We analyze the following cases for the problem \eqref{frac_pblm} as $s \to 0^+$:
\begin{enumerate}
    \item when \[p<q(s)<\frac{Np}{N-sp}  \ \text{and} \ q'(0) \in (0,\frac{p^2}{N}). \]
    In this case, \eqref{frac_pblm} belongs to the class of superlinear problems.
    \item when \[1<q(s)<p  \ \text{and} \ q'(0) \in (-\infty,0).\] In this case, \eqref{frac_pblm} belongs to the class of sublinear problems.
\end{enumerate}
In both cases, the nontrivial sequence of least energy solution $u_s$ of \eqref{frac_pblm} converges to a non-trivial least solution of the problem \eqref{lim_pblm} in $L^p(\Omega)$ (see below Theorems \ref{asym-sup} and \ref{asym-sub}). Variational tools (inspired by \cite{Angeles-Saldana}, \cite{Santamaria-Saldana} and \cite{Szulkin-Weth}), a new $p$-Logarithmic Sobolev inequality in Theorem \ref{log-Sob-ineq}, new integral estimates needed due to the nonlinear structure of the operator taking into account the non-availability of Fourier transform as in \cite{Angeles-Saldana, Santamaria-Saldana} and embedding results in Theorem \ref{thm:embd:results} are the main tools in the analysis of small order limits of the $p$-fractional problem. 

\textbf{Outline of the paper:} The paper is arranged in the following way: In Section \ref{Preliminaries}, we setup the energy framework and statement of the main results.  Section \ref{Embeddings into Orlicz spaces} is dedicated to proofs of $p$-logarithmic Sobolev inequality and optimal continuous and compact embedding results. In Section \ref{{p}-Logarithmic Laplacian equations with critical growth}, we derive new integral estimates and prove existence results to the problem \eqref{main_problem}. We also show the uniqueness result for the case $\mu \in (-\infty,0)$ to the problem \eqref{main_problem}. In Section \ref{{p}-fractional Laplacian problem}, we perform the asymptotics analysis of the $p$-fractional problem \eqref{frac_pblm}. In Appendix \ref{Appendix}, we recall and derive a few preliminary results.
\section{Preliminaries}
\label{Preliminaries}
In this section, we collect standard notations, introduce the functional setting and give statements of the main results proved in this work.
Throughout the text, let $N\geq 1$, $p \in (1, \infty)$, $s \in (0, \frac{1}{2})$ and $\Omega\subseteq \R^N$ be a bounded open set. $ \text{By} \ \|\cdot\|_{L^b(\Omega)}$, we denote the standard $L^b(\Omega)$ norm, defined as
\[
\|u\|_{L^b(\Omega)}:= \left(\int_{\Omega} |u|^b ~dx\right)^\frac{1}{b}\ \text{for} \ 1 \leq b < \infty \quad \text{and} \quad \|u\|_{L^\infty(\Omega)}:= \ess\sup_\Omega |u|.
 \]
 \subsection{Variational framework}
 \label{variational}
  We first set up the variational framework for the problem \eqref{main_problem}.
  Let  \[h(t):= |t|^{p-2} t \quad \text{for all} \ t \in \R,\]
and
$\Upsilon: \mathbb{R}^4 \to \mathbb{R}$ be given by
\begin{equation}\label{def:upsilon}
    \Upsilon(z_1, z_2, t_1, t_2) = h(z_1-z_2)(t_1-t_2)-h(z_1)t_1-h(z_2)t_2.
\end{equation}
The natural solution space for \eqref{main_problem} is \[\XO:= \{u \in L^p(\Omega) \ | \ [u]_{\XO}< +\infty \ \text{and} \ u=0 \ \text{in} \ \R^N\setminus\Omega\},\]
where \[[u]_{\XO}^p := C_{N,p}\inta \frac{|u(x)-u(y)|^{p}}{|x-y|^N} ~dx ~dy, \quad C_{N,p}:= \frac{p\Gamma\l(\frac{N}{2}\r)}{2\pi^{\frac{N}{2}}},\]
and $\Gamma$ is the Gamma function. The norm on $\XO$ is given by \[\nX= {\l(\nP^p+ [u]_{\XO}^p\r)}^{\frac{1}{p}}.\] With the above norm $\XO,$ forms a reflexive Banach space (see \cite[Section 3]{Foghem}). From \cite[Proposition 4.1]{Dyda-Jarohs-Sk}, it follows further that $[\cdot]_{\XO}$ is an equivalent norm for $\XO.$ For $u,v \in \XO$, the nonlinear form $\qform(u,v)$,  is defined as
\begin{equation}\label{quad:form}
    \begin{split}
&\qform(u,v) = \mathcal{E}_p(u,v) + \mathcal{F}_p(u,v) + \rho_N\int_{\R^N} h(u(x)) v(x) ~dx,  \end{split}
\end{equation}
where
\[
\mathcal{E}_p(u,v):= \frac{C_{N,p}}{2}\inta \frac{h(u(x)-u(y))(v(x)-v(y))}{|x-y|^{N}} ~dx ~dy,
\]
\begin{equation}\label{def:reminder-term}
    \begin{split}
     \mathcal{F}_p(u,v):= \frac{C_{N,p}}{2}\intb \frac{\Upsilon(u(x), u(y), v(x), v(y))}{|x-y|^{N}} ~dx ~dy,
     \end{split}
\end{equation}
\[\rho_N= \rho_N(p):= 2\ln 2 -\gamma+ \frac{p}{2}\psi(\frac{N}{2}),\quad \psi := \Gamma'/\Gamma\]
and $\gamma := -\Gamma'(1)$ is the Euler-Mascheroni constant.
The energy functional $J_{L_{\Delta_p}}: \XO \to \R$ corresponding to \eqref{main_problem} is given by
\begin{equation}
\label{energy-wei}
J_{L_{\Delta_p}}(u) = \frac{1}{p}\qform(u,u) -\into G(x,u)~dx -\frac{\mu}{p^2}\into |u|^p(\ln|u|^p -1) ~dx,
\end{equation}
where $G(x,t):=\int_0^t g(x,s) ~ds$ and $g:\Omega\times \R \to \R$ is a function satisfying the following assumptions:
\begin{enumerate}
[label=\textnormal{($g_1$)}, ref=\textnormal{$g_1$}]
\item  \label{assump g1}
$g \in L^\infty(\Omega \times K)$ for any $K \subset \subset \R,$
\end{enumerate}

\begin{enumerate}
[label=\textnormal{($g_2$)}, ref=\textnormal{$g_2$}]
\item  \label{assump g2}
$\lim\limits_{|t| \to \infty} \frac{g(x,t)}{h(t)\ln|t|}=0$ uniformly in $\Omega,$
\end{enumerate}

\begin{enumerate}
[label=\textnormal{($g_3$)}, ref=\textnormal{$g_3$}]
\item  \label{assump g3}
there exists $f \in L^\infty(\Omega)$ such that $\lim\limits_{t \to 0} \frac{g(x,t)}{h(t)} = f(x)$ uniformly in $\Omega,$
\end{enumerate}

 \begin{enumerate}
[label=\textnormal{($g_4$)}, ref=\textnormal{$g_4$}]
\item  \label{assump g4} $g(x,\cdot)\in C^1(\R)$ a.e. in $\Omega$ and there exists a constant $\delta < \mu$ such that
\[
t^2g'(x,t)+\delta |t|^p -(p-1)tg(x,t) \geq 0 \quad \text{for all} \ t \in \mathbb{R} \ \text{and a.e} \ x \in \Omega,
\]
\end{enumerate}

 \begin{enumerate}
[label=\textnormal{($g_5$)}, ref=\textnormal{$g_5$}]
\item  \label{assump g5}
$\lim\limits_{|t| \to \infty} \frac{tg'(x,t)}{h(t)\ln|t|}=0$ uniformly in $\Omega$.
\end{enumerate}

The assumptions \eqref{assump g1}-\eqref{assump g3} imply that for a given $\eps>0$ there exist constants $a_1>0$ and $a_2>0$, depending upon $\eps$, $p$, $\|f\|_{L^\infty(\Omega)}$ such that
\begin{equation}
\label{assumption g(x,t)}
 |g(x,t)| \leq a_1 |t|^{p-1}+\eps |t|^{p-1}|\ln|t|| \quad \text{for all} \ t \in \mathbb{R} \ \text{and a.e.} \ x \in \Omega
\end{equation}
and
\begin{align}
\label{assumption G(x,t)}
    \begin{split}
    |G(x,t)| \leq  a_2 |t|^p + \eps |t|^p |\ln|t||  \quad \text{for all} \ t \in \mathbb{R} \ \text{and a.e.} \ x \in \Omega.
    \end{split}
\end{align}
\begin{remark}
Note that assumption \eqref{assump g4}, implies
\begin{align}
    \label{assump integ g4}
    tg(x,t) + \frac{\delta |t|^p}{p} - pG(x,t) \geq 0 \quad \text{for all} \ t \in \mathbb{R} \ \text{and a.e.} \ x \in \Omega.
\end{align}
\end{remark}
\begin{remark}
Some examples of functions satisfying assumptions \eqref{assump g1}- \eqref{assump g5} are:
\begin{enumerate}
    \item  [\textnormal{(i)}] $g(x)= f(x)h(t)\ln^\beta(e+|t|),$ $\beta \in [0,1)$ and $f \in L^\infty(\Omega),$
    \item [\textnormal{(ii)}] $g(x)= f(x)h(t)\ln(e+\ln(1+|t|)),$ $f \in L^\infty(\Omega).$
\end{enumerate}
\end{remark}
Next, we introduce the notion of weak solution for the problem \eqref{main_problem}.
\begin{definition}
    A function $u \in \XO$ is said to be a weak solution to \eqref{main_problem} if
\[\qform(u,v)= \into g(x,u)v ~dx + \mu \into h(u) \ln|u| v ~dx \quad  \text{for all} \ v \in \XO.\]
Moreover, a weak solution $u$ to \eqref{main_problem} is said to be of ``least energy" if
\[J_{L_{\Delta_p}}(u)= \inf_{v\in \mathcal{N}_{L_{\Delta_p}}} J_{L_{\Delta_p}}(v),\]
where the Nehari manifold associated with \eqref{main_problem} is the natural constraint given by
\[\mathcal{N}_{L_{\Delta_p}} = \bigg\{u \in \XO \setminus\{0\} \ | \ \qform(u,u)= \into g(x,u)u ~dx + \mu \into |u|^{p}\ln|u| ~dx\bigg\}.\]
\end{definition}
Next, we give a variational framework to the problem \eqref{frac_pblm}.
   Define
    \[\WO=\{u \in L^p(\R^N)\,:\, [u]^p_{s,p,\R^N}<\infty \ \text{and} \ u=0 \ \text{a.e. in} \ \R^N\setminus\Omega  \},\]
    where
    \begin{equation}
\label{C-nsp}
[u]^p_{s,p,\R^N}:= \frac{C_{N,s,p}}{2}\intr \pfrac ~dx ~dy, \quad C_{N,s,p}= \frac{sp 2^{2s-1}\Gamma(\frac{N+sp}{2})}{\pi^\frac{N}{2}\Gamma(1-s)}.
\end{equation}
The norm on $\WO$ is given by
\[\nW=\l(\nP^p+[u]^p_{s,p,\R^N}\r)^{\frac{1}{p}}.\] Under this norm, $\WO$ forms a reflexive Banach space. It is easy to see that $[u]_{s,p,\R^N}$ is an equivalent norm on $\WO.$ The Sobolev conjugate $p^\ast_s$ of $p$ is defined as $p^\ast_s=\frac{Np}{N-sp}$.   We refer the reader to \cite{Nezza-Palatucci-Valdinoci} for an introduction and further details to fractional Sobolev spaces.

The energy functional $J_{s,p}: W^{s,p}_0(\Omega) \to \R$ corresponding to \eqref{frac_pblm} is defined as
    \begin{equation}
    \label{energy-frac}
          J_{s,p}(u)= \frac{1}{p}\nW^p-\frac{1}{q(s)}\into a(s,x)|u|^{q(s)} ~dx.
    \end{equation}
\begin{definition}
     A function $ u \in \WO $ is called a weak solution of \eqref{frac_pblm} if
     \[
     \begin{split}
     \frac{C_{N,s,p}}{2}\intr \frac{h(u(x)-u(y))(v(x)-v(y))}{|x-y|^{N+sp}} ~dx ~dy = \into a(s,x)|u|^{q(s)-2}uv ~dx
     \end{split}
     \]
for all $v \in W^{s,p}_0(\Omega)$. Moreover, a weak solution $u$ to the problem \eqref{frac_pblm} is said to have ``least energy" if
    \[J_{s,p}(u)= \inf_{v \in \mathcal{N}_{s,p}} J_{s,p}(v),\]
    where the Nehari manifold $\mathcal{N}_{s,p}$ associated with \eqref{frac_pblm} is given by
     \[\mathcal{N}_{s,p}:= \l\{u \in \WO\setminus\{0\} \, : \, [u]^p_{s,p,\R^N} = \into a(s,x)|u|^{q(s)} ~dx\r\}.\]
\end{definition}
\subsection{Main results}
\label{main results}
First, we present a $p$-logarithmic Sobolev type inequality, which is one of the key tools in the analysis of problems involving logarithmic nonlinearities. 

\begin{theorem}(p-Logarithmic Sobolev inequality)
\label{log-Sob-ineq}
    For $u \in \XO,$ it holds that
    \[\frac{p^2}{N} \into |u|^p \ln|u| ~dx \leq \qform(u,u)+ \frac{p^2}{N} \nP^p \ln \nP + k_0\nP^p,\]
    where $k_0$ is a constant depending upon $N$ and $p.$
\end{theorem}


The above inequality naturally leads us towards an examination of the continuous and compact embedding of the space $\XO$ into Orlicz type spaces. It is important to note that $t^p \ln |t|$ does not qualify as a $\Phi$-function within the Orlicz space theory, due to its sign changing nature. To tackle this issue, it is natural to consider $t^p \ln(e+t)$ as a candidate for $\Phi$-function, in view of the following relation
\begin{equation}\label{eq:relation:embd}
    t^p \ln(e+t) = t^p \ln(t) + t^p \ln\l(1+\frac{e}{t}\r) \quad \text{for} \ t>0.
\end{equation}
Before giving the embedding results, we recall the definitions of $\Phi$-function and Orlicz spaces.
\begin{definition}
	Let  $\ph \colon (0,+\infty) \to \R$ be a function and $p,q>0$. We say that
	\begin{enumerate}
		\item[\textnormal{(i)}]
			$\ph$ is almost increasing , if there exists $a \geq 1$ such that $\ph(s) \leq a \ph(t)$ for all $0 < s < t$.
		\item[\textnormal{(ii)}]
			$\ph$ is almost decreasing, if there exists $a \geq 1$ such that $a \ph(s) \geq \ph(t)$ for all $0 < s < t$.
	\end{enumerate}
We say that $\ph$ satisfies the property
	\begin{enumerate}
		\item[\textnormal{(aInc)}$_p$]
			if $\frac{\ph(t)}{t^p}$ is almost increasing,
		\item[\textnormal{(aDec)}$_q$]
			if $\frac{\ph(t)}{t^q}$ is almost decreasing.
\end{enumerate}
\end{definition}
\begin{definition}
	A function $\ph \colon [0,+\infty) \to [0,+\infty]$ is said to be a $\Phi$-function, if $\ph$ is increasing and satisfies
    \[
    \ph(0)=0, \quad \lim_{t\to 0^+} \ph(t) = 0 \quad \text{and} \quad \lim_{t \to +\infty} \ph(t) = +\infty.
    \]
Moreover, $\ph$ is said to be a convex $\Phi$-function (denoted by $\ph \in \Phi_c$) if $\ph$ is left-continuous, convex and it is said to be a weak $\Phi$-function (denoted by $\ph \in \Phi_w$) if $\ph$ satisfies \textnormal{(aInc)}$_1$ on $(0,\infty)$.
    \end{definition}

Throughout the text, we denote by $M(\Omega)$, the set of all real-valued Lebesgue measurable functions defined on $\Omega.$
Let $\vartheta \colon [0,+\infty) \to [0,+\infty]$ be a convex $\Phi$-function. Then, the modular function associated with $\vartheta$ is defined as
	\begin{align}
		\varrho_\vartheta (u) := \into \vartheta(\abs{u(x)}) ~dx \quad \text{for every} \ u \in M(\Omega).
	\end{align}
	The set
	\begin{align*}
		L^\vartheta (\Omega) := \{ u \in M(\Omega)\,:\, \varrho_\vartheta(\lambda u) < \infty \text{ for some } \lambda > 0 \}
	\end{align*}
	equipped with the Luxemburg norm
	\[
		\|u\|_\vartheta = \inf \bigg\{ \lambda > 0 \,:\, \varrho_\vartheta \l( \frac{u}{\lambda} \r)  \leq 1 \bigg\}
\]
is a Banach space (see \cite[Theorem 3.3.7]{Harjulehto-Hasto}). Next results correspond to the new continuous and compact embeddings of the energy space $\XO$ in Orlicz type spaces.

\begin{theorem}\label{thm:embd:results}
Let $\Omega\subseteq \R^N$ be a bounded domain and $\ph, \psi:[0,\infty) \to [0,\infty)$ be convex $\Phi$-functions such that $\ph(t):= t^p \ln(e+t)$ for all $t \geq 0$,
\begin{enumerate}[label=\textnormal{($\psi_1$)},ref=\textnormal{$\psi_1$}]
       \item \label{compact:cond} $\lim_{t\to 0^+} \frac{\psi(t)}{t^p} $ exists in $\mathbb{R}$ and $\lim_{t \to \infty} \frac{\psi(t)}{\ph(t)} = 0.$
   \end{enumerate}
Then,
    \begin{enumerate}
        \item[\textnormal{(i)}] the embedding $\XO \hookrightarrow L^\ph(\Omega)$ is continuous,
        \item[\textnormal{(ii)}] the embedding $\XO \hookrightarrow L^\psi(\Omega)$ is compact.
    \end{enumerate}
\end{theorem}
\begin{remark}
    The compactness result stated in the above theorem is a significant improvement compared to the results available in the literature (see \cite[Corollary 6.3]{Foghem} where compactness in $L^p(\Omega)$ is proved). We present some examples of the function $\psi$ which satisfy the condition \eqref{compact:cond}:
\begin{enumerate}
    \item[\textnormal{(i)}] $\psi(t) = t^p \ln^\beta(e+ |t|)$ for $\beta \in [0,1),$
    \item[\textnormal{(ii)}] $\psi(t) = t^p \ln(e+ \ln(1+|t|)).$
\end{enumerate}
\end{remark}
Next, we show the sharpness of the results obtained in Theorem \ref{thm:embd:results} by constructing a suitable class of functions in $\XO$ with nontrivial logarithmic scaling parameters.
\begin{theorem}\label{thm:sharp:embd:results}
    Let $\Omega\subseteq \R^N$ be a bounded, convex domain, and $\ph, \gamma \colon [0,\infty) \to [0,\infty)$ be convex $\Phi$-functions, $\ph(t):= t^p \ln(e+t)$ for all $t \geq 0$ and there exists a $\beta >0$ such that
    \begin{equation}\label{embd:cond:smaller}
        \lim_{\ell \to \infty} \frac{\gamma(\ell \beta)}{\ph(\ell \beta)} =\infty.
    \end{equation}
 Then,
        \begin{enumerate}
        \item[\textnormal{(i)}] the embedding $\XO \hookrightarrow L^\ph(\Omega)$ is not compact,
        \item[\textnormal{(ii)}] the embedding $\XO \hookrightarrow L^\gamma(\Omega)$ is not continuous.
    \end{enumerate}

\end{theorem}
\begin{remark}
    The above result indicates that if any Orlicz space $L^\gamma(\Omega)$ is smaller than $L^\ph(\Omega)$ in the sense of \eqref{embd:cond:smaller}, {\it i.e}, $L^\gamma(\Omega) \hookrightarrow L^\ph(\Omega)$, then the embedding $\XO \hookrightarrow L^\gamma(\Omega)$ is not continuous. Moreover, if \eqref{embd:cond:smaller} holds for some $\beta>0$, then it holds for all $\beta'>0.$
\end{remark}
Now, we state the existence result to the problem \eqref{main_problem}.
\begin{theorem}
    \label{Existence L1}
    Let $\Omega\subseteq \R^N$ be a bounded domain with Lipschitz boundary. Then, for $\mu \in (0, \frac{p^2}{N})$ and $g$ satisfying assumptions \eqref{assump g1}-\eqref{assump g5}, the problem \eqref{main_problem}
 has a nontrivial, least energy solution. Moreover, if it holds that $G(x,t)\leq G(x,|t|)$ for all $t \in \mathbb{R},$ a.e. $x \in \Omega,$ then the least energy solution of \eqref{main_problem} has a constant sign in $\Omega.$
\end{theorem}

\begin{theorem}
\label{Existence L2}
Let $\Omega\subseteq \R^N$ be a bounded domain with Lipschitz boundary. Then, for $\mu \in (-\infty, 0)$ and $g$ satisfying assumptions \eqref{assump g1}-\eqref{assump g3}, the problem \eqref{main_problem}
 has a nontrivial, least energy bounded solution $u.$ Moreover,
 \begin{itemize}
 \item[\textnormal{(i)}] if it holds that $G(x,t) \leq G(x,|t|)$ for all $t \in \R,$ a.e. $x \in \Omega$, then the least energy solution to \eqref{main_problem} has a constant sign in $\Omega,$
\item[\textnormal{(ii)}] if $g\in C^1(\R)$ such that \begin{equation}
        \label{assump g6}
        \l((p-1)tg(x,t) - g'(x,t)t^{2} - \mu t^p \r) \geq 0 \quad \text{for all} \ t \geq 0, \  \text{a.e.} \ x\in \Omega,
    \end{equation}
then any nontrivial, nonnegative least energy solution is unique.
\end{itemize}
\end{theorem}
We continue by examining the asymptotics of the fractional weighted problem. For this purpose, we assume the following conditions on the weight function $a$ and exponent function $q:$

The function $q : [0, \frac{1}{2}] \to \R$ is $C^1$ with $q'(0)=\frac{dq}{ds}\bigg{|}_{s=0}  =:\mu$ such that
\begin{enumerate}
[label=\textnormal{(q)}, ref=\textnormal{q}]
\item  \label{assump q}
$
\lim\limits_{s\to 0^+} q(s) =p.
$
\end{enumerate}
The weight function $a:[0,\frac{1}{2}]\times \Omega \to \R$ satisfies the following assumptions:
\begin{enumerate}
[label=\textnormal{(a{$_0$})}, ref=\textnormal{a{$_0$}}]
\item \label{assump1-a}
$a(\cdot, x) \in C^1([0,\frac{1}{2}]) $ and $\lim_{s\to 0^+} a(s,x) =1\  \text{ for a.e.}  \ x \in \Omega,$
\end{enumerate}
\begin{enumerate}
[label=\textnormal{(a{$_1$})}, ref=\textnormal{a{$_1$}}]
    \item \label{assump2-a}
$a(s,\cdot) \in L^{\nu}(\Omega)$, where for a $\delta\in (0,1)$
\begin{equation} \label{nu}\nu:=\nu(s) > \max\l\{\frac{p^\ast_s}{p^\ast_s-q(s)}, 1+\frac{(N-sp)p}{s(p^2-N\mu-\delta p^2)}\r\},\end{equation}
for all $s \in \l[0, \frac{1}{2}\r],$
\end{enumerate}
\begin{enumerate}
[label=\textnormal{(a{$_2$})}, ref=\textnormal{a{$_2$}}]
\item \label{assump3-a}
$ \frac{\partial a(s,\cdot)}{\partial s}=:a'(s,\cdot) \in L^\infty(\Omega) \quad \text{for all} \ s\in [0,\frac{1}{2}]$ and \[\sup_{s\in [0,\frac{1}{2}]} \|a'(s,\cdot)\|_{L^\infty(\Omega)}<+\infty,\]
\end{enumerate}
\begin{enumerate}
[label=\textnormal{(a{$_3$})}, ref=\textnormal{a{$_3$}}]
  \item  \label{assump4-a} there exist constants $c_1, \ c_2$ (independent of $s$) such that \[c_1 \leq \|a(s,\cdot)\|_{L^\nu(\Omega)}^{\frac{1}{p-q(s)}}\leq c_2 \quad \text{for all} \ s\in [0,\frac{1}{2}].\]
\end{enumerate}
With the above assumptions, we have obtained the following:
\begin{theorem}
\label{asym-sup}
Let $q$ satisfy \eqref{assump q} such that $q(s)\in(p,p^\ast_s)$, $\mu \in (0,\frac{p^2}{N})$ and $a$ satisfies \eqref{assump1-a}, \eqref{assump2-a}, \eqref{assump3-a}, \eqref{assump4-a}. Let $(u_s) \subseteq\WO$ be a sequence of nontrivial least energy solutions to the problem \eqref{frac_pblm}. Then, there exists a nontrivial least energy solution  $u_0 \in \XO$ to the problem \eqref{lim_pblm} such that, upto a subsequence, as $s\to 0^+$, $u_s \to u_0$ in $L^p(\Omega)$.
\end{theorem}

\begin{theorem}\label{asym-sub}
  Let $q$ satisfy \eqref{assump q} such that $q(s)\in(1,p)$, $\mu \in (-\infty,0)$ and $a$ satisfies \eqref{assump1-a}, \eqref{assump2-a}, \eqref{assump3-a}, \eqref{assump4-a}. Let $(u_s) \subseteq\WO$ be a sequence of nontrivial positive least energy solutions to the problem \eqref{frac_pblm}. Then, there exists a unique, nontrivial least energy solution  $u_1 \in \XO \cap L^\infty(\Omega)$, to the problem \eqref{lim_pblm} such that, upto a subsequence, as $s\to 0^+$, $u_s \to u_1$ in $L^p(\Omega)$.
\end{theorem}


\section{Embeddings into Orlicz spaces }
\label{Embeddings into Orlicz spaces}
This section is devoted to the proofs of optimal embedding results stated in Theorems \ref{thm:embd:results} and \ref{thm:sharp:embd:results}. Before proceeding, we recall the following:

\begin{lemma}\cite[Corollary 6.3]{Foghem}
    \label{compct-embeding}
Let $\Omega\subseteq \R^N$ be an open set, the embedding $\XO \hookrightarrow L^p_{\text{loc}}(\Omega)$ is continuous. Moreover, if $|\Omega|<\infty,$ then $\XO \hookrightarrow L^p(\Omega)$ is compact.
\end{lemma}
\begin{lemma}\cite[Corollary 4.2]{Frank-Seiringer} and \cite[Theorem 1]{Maz'ya-Shaposhnikova}
\label{thm1-MS}
Let $N \geq 1$, $p \geq 1$, $s \in (0,1)$ and  $N > sp.$ Then, for any $u \in W_0^{s,p}(\Omega)\setminus\{0\}$, the following holds
    \[\|u\|^p_{L^{p^\ast_s}(\Omega)} \leq B(N,s,p) \int_{\mathbb{R}^N} \int_{\mathbb{R}^N} \pfrac ~dx ~dy := B_{N,s,p} \lceil u \rceil_{s,p}^p, \]
where $B_{N,s,p}$ is the Sobolev constant given by
     \[
    B_{N,s,p} = \l(\frac{p}{p^\ast_s}\r) \l(\frac{N}{|\mathbb{S}^{N-1}|}\r)^\frac{sp}{N} \frac{1}{\mathcal{C}(N,s,p)},
    \]
    \[
    \mathcal{C}(N,s,p) = 2 \int_0^1 \rho^{sp-1} \l(1-\rho^{\frac{N-sp}{p}}\r)^p \Phi_{N,s,p}(\rho) ~d \rho
    \]
and
    \[
    \begin{split}
    \Phi_{N,s,p}(r)
    & = \begin{cases}
        |\mathbb{S}^{N-2}| \int_{-1}^1 (1-t^2)^\frac{N-3}{2} (1-2rt+r^2)^\frac{-(N+sp)}{2}~dt \ &\text{if} \ N \geq 2,\\
        \l(\frac{1}{(1-r)^{1+sp}} + \frac{1}{(1+r)^{1+sp}}\r) \qquad  &\text{if} \ N =1.
    \end{cases}
    \end{split}
    \]

\end{lemma}
\begin{lemma}
\label{diff-F_p}
    Let $\Omega\subseteq\R^N$ be an open, bounded set and $u, v \in L^p(\Omega)$ such that $u=0=v$ in $\mathbb{R}^N \setminus \Omega.$ Then, the map $(u,v) \mapsto \mathcal{F}_p(u,v)$ is well-defined and
    \[
    |\mathcal{F}_p(u,v)| \leq a \l(\|u\|_{L^{p-1}(\Omega)}^{p-1} \|v\|_{L^1(\Omega)} + |\Omega| \|v\|_{L^p(\Omega)} \|u\|_{L^p(\Omega)}^{p-1}\r) \leq b\|u\|_{L^p(\Omega)}^{p-1}\|v\|_{L^p(\Omega)},
    \]
    where $a, b$ are constants depending on $N, \ p.$
\end{lemma}
\begin{proof}
Let $u, v \in L^p(\Omega)$ such that $u=0=v$ in $\mathbb{R}^N \setminus \Omega.$ Then, by the definition of $\mathcal{F}_p(\cdot,\cdot)$ in \eqref{def:reminder-term} and using the symmetry of the integrand, we obtain
\[\begin{split}
    & \frac{1}{C_{N,p}} \mathcal{F}_p(u,v) = \frac{1}{2} \intb \frac{\Upsilon(u(x),u(y),v(x),v(y))}{|x-y|^{N}} ~dx ~dy\\
    & =  \intb \frac{h(u(x)-u(y))v(x)-h(u(x))v(x)}{|x-y|^{N}} ~dx ~dy\\
    & = \iint\limits_{\substack{x,y\in \Omega \\ \abs{x-y}\geq1}} \cdots \ + \iint\limits_{\substack{x \in \Omega, y \in \mathbb{R}^N \setminus \Omega \\ \abs{x-y}\geq1}} \cdots \ + \iint\limits_{\substack{y \in \Omega, x \in \mathbb{R}^N \setminus \Omega \\ \abs{x-y}\geq1}} \cdots \ + \iint\limits_{\substack{x, y \in \mathbb{R}^N \setminus \Omega \\ \abs{x-y}\geq1}}  \cdots = \sum_{i=1}^4 L_i.
\end{split}\]
Since $u=0=v$ in $\mathbb{R}^N \setminus \Omega$, $L_i =0$ for $i=2,3,4$. Therefore, using Lemma \ref{Lemma1-Lindgren}, we have
\begin{equation}\label{p-log-tail-est}
    \begin{split}
    &\mathcal{F}_p(u,v) = C_{N,p} L_1 \leq C_{N,p}  \iint\limits_{\substack{x,y\in \Omega \\ \abs{x-y}\geq1}} \frac{|h(u(x)-u(y))-h(u(x))||v(x)|}{|x-y|^{N}} ~dx ~dy\\
    & \leq C_{N,p}
    \begin{cases}
    (3^{p-1}+2^{p-1}) \iint\limits_{\substack{x,y\in \Omega \\ \abs{x-y}\geq1}} |u(y)|^{p-1} |v(x)| ~dx ~dy  & \text{if} \ p \in (1,2),\\
    (p-1) \iint\limits_{\substack{x,y\in \Omega \\ \abs{x-y}\geq1}} \l(|u(y)|^{p-1} + |u(x)|^{p-1}\r) |v(x)| ~dx ~dy & \text{if} \ p \geq 2,
    \end{cases} \\
    &\leq a \l(\|u\|_{L^{p-1}(\Omega)}^{p-1} \|v\|_{L^1(\Omega)} + |\Omega| \|v\|_{L^p(\Omega)} \|u\|_{L^p(\Omega)}^{p-1}\r) \leq b\|u\|_{L^p(\Omega)}^{p-1}\|v\|_{L^p(\Omega)},
\end{split}
\end{equation}
where $a,b$ are constants depending on $N$ and $p.$
\end{proof}
\subsection{\texorpdfstring{$p$}{p}-Logarithmic Sobolev inequality}
In this subsection, first we prove the asymptotics of the constant involved in the fractional Sobolev inequality in Lemma \ref{thm1-MS} and then derive a $p$-logarithmic Sobolev type inequality which will be very crucial in the study of nonlinear problems, involving zero-order operator, in particular, the logarithmic $p$-Laplacian.
\begin{lemma}
\label{lem-asymp}
    It holds that \[\lim\limits_{s \to 0} \l(\frac{C_{N,s,p}}{s}\r)= \frac{p}{|\mathbb{S}^{N-1}|} \quad \text{and} \quad \lim_{s \to 0^+}\l( \frac{s}{B_{N,s,p}}\r)=\frac{2|\mathbb{S}^{N-1}|}{p},\]
where $|\mathbb{S}^{N-1}|$ denotes the surface area of sphere in $\R^N.$
\end{lemma}
\begin{proof}
By \eqref{C-nsp}, we have
   \[
   \begin{split}
       \lim\limits_{s \to 0} \l(\frac{C_{N,s,p}}{s}\r) = \lim_{s \to 0} \l(\frac{p 2^{2s-1}\Gamma(\frac{N+sp}{2})}{\pi^\frac{N}{2}\Gamma(1-s)}\r)= \frac{p \Gamma(\frac{N}{2})}{2\pi^\frac{N}{2}}=\frac{p}{|\mathbb{S}^{N-1}|}.
   \end{split}
   \]
Note that
\[
\lim\limits_{s \to 0^+} \frac{s}{B_{N,s,p}} = \lim\limits_{s \to 0^+} s \ \mathcal{C}_{N,s,p} = \lim\limits_{s \to 0^+} 2 \int_0^1 \rho^{sp-1} h_{N,s,p}(\rho) ~d \rho,
\]
where
\[
h_{N,s,p}(r):= \l(1-r^{\frac{N-sp}{p}}\r)^p \Phi_{N,s,p}(r).
\]
Now, in order to compute the above limit, first we show that
\begin{equation}\label{limit-est-1}
    \lim\limits_{(r,s) \to (0,0)} h_{N,s,p}(r) = |\mathbb{S}^{N-1}|.
\end{equation}
For $N=1$, the claim is obvious. Fix $p >1$ and $N \geq 2$, then it is easy that for $0<s < \frac{N}{2}$, we have $\frac{N-sp}{p}> \frac{N}{2p}>0$ and
\[
0 < r^{\frac{N-sp}{p}}< r^{\frac{N}{2p}} \to 0 \quad \text{as} \ r \to 0^+.
\]
This further gives
\begin{equation}\label{limit-est-2}
    \lim\limits_{(r,s) \to (0,0)} \l(1-r^{\frac{N-sp}{p}}\r)^p = 1.
\end{equation}
Again, for $0< s < \min\l\{\frac{N}{2},1\r\}$, $r \in (0, \frac{1}{4})$ and $t \in [-1,1]$, we get
\[
1-2tr +r^2 \geq 1- 2 |t| r + r^2 > 1-2r > \frac{1}{2},
\]
\[
|g_{r,s}(t)|:= \l|\frac{(1-t^2)^\frac{N-3}{2}}{(1-2rt+r^2)^\frac{(N+sp)}{2}}\r| \leq C(N,p) (1-t^2)^\frac{N-3}{2} \in L^1[-1,1] \ \text{for} \ N \geq 2.
\]
and $g_{r,s}(t) \to (1-t^2)^\frac{N-3}{2}$ pointwise a.e. in $[-1,1]$ as $(r,s) \to (0,0)$. Therefore, by Lebesgue dominated convergence theorem, we obtain
\begin{equation}\label{limit-est-3}
    \lim\limits_{(r,s) \to (0,0)} \Phi(N,s,p)(r) = |\mathbb{S}^{N-2}| \int_{-1}^1 (1-t^2)^\frac{N-3}{2} ~dt = |\mathbb{S}^{N-1}| \quad \text{for} \ N \geq 2.
\end{equation}
Now, combining \eqref{limit-est-2} and \eqref{limit-est-3}, we obtain \eqref{limit-est-1}. Finally, using \eqref{limit-est-1} and Lemma \ref{prelim-lemma}, we obtain
\[
\lim\limits_{s \to 0^+} s \ \mathcal{C}_{N,s,p} = \frac{2 }{p} |\mathbb{S}^{N-1}|.
\]
\end{proof}

\noindent \textit{Proof of Theorem \ref{log-Sob-ineq}:}
Let $u\in C_c^\infty(\Omega)\setminus\{0\}.$ Multiplying and dividing by $\frac{C_{N,s,p}}{2}$ on the right hand side of the inequality in Lemma \ref{thm1-MS} and putting the value of $C_{N,s,p}$ from \eqref{C-nsp}, we get
    \begin{equation}
    \label{sob_ineq}
    \|u\|^p_{L^{p^\ast_s}(\Omega)} \leq A(N,p,s) [u]^p_{s,p, \mathbb{R}^N}, \quad A(N,p,s):= \frac{2B_{N,s,p}}{C_{N,s,p}}.
    \end{equation}
By \cite[Lemma 7.2]{Dyda-Jarohs-Sk}, we have
    \[\frac{d}{ds}\bigg{|}_{s=0} [u]^p_{s,p, \mathbb{R}^N} = \lim_{s \to 0} \frac{[u]^p_{s,p, \mathbb{R}^N} - \|u\|_{L^p(\Omega)}^p}{s} = \qform(u,u).\]
    Thus,
    \[
        [u]^p_{s,p, \mathbb{R}^N} = \nP^p+s\qform(u,u) + o(s).
    \]
By Lemma \ref{lem-asymp} and Taylor's formula, we get
\begin{equation}\label{constant:est-2}
\lim\limits_{s \to 0} A(N,p,s) =1
    \end{equation}
and $A(N,p,s) = 1 + s k_0 + o(s)$, where $k_0 := \partial_s A(N,p,0).$ This further implies
    \begin{equation}
    \label{norm:est-1}
    A(N,p,s) [u]^p_{s,p, \mathbb{R}^N} = \nP^p + s\bigg(k_0\nP^p + \qform(u,u)\bigg) + o(s).
    \end{equation}

Denote
\[
F(s):=\|u\|^p_{L^{p^\ast_s}(\Omega)}=\left(\into |u|^{p^\ast_s} ~dx\right)^{\frac{p}{p^\ast_s}}.
\]
Clearly, $F(0)= \nP^p.$ Now, by taking the logarithm on both sides and differentiating with respect to $s$, we obtain
\[
\begin{split}
    \frac{F'(s)}{F(s)} &= \frac{-p}{N}\ln\bigg[\into |u|^\frac{Np}{N-sp} ~dx\bigg] + \l(\frac{N-sp}{N}\r)\frac{d}{ds}\l(\ln \bigg[\into |u|^\frac{Np}{N-sp} ~dx\bigg]\r)\\
    &= \frac{-p}{N}\ln\bigg[\into |u|^\frac{Np}{N-sp} ~dx\bigg] + \l(\frac{N-sp}{N}\r)\frac{1}{\into |u|^\frac{Np}{N-sp} ~dx}\frac{d}{ds}\bigg[\into |u|^\frac{Np}{N-sp} ~dx\bigg]\\
    &= \frac{-p}{N}\ln\bigg[\into |u|^\frac{Np}{N-sp} ~dx\bigg]\\&\qquad \qquad+ \l(\frac{N-sp}{N}\r)\frac{1}{\into |u|^\frac{Np}{N-sp} ~dx}\l(\into \frac{Np^2}{(N-sp)^2}|u|^{\frac{Np}{N-sp} } \ln |u| ~dx\r).
\end{split}
\]
This further gives
\begin{equation}\label{norm:est-2}
    \begin{split}
\|u\|^p_{L^{p^\ast_s}(\Omega)} & = F(s) = F(0) + s F'(0) + o(s)\\
& = \nP^p +
\frac{s p^2}{N}\l(\into |u|^p \ln|u| ~dx - \nP^p \ln (\nP)\r) + o(s).
\end{split}
\end{equation}
Using \eqref{norm:est-1} and \eqref{norm:est-2} in \eqref{sob_ineq}, we get
    \[
    \begin{split}
        \bigg[\nP^p&+ \frac{s p^2}{N}\l(\into |u|^p \ln|u| ~dx - \nP^p \ln (\nP)\r)+o(s)\bigg]\\& \qquad \qquad \qquad \leq \nP^p + s\l(k_0\nP^p +  \qform(u,u)\r) + o(s).
    \end{split}
    \]
Finally, dividing by $s$ and passing limit $s \to 0^+$, we obtain
\begin{equation}
\label{*}
\frac{p^2}{N}\l(\into |u|^p \ln|u| ~dx - \nP^p \ln \nP\r)\leq \qform(u,u)+ k_0\nP^p.
\end{equation}
By \cite[Proposition 4.3]{Dyda-Jarohs-Sk}, $C_c^\infty(\Omega)$ is dense in $\XO$. Let $(u_n)_{n \in \N} \subseteq C^\infty_c(\Omega)$ be such that $u_n \to u$ in $\XO.$ From the above, we get \eqref{*} holds for each $u_n$. Since $u_n \to u $ in $\XO$, by Lemma \ref{compct-embeding}, $u_n \to u$ in $L^p(\Omega)$. Using this, the definition of $\qform(\cdot,\cdot)$ in \eqref{quad:form} and Lemma \ref{diff-F_p}, we get
\[\lim\limits_{n \to \infty} \qform(u_n,u_n)= \qform(u,u).\] Now, using the fact that $t^p \ln|t|$ is bounded for $t \in (0,1)$ and Fatou's lemma, we get
\[
\into |u|^p \ln|u| ~dx \leq \lim\limits_{n \to \infty}\into |u_n|^p \ln|u_n| ~dx.
\] Combining all the above estimates, we get that \eqref{*} holds for every $u\in \XO.$
\qed
\subsection{Continuous and compact embeddings}
We recall some important relations between the norm and modular.
\begin{lemma}
	\label{norm_modular}
    \cite[Lemma 3.2.9]{Harjulehto-Hasto}\newline
	Let $\vartheta \colon [0,+\infty) \to [0,+\infty]$ be a  weak $\Phi$-function that satisfies \textnormal{(aInc)}$_p$ and \textnormal{(aDec)}$_q$, $1 \leq p \leq q < \infty$. Then,
	\begin{align*}
 \min \bigg\{ {\l(\frac{1}{a}\varrho_\vartheta (u)\r)}^{\frac{1}{p}} , {\l(\frac{1}{a}\varrho_\vartheta (u)\r)}^{\frac{1}{q}} \bigg\}
		\leq \|u\|_\vartheta
		\leq \max \bigg\{ {\l(a\varrho_\vartheta (u)\r)}^{\frac{1}{p}} , {\l(a\varrho_\vartheta(u)\r)}^{\frac{1}{q}} \bigg\}
	\end{align*}
	for all $u \in M(\Omega)$, where $a$ is the maximum of the constants of \textnormal{(aInc)}$_p$ and \textnormal{(aDec)}$_q$.
\end{lemma}
\begin{lemma}
	\label{epsilon}
    \cite[Lemma 3.1]{Arora-Crespo-Blanco-Winkert} \newline
	The function $f_\eps \colon [0,+\infty) \to  [0,+\infty)$ given by
	\begin{align*}
		f_\eps(t) = \frac{ t^\eps } { \ln (e + t) }
	\end{align*}
	is almost increasing for $0 < \eps < \kappa :=e/(e + t_0)$ with constant $a_\eps:=\frac{f_{\eps}(t_{1, \eps})}{f_{\eps}(t_{2, \eps})} >1,$ where $t_0$ is the only positive solution to $t_0 = e \ln(e + t_0)$ and $t_{1,\eps}$, $t_{2, \eps}$ are respectively, the points of local maximum and minimum of $f_\eps$ in $(0, +\infty)$.
\end{lemma}
\begin{lemma}
	\label{ln(e+t)}
    \cite[Lemma 3.3]{Arora-Crespo-Blanco-Winkert}
	The function given by
$\ph(t)=t^p \ln (e+t)\quad \text{for} \ t>0$,  fulfills \textnormal{(aDec)}$_{p + \eps}$ for $0 < \eps < \kappa$ with constant $a_\eps$, where $\kappa$ and $a_\eps$ are the same as in Lemma \ref{epsilon}.
\end{lemma}
Moreover, we recall that the following holds
\begin{equation}
\label{logineq}
\frac{t}{t+1}\leq \ln(1+t) \leq t \quad \text{for all} \ t>-1.
\end{equation}

\noindent \textit{Proof of Theorem \ref{thm:embd:results}} (i):
We show that the inclusion map $i: \XO\to L^{\ph}(\Omega)$ defined by $i(u) = u$ maps bounded sets in $\XO$ to bounded sets in $L^{\ph}(\Omega)$. Suppose $A$ be a bounded subset of $\XO$, {\it i.e.}, there exists a positive constant $C_1$, such that
\begin{equation}
\label{bdd-XO}
    \nX^p \leq C_1 \quad \text{for all} \ u \in A.
\end{equation}
In view of Lemma \ref{norm_modular}, it is enough to show that there exists a $C >0 $ such that
\begin{equation}
\label{bdd-ph-1}
    \int_{\Omega} \ph(|u|) ~dx = \into |u|^p \ln (e+|u|) ~dx \leq C \quad  \text{for all} \ u \in A.
\end{equation}
Let $u \in A.$ Using \eqref{logineq}, Theorem \ref{log-Sob-ineq}, the definition of $\qform$ and Lemma \ref{diff-F_p}, we obtain
\[
    \begin{split}
   \into |u|^p & \ln (e+|u|) ~dx \\
   &= \into |u|^p \ln (|u|) ~dx + \into |u|^p \ln \l(1+\frac{e}{|u|}\r) ~dx \\
   &\leq \frac{N}{p^2}\l(\qform(u,u)+k_0\nP^p\r)+\nP^p \ln \nP + e \into |u|^{p-1} ~dx\\
   & \leq \frac{N}{p^2}\l(\nX^p+\l(\rho_N+b\r)\nP^p+k_0\nP^p\r)\\
   & \qquad \qquad  +\nP^p \ln \nP + e \ |\Omega|^{\frac{1}{p}} \  \nP^{p-1}  <\infty.
   \end{split}
   \]
Finally, in the last estimate, using \eqref{bdd-XO} and the embedding $\XO \hookrightarrow L^p(\Omega)$ (see Lemma \ref{compct-embeding}), we obtain \eqref{bdd-ph-1}.
\qed

Next, we prove the compact embedding of the energy space $\XO$ into Orlicz-type spaces.
\vspace{0.1cm}\\
\noindent \textit{Proof of Theorem \ref{thm:embd:results}} (ii):
Consider the inclusion map
\[
i: \XO \to L^{\psi}(\Omega) \quad \text{defined by} \quad i(u)=u.
\]
Note that by Theorem \ref{thm:embd:results} (i) and \cite[Proposition 3.2.6]{Harjulehto-Hasto}, we have $\XO \hookrightarrow  L^{\ph}(\Omega) \hookrightarrow L^{\psi}(\Omega).$ Therefore, the inclusion map is well-defined and continuous. Let $\{f_n\}_{n \in \mathbb{N}}$ be a bounded sequence in $\XO.$ The reflexivity of $\XO$ and Lemma \ref{compct-embeding}, imply that there exists a subsequence $\{f_n\}_{n \in \N}$ (denoted by the same notation) such that $f_n \rightharpoonup f$ in $\XO$ and $f_n \to f$ in $L^p(\Omega).$ In view of Lemma \ref{norm_modular}, it is enough to show that for every $\eps >0,$ there exists a $n_0 \in \mathbb{N}$ such that
\[
\into \psi(|f_n-f|) ~dx \leq \eps \quad \text{for all} \ n \geq n_0.\]
Note that by \eqref{compact:cond}, for every $\eps_1>0$, there exist $\delta_1, \delta_2>0$ such that \begin{equation}\label{limits:est}
    \psi(t)\leq C(\eps_1) t^p  \quad \text{for} \ 0 < t < \delta_1 \quad \text{and} \quad \psi(t)\leq \eps_1 t^p\ln(e+t) \quad \text{for} \ t>\delta_2.
\end{equation}
Now, using the fact that $\psi$ is an increasing function and \eqref{limits:est}, we obtain
\[
\begin{split}
     \into \psi(|f_n-f|) ~dx & \leq C(\eps_1) \int_{\{|f_n-f|<\delta_1\}} |f_n-f|^p ~dx\\
   & \qquad + \frac{\psi(\delta_2)}{\delta_1^p}\int_{\{\delta_1\leq |f_n-f| \leq \delta_2\}} |f_n-f|^p  ~dx\\ &\qquad + \eps_1 \int_{\{|f_n-f|>\delta_2\}} |f_n-f|^p \ln (e+|f_n-f|) ~dx \\
   &\leq \l(C(\eps_1)  + \frac{\psi(\delta_2)}{\delta_1^p} \r)\|f_n-f\|_{L^p(\Omega)} \\ &\qquad  + \eps_1 \into |f_n-f|^p \ln (e+|f_n-f|) ~dx.
\end{split}
\]
Since $f_n \to f$ in $L^p(\Omega)$, for every $\eps >0$, there exists a $n_0 \in \mathbb{N}$ such that
\begin{equation}\label{limits:est-1}
\|f_n-f\|_{L^p(\Omega)} < \frac{\eps}{2 \l(C(\eps_1)  + \frac{\psi(\delta_2)}{\delta_1^p} \r)} \quad \text{for all} \ n \geq n_0.
\end{equation}
Moreover, using Theorem \ref{thm:embd:results} (i) and the fact that $\{f_n\}_{n \in \mathbb{N}}$ is a bounded sequence in $\XO$, we get
\begin{equation}\label{limits:est-2}
\into |f_n-f|^p \ln (e+|f_n-f|) ~dx \leq C \quad \text{for some} \ C>0.
\end{equation}
Finally, using \eqref{limits:est-1}, \eqref{limits:est-2} and taking $\eps_1= \frac{\eps}{2C}$ in \eqref{limits:est}, we obtain that $\XO$ is compactly embedded in $L^{\psi}(\Omega).$
\qed
\subsection{Sharpness of embeddings}
In the next result, we show that the embedding in Theorem \ref{thm:embd:results} (i) is not compact by constructing a suitable counterexample. In this subsection, without loss of generality, we can assume that $\Omega$ contains $0.$ Let $u\in \XO$ be such that $\into |u|^p \ln (e+|u|) ~dx =1.$ For $k \in \mathbb{N} \setminus \{1\}$,  define the sequence $\{u_k\}_{k \in \mathbb{N} \setminus \{1\}}$ as
\begin{equation}\label{def:seq-of-functions}
    u_k(x)= \begin{cases}
      {\l(\frac{k^N}{\ln k}\r)}^\frac{1}{p}u(kx) & x\in \frac{\Omega}{k}:=\{\frac{x}{k} \,:\, x \in \Omega\}, \\
      0 & \text{otherwise}.
   \end{cases}
\end{equation}

\begin{lemma}
    \label{Step1}
 The sequence $\{u_k\}_{k \in \mathbb{N} \setminus \{1\}}$ is bounded in $\XO.$
\end{lemma}
\begin{proof}
    By change of variables $s=kx,\ t=ky$ and the symmetry of the integrand, we obtain
\[
\begin{split}
    & \nX^p =C_{N,p} \inta \frac{|u_k(x)-u_k(y)|^p}{|x-y|^N} ~dx ~dy\\
    &= \frac{C_{N,p}}{\ln k} \bigg(\iint\limits_{\substack{s,t\in\R^N \\ \abs{s-t}\leq 1}} \frac{|u(s)-u(t)|^p}{|s-t|^N} ~ds ~dt + \iint\limits_{\substack{s,t\in\R^N \\ 1<\abs{s-t}\leq k}}\frac{|u(s)-u(t)|^p}{|s-t|^N} ~ds ~dt\bigg)\\
 & = \frac{1}{\ln k} \nX^p +  \frac{C_{N,p}}{\ln k} \iint\limits_{\substack{s,t\in\R^N \\ 1<\abs{s-t}\leq k}}\frac{|u(s)-u(t)|^p}{|s-t|^N} ~ds ~dt =:I_1+I_2.
\end{split}
\]
A straightforward calculation leads to
\[
\begin{split}
    I_1=\frac{1}{\ln k} \nX^p \leq \frac{1}{\ln 2} \nX^p
\end{split}
\]
and
\[
\begin{split}
    I_2 & = \frac{C_{N,p}}{\ln k} \iint\limits_{\substack{s,t\in\R^N \\ 1<\abs{s-t}\leq k}}\frac{|u(s)-u(t)|^p}{|s-t|^N} ~ds ~dt  \leq  \frac{2^{p-1}C_{N,p}}{\ln k}\iint\limits_{\substack{s,t\in\R^N \\ 1<\abs{s-t}\leq k}} \frac{|u(s)|^p+|u(t)|^p}{|s-t|^N} ~ds ~dt \\
 & \leq \frac{2^p C_{N,p}}{\ln k} \int_{s\in \R^N}|u(s)|^p\l(\int_{t \in  \overline{B_k(s)} \setminus \overline{B_1(s)} } \frac{dt}{|s-t|^N}\r) ~ds\\
 & \leq \frac{2^p C_{N,p}}{\ln k} \int_{s\in \R^N}|u(s)|^p \l(\int_{z \in  \overline{B_k(0)} \setminus \overline{B_1(0)}} \frac{dz}{|z|^N}\r) ~ds\\
 & \leq \frac{2^p C_{N,p} |\mathbb{S}^{N-1}|}{\ln k} \nP^p \l(\int_1^k \frac{dr}{r}\r) \leq  2^p C_{N,p} |\mathbb{S}^{N-1}| \nP^p,
\end{split}
\]
where $|\mathbb{S}^{N-1}|$ denotes the surface area of sphere in $\mathbb{R}^N.$ Collecting the above estimates for $I_1$ and $I_2$, we get that $u_k \in \XO$, for all $k\in \N \setminus\{1\}$ and the sequence $\{u_k\}_{k \in \mathbb{N} \setminus \{1\}}$ is bounded in $\XO.$
\end{proof}
\begin{lemma}
    \label{Step2}
There exist $C_0>0$ and $\tilde{k} \in \mathbb{N} \setminus \{1\}$ such that $\|u_k\|_{\varphi} \geq C_0$ for all $k \geq \tilde{k}.$
\end{lemma}
\begin{proof}
For $\beta \in (0,1),$ $\Omega_\beta:= \{x\in \Omega \ | \  |u(x)| \geq \beta \}$. For every $u \in \XO \setminus \{0\}$, we can choose a $\beta = \beta(u) >0$ small enough such that $\Omega_\beta$ is nonempty and $|\Omega|_\beta \neq 0.$ Then, we have
\[
\begin{split}
   \into |u_k|^p \ln (e+|u_k|) ~dx &= \frac{k^N}{\ln k}\into |u(kx)|^p \ln \l(e+{\l(\frac{k^N}{\ln k}\r)}^\frac{1}{p}|u(kx)|\r)  ~dx \\
    & \geq \frac{1}{\ln k} \int_{\Omega_\beta} |u(y)|^p  \ln \l(e+{\l(\frac{k^N}{\ln k}\r)}^\frac{1}{p}|u(y)|\r) ~dy\\
    & \geq \frac{1}{\ln k} \int_{\Omega_\beta} \beta^p \ln \l(e+{\l(\frac{k^N}{\ln k}\r)}^\frac{1}{p}\beta\r) ~dy\\
    & \geq \frac{\beta^p}{p\ln k} \ln {\l(\frac{k^N \beta ^p}{\ln k}\r)} \ |\Omega_\beta| = \frac{\beta^p |\Omega_\beta|}{p} \bigg(N+\frac{\ln \beta^p}{\ln k}-\frac{\ln(\ln k)}{\ln k}\bigg).
\end{split}
\]
Since $\frac{\ln(\ln k)}{\ln k} \to 0$ as $k \to \infty,$ there exist constants $C_0 >0$ and $\tilde{k} \in \mathbb{N} \setminus \{1\}$ depending upon $\beta, N, p$ such that
\[
N+\frac{\ln \beta^p}{\ln k}-\frac{\ln(\ln k)}{\ln k} \geq C_0 >0 \quad \text{for all} \ k \geq \tilde{k}.
\]
The above implies
\begin{equation}\label{lower:est:norm}
   \into |u_k|^p \ln (e+|u_k|) ~dx \geq C_0 \frac{\beta^p |\Omega_\beta|}{p} \quad \text{for all} \ k \geq k_0,
\end{equation}
and in view of Lemma \ref{norm_modular}, we get the claimed assertion.
\end{proof}
\noindent \textit{Proof of Theorem \ref{thm:sharp:embd:results}} (i):
Let $u\in \XO$ be such that $\into |u|^p \ln (e+|u|) ~dx =1.$ Let $k \in \mathbb{N} \setminus \{1\}$ and $u_k$ be as defined in \eqref{def:seq-of-functions}. It is easy to see that $u_k \to 0$ a.e. in $\Omega.$ Now, we proceed by contradiction. Suppose that the embedding  $\XO \hookrightarrow L^{\ph}(\Omega)$ is compact. By Lemmas \ref{Step1} and \ref{Step2}, the reflexivity of $\XO$, the compact embedding $\XO\hookrightarrow L^p(\Omega)$ in Lemma \ref{compct-embeding}, there exists a $\tilde{u} \in \XO \setminus \{0\}$ such that (up to a subsequence) $u_k \rightharpoonup \tilde{u}$ in $\XO$ , $u_k \to \tilde{u}$ in $L^{\ph}(\Omega) \cap L^p(\Omega)$ and $u_k \to \tilde{u} \neq 0$ a.e. in $\Omega.$ But this contradicts the fact that $u_k \to 0.$ Hence, $\XO \hookrightarrow L^{\ph}(\Omega)$ is not compact.
\qed

Next, we prove that the embedding in Theorem \ref{thm:embd:results} (i) cannot be improved further in the Orlicz space setting. For this, let $\gamma:[0,\infty) \to [0, \infty)$ be a convex $\Phi$-function satisfying \eqref{embd:cond:smaller}. In view of \cite[Proposition 3.2.6]{Harjulehto-Hasto}, it is easy to see that $L^{\gamma}(\Omega) \hookrightarrow L^{\ph}(\Omega).$ \vspace{0.1cm}\\

\noindent \textit{Proof of Theorem \ref{thm:sharp:embd:results}} (ii):
Let $u \in \XO$ such that $\Omega_\beta:=\l\{ x \in \Omega: |u(x)| \geq \beta \r\} \neq \emptyset$ and $|\Omega|_\beta \neq 0.$ Define
\[
A_u:= \{u_k : k \in \mathbb{N} \setminus \{1\}\},
\]
where $u_k'$s are defined in \eqref{def:seq-of-functions}. By Lemma \ref{Step1}, the set $A_u$ is bounded in $\XO.$ To prove the claim, we show that the set $A_u$ is not bounded in $L^\gamma(\Omega).$
On the contrary, we assume that there exists a constant $C>0$ (independent of $k$) such that
   \begin{equation}
   \label{bdd_gamma}
   \into \gamma (|u_k|) ~dx \leq C \quad \text{for all} \  u_k \in A_u.
   \end{equation}
The property \eqref{embd:cond:smaller} implies that for every $M>0$, there exists a $\ell_0 \gg e$ depending upon $M$ and $\beta$ such that
\begin{equation}\label{lower:est:gamma1}
    \gamma(\ell \beta)> M \ph(\ell \beta) \quad \text{for all} \ |\ell|>\ell_0.
\end{equation}
Choose $k_0 \in \mathbb{N}$ large enough such that
\[
\l(k^N/\ln k\r)^\frac{1}{p} \frac{|u(x)|}{\beta} \geq \l(k^N/\ln k\r)^\frac{1}{p} > \ell_0 \quad \text{for all} \ k \geq k_0 \ \text{and} \ x \in \Omega_\beta.
\]
Applying change of variables and using \eqref{lower:est:gamma1} for $k \geq k_0$, we obtain
\begin{equation}\label{lower:est:gamma2}
    \begin{split}
   \into \gamma(|u_k|)~dx & = \frac{1}{k^N} \int_{\Omega} \gamma\l({\l(\frac{k^N}{\ln k}\r)}^\frac{1}{p} |u(x)|\r)~dx \\
   & \geq \frac{1}{k^N} \int_{\Omega_\beta} \gamma\l({\l(\frac{k^N}{\ln k}\r)}^\frac{1}{p} |u(x)|\r)~dx \\
   & > \frac{M}{\ln k} \int_{\Omega_\beta} |u|^p \ln \l(e+{\l(\frac{k^N}{\ln k}\r)}^\frac{1}{p} \beta \frac{|u(x)|}{\beta}\r) ~dx\\
   &  > \frac{M \ln \l[\beta{\l(\frac{k^N}{\ln k}\r)}^{\frac{1}{p}}\r]}{\ln k}\int_{\Omega_\beta} |u(x)|^p  ~dx \\
       &> M \frac{|\Omega_\beta| \beta^p}{p} \l(N - \frac{\ln(\ln k)}{\ln k}+\frac{p\ln \beta}{\ln k}\r).
   \end{split}
\end{equation}
Note that $\frac{\ln(\ln k)}{\ln k} - \frac{ 2 \ln \beta}{\ln k} \to 0$ as $k \to \infty,$ {\it i.e,} there exists a $k_1 \in \mathbb{N}$ large enough such that
\[
\frac{\ln(\ln k)}{\ln k} - \frac{ p \ln \beta}{\ln k} \leq \frac{N}{2} \quad \text{for all} \ k \geq k_1.
\]
Finally, choosing $k \geq \max\{k_0, k_1\}$ and $M = \frac{4C}{\beta^p |\Omega_\beta|N}$ in \eqref{lower:est:gamma1} and \eqref{lower:est:gamma2}, we get
\[\into \gamma(|u_k|)~dx > M \frac{|\Omega_\beta| \beta^p N}{2p} > 2C,\]
which is a contradiction. Hence, $A_u$ is not bounded in $L^\gamma(\Omega)$ and $\XO$ is not continuously embedded in $L^\gamma(\Omega).$
\qed
\section{Logarithmic \texorpdfstring{$p$}{p}-Laplacian equations with critical growth}
\label{{p}-Logarithmic Laplacian equations with critical growth}
Before beginning this section, we prove an elementary result.
\begin{lemma}
    \label{log t>1}
    For any $\tau>0,$ it holds that $t^\tau - e \tau \ln t \geq 0$ for $t\geq 1.$
\end{lemma}
\begin{proof}
Define $f(t)= \ln t -\frac{t^\tau}{e\tau}.$ Then,
$f'(t)=0$ at $e^{\frac{1}{\tau}}$. It is easy to see that $f'(t)>0$ for $ t \in (0,e^{\frac{1}{\tau}})$ and $f'(t)<0$ for $ t >e^{\frac{1}{\tau}}$, {\it i.e.}, $e^{\frac{1}{\tau}}$ is the unique maximum of $f$. This completes the proof.
\end{proof}
\begin{lemma}
    \label{ln-limit}
    Let $(u_n)_{n \in \N}\subseteq L^\ph(\Omega)$ be such that $u_n \to u$ in $L^\ph(\Omega).$ Then, for every $v \in L^\ph(\Omega),$ it holds that
    \begin{enumerate}
    \item[\textnormal{(i)}]
    $\lim\limits_{n \to \infty} \into \ln|u_n| h(u_n) v ~dx = \into \ln |u| h(u)v ~dx,$

    \item [\textnormal{(ii)}] $\lim\limits_{n \to \infty} \mathcal{F}_p(u_n,v)= \mathcal{F}_p(u,v)$ and $\lim\limits_{n \to \infty}\int_{\R^N} h(u_n)v ~dx = \int_{\R^N} h(u)v ~dx$,

    \item [\textnormal{(iii)}]  $\lim\limits_{n\to \infty} \into g(x,u_n)v ~dx= \into g(x,u)v ~dx, \lim\limits_{n\to \infty} \into G(x,u_n) ~dx= \into G(x,u) ~dx$ provided $g$ satisfies \eqref{assump g1}-\eqref{assump g3}.
    \end{enumerate}
\end{lemma}

\begin{proof}
First we prove (i). Since $u_n \to u$ in $L^\ph(\Omega)$, Theorem \ref{thm:embd:results} gives
\begin{equation}
\label{convergence}
u_n \to u \ \text{in} \ L^p(\Omega) \ \text{and a.e. in} \  \R^N.
\end{equation}
Let $\delta>0$ and $A \subseteq \Omega$ such that $|A|<\delta$ and $\ph^\ast$ be the conjugate of $\ph$. By the H\"older inequality (\cite[Lemma 3.2.11]{Harjulehto-Hasto}), we have
\begin{equation}\label{holder-est-1}
    \int_{A} |\ln|u_n| h(u_n) v| ~dx \leq 2 \| \ln|u_n|h(u_n) \|_{L^{\ph^\ast}(A)}\|v\|_{L^\ph(A)}.
\end{equation}
    It is easy to see that
    \[
    |\ln|u_n| h(u_n)| \leq \begin{cases}
    \frac{1}{e(p-1)}  & \ \text{if} \ |u_n| < 1\\
    \ln(e+|u_n|)|u_n|^{p-1}  & \ \text{if} \ |u_n| \geq  1.
    \end{cases}
    \]
    Since $\ph$ is increasing and fulfills $(Dec)_{2+\eps}$, by \cite[Proposition 2.4.9]{Harjulehto-Hasto}, $\ph^\ast$ is increasing. Thus, by \cite[Theorem 2.4.8 and Corollary 2.4.11]{Harjulehto-Hasto}, we obtain

    \[
    \begin{split}
    \ph^\ast(|\ln|u_n| h(u_n)|) & \leq \ph^\ast(\ln(e+|u_n|)|u_n|^{p-1}) = \ph^\ast\l(\frac{\ph(|u_n|)}{|u_n|}\r) \\
    & \leq \ph(|u_n|) \quad \text{for} \ |u_n| \geq 1.
    \end{split}
    \]
Since $u_n \to u$ in $L^\ph(\Omega)$, this implies the sequence $(u_n)$ is bounded in $L^\ph(\Omega)$. In view of Lemma \ref{norm_modular} and the above calculations, there exists a $C_1>0$ (independent of $n$ and $A$) such that
\begin{equation}\label{phiast-upper}
    \int_{A} \ph^\ast(|\ln|u_n| h(u_n)|) ~dx \leq \frac{|A|}{e(p-1)} + \int_{A} \ph(|u_n|) ~dx \leq C_1 \ \text{for all} \ n \in \mathbb{N}.
\end{equation}
This further gives $\{\ln|u_n| h(u_n) v\}_{n\in \N}$ is a family of uniformly integrable functions on $\Omega.$ Finally, using \eqref{convergence} and Vitali convergence theorem, we get
 \[\lim\limits_{n \to \infty} \into \ln|u_n| h(u_n) v ~dx = \into \ln |u| h(u)v ~dx. \]
Now, by the symmetry of the integrand in the definition of $\mathcal{F}_p(\cdot,\cdot),$ and the fact that $u_n = 0 =v$ in $\R^N\setminus\Omega$ and \eqref{convergence}, we have
\[\mathcal{F}_p(u_n,v)=C_{N,p}\iint\limits_{\substack{x,y\in \Omega \\ \abs{x-y}\geq1}} \frac{(h(u_n(x)-u_n(y))-h(u_n(x)))v(x)}{|x-y|^{N}} ~dx ~dy\]
and
\[\frac{(h(u_n(x)-u_n(y))-h(u_n(x)))v(x)}{|x-y|^{N}} \to \frac{(h(u(x)-u(y))-h(u(x)))v(x)}{|x-y|^{N}} \ \text{a.e. in} \ \Omega\times\Omega.\]
By the same arguments used in deriving \eqref{p-log-tail-est} and Young's inequality, we get
\[
\begin{split}
\bigg| & \frac{(h(u_n(x)-u_n(y))-h(u_n(x)))v(x)}{|x-y|^{N}}\bigg|\\
         & \qquad \qquad \leq  C(p)
    \begin{cases}
    |u_n(y)|^{p} + |v(x)|^p & \text{if} \ p \in (1,2),\\
    |u_n(y)|^{p} + |u_n(x)|^{p} + |v(x)|^p & \text{if} \ p \geq 2.
    \end{cases}
\end{split}
\]
By \eqref{convergence} and applying Vitali convergence theorem, we get the claims in (ii).
Now, we proceed to prove (iii). By \eqref{assumption g(x,t)}, we have
 \[|g(x,u_n)v| \leq a_1 |u_n|^{p-1}|v|+\eps |u_n|^{p-1}|\ln|u_n|||v|.\]
 Thus, for any $\eta>0$, by choosing $\eps$ small enough in \eqref{assumption g(x,t)}, \eqref{assumption G(x,t)} and using \eqref{convergence}-\eqref{phiast-upper}, there exists a $\delta>0$ such that for all $A \subset \Omega$, $|A|< \delta$, we have
 \[
 \begin{split}
     \l|\int_A g(x,u_n)v ~dx\r| &\leq  a_1 \int_A |h(u_n)| |v| ~dx+ \eps \int_A |h(u_n)| |\ln|u_n|||v|~dx\\
     &\leq a_1 \|u_n\|_{L^p(A)}^{p-1} \|v\|_{L^p(A)} + 2 \eps \| \ln|u_n||h(u_n)| \|_{L^{\ph^\ast}(A)}\|v\|_{L^\ph(A)} < \eta
 \end{split}
 \]
 and
  \[
 \begin{split}
     \l|\int_A G(x,u_n) ~dx\r| &\leq  a_2 \int_A |u_n|^p ~dx+ \eps \int_A |u_n|^p |\ln|u_n||~dx\\
     &\leq a_1 \|u_n\|_{L^p(A)}^{p} +  \eps \left(\frac{|A|}{ep} + \int_A \ph(|u_n|)~dx\right)  < \eta.
 \end{split}
 \]
Finally, applying Vitali's convergence theorem, we get
\[  \lim_{n\to \infty} \into g(x,u_n)v ~dx= \into g(x,u)v ~dx, \quad  \lim_{n\to \infty} \into G(x,u_n) ~dx= \into G(x,u) ~dx.\]
This proves (iii).
\end{proof}
\subsection{Properties of the energy functional}

Recall that by \eqref{energy-wei}, $J_{L_{\Delta_p}}$ is given by
\[
J_{L_{\Delta_p}}(u) = \frac{1}{p} \l(\mathcal{E}_p(u,u) + \rho_N \into |u|^p ~dx\r)- \into G(x,u) ~dx   + I(u),
\]
where
\[
I(u) := \frac{1}{p} \mathcal{F}_p(u,u) - \frac{\mu}{p^2}\into |u|^p(\ln|u|^p -1) ~dx.
\]
\begin{lemma}
    \label{differentiability of J}
Let $g$ satisfy \eqref{assump g1}-\eqref{assump g3}. Then, the energy functional $J_{L_{\Delta_p}}$ defined in \eqref{energy-wei} is of class $C^1$ in $\XO$ and $\text{for all} \ v \in \XO$
\begin{equation}\label{energy:deriv}
    (J'_{L_{\Delta_p}}(u),v)_{\XO} =  \qform(u,v)- \into g(x,u) v ~dx - \mu \into \ln |u|h(u)v ~dx.
\end{equation}
\[\]
Moreover, $J'_{L_{\Delta_p}}(u)\in \mathcal{L}(\XO,\R)$ for every $u\in \XO$, where $\mathcal{L}(\XO,\R)$ is the set of all bounded linear functionals on $\XO.$
    \end{lemma}
    \begin{proof}
    Let $u \in \XO$. Then, by Lemma \ref{compct-embeding} and Lemma \ref{diff-F_p}, we have  \[\frac{1}{p} \l(\mathcal{E}_p(u,u) + \mathcal{F}_p(u,u)+ \rho_N \into |u|^p ~dx\r)< \infty.\]
    Furthermore, in view of \eqref{assumption G(x,t)} and Theorem \ref{log-Sob-ineq}, we can conclude that $J_{L_{\Delta_p}}$ is well-defined. It is easy to see that the norms $\|\cdot\|_{\XO}$ and $\|\cdot\|_{L^p(\Omega)}$
    are differentiable. Now, we check the differentiability of the map $u \mapsto I(u).$
Let $u,v \in \XO$ and  $\delta \in (-1,1) \setminus \{0\}$.
By the mean value theorem, there exists a $\beta = \beta(x) \in (0,1)$ such that
\[
\begin{split}
    \Upsilon((u+\delta v)(x), (u+\delta v)(y), & (u+\delta v)(x), (u+\delta v)(y)) - \Upsilon\l(u(x), u(y), u(x), u(y)\r)\\
    & =  p \delta \Upsilon((u+\delta \beta v)(x), (u+ \beta \delta v)(y), v(x), v(y)),
\end{split}
\]
where $\Upsilon$ is defined in \eqref{def:upsilon}. Now, using the symmetry of the integrand and $u=0=v$ in $\mathbb{R}^N \setminus \Omega$, we obtain
\[
\begin{split}
    & \mathcal{F}_p(u+\delta v,u+\delta v) - \mathcal{F}_p( u,u) \\
    & \quad = \frac{p \delta C_{N,p}}{2} \intb \frac{\Upsilon((u+\delta \beta v)(x), (u+ \beta \delta v)(y), v(x), v(y))}{|x-y|^N} ~dx ~dy \\
    & \quad = \frac{p \delta C_{N,p}}{2} \iint\limits_{\substack{x,y\in \Omega \\ \abs{x-y}\geq1}} \frac{h((u+\delta \beta v)(x)-(u+ \beta \delta v)(y))v(x)-h((u+\delta \beta v)(x))v(x)}{|x-y|^{N}} ~dx ~dy\\
    & \quad := \frac{p \delta C_{N,p}}{2} \iint\limits_{\substack{x,y\in \Omega \\ \abs{x-y}\geq1}} M_1(x,y) ~dx ~dy.
\end{split}
\]
Note that for $x, y \in \mathbb{R}^N$ such that $|x-y| \geq 1$ and using Lemma \ref{Lemma1-Lindgren}, we have
\[
|M_1(x,y)| \leq C (|(u+\delta \beta v)(y)|^{p-1} + |(u+\delta \beta v)(x)|^{p-1}) |v(x)| \in L^1(\Omega\times\Omega)
\]
for some $C>0.$
Now, applying the Lebesgue dominated convergence theorem, we obtain
\[
\lim_{\delta \to 0} \frac{\mathcal{F}_p(u+\delta v,u+\delta v) - \mathcal{F}_p( u,u)}{\delta} = p \mathcal{F}_p(u,v).
\]
Let $\omega(t):= |t|^p(\ln |t|^p -1)$ such that $\omega'(t)=p^2|t|^{p-2}t\ln |t|.$  By the mean value theorem, there exists a $\tau=\tau(x) \in (0,1)$ such that
\[
\begin{split}
Q(x)&:= \frac{\omega(u(x)+\delta v(x))-\omega(u(x))}{\delta}\\
&= p^2|u(x)+\delta\tau v(x)|^{p-2}(u(x)+\delta\tau v(x)) \ln |u(x)+\delta\tau v(x)| v(x).
\end{split}\]
Note that if $|u(x)+\delta \tau v(x)| < 1$, we have
\[|Q(x)|\leq \frac{p^2}{e(p-1)},\]
and if $|u(x)+ \delta\tau v(x)| \geq 1$, by \cite[Section 2.2]{Iannizzotto-Mosconi-Papageorgiou}, we have
\[1\leq |u(x)+\tau \delta v(x)|\leq |u(x)|+|v(x)|\leq 2 \max\{|u(x)|,|v(x)|\}\]
and
\[
\begin{split}
    |Q(x)|&\leq p^2 |v(x)||u(x)+\delta\tau v(x)|^{p-1} |\ln|u(x)+\delta\tau v(x)|| \\
    & \leq 4p^2 \max\{1,2^{p-2}\}\bigg(\max\{|u(x)|,|v(x)|\}.\max\{|u(x)|^{p-1},|v(x)|^{p-1}\}.\\& \qquad \qquad  \qquad \qquad \qquad \qquad |\ln\l(2\max\{|u(x)|,|v(x)|\}\r)|\bigg)\\
    & \leq 4p^2 \max\{1,2^{p-2}\}\bigg(|u(x)|^p |\ln 2|u(x)||+|v(x)|^p |\ln 2|v(x)||\bigg)=:M_2(x).
\end{split}
\]
Now, using Theorem \ref{log-Sob-ineq} and Theorem \ref{thm:embd:results} (i), we get
\[
\begin{split}
    \into |u|^p |\ln (2|u|)| ~dx & \leq \ln 2\into |u|^p ~dx +\into |u|^p |\ln(|u|)| ~dx\\
    & \leq \ln 2 \into |u|^p ~dx + \int_{\Omega} |u|^p \ln(|u|) ~dx \\
    & \quad - \int_{\{|u|< 1\}} |u|^p (\ln |u| - |\ln(|u|)|) ~dx  \\
    &\leq \frac{N}{p^2}\bigg(\qform(u,u)+ \frac{p^2}{N} \nP^p \ln \nP \\ & \qquad \qquad \qquad+ (k_0+\ln 2)\nP^p\bigg)+ \frac{2|\Omega|}{pe}<\infty.
\end{split}
\]
As a conclusion, we obtain
\[|Q(x)|\leq M_2(x)+\frac{p^2}{e(p-1)} \in L^1(\Omega).\]
Finally, applying the Lebesgue dominated convergence theorem, we can conclude that
\[
\lim_{\delta \to 0} \frac{F(u+\delta v)-F(u)}{\delta}= \mu \into \ln |u| h(u)v ~dx,
\]
where $$F(u):=\frac{\mu}{p^2}\into |u|^p(\ln|u|^p-1) ~dx.$$
Again, by the fundamental theorem of calculus, \eqref{assumption G(x,t)} and the preceding calculations, we can easily conclude that $u \mapsto \into G(x,u) ~dx$ is differentiable.
Using the linearity of $J'_{L_{\Delta_p}}(u)$, H\"older inequality, Theorem \ref{log-Sob-ineq} and Lemma \ref{diff-F_p} for $u,v, w \in \XO$, it is easy to see that
\[
\lim_{\delta \to 0}|(J'_{L_{\Delta_p}}(u), v+\delta w)_{\XO} - (J'_{L_{\Delta_p}}(u), v)_{\XO}|
= 0.
\]
This implies $J'_{L_{\Delta_p}}(u) \in \mathcal{L}(\XO,\R).$ To prove the $C^1$ property, we show that if $u_n \to u$ in $\XO$
\begin{equation}
\label{C1-prop}
\lim_{n \to \infty}|(J'_{L_{\Delta_p}}(u_n), v)_{\XO} - (J'_{L_{\Delta_p}}(u), v)_{\XO}| = 0.
\end{equation}

Note that
\[\qform(u_n,v)=\mathcal{E}_p(u_n,v)+\mathcal{F}_p(u_n,v)+\rho_N \int_{\R^N} h(u_n)v ~dx.\]
\textbf{Claim:} $\lim\limits_{n \to \infty} \mathcal{E}_p(u_n,v)= \mathcal{E}_p(u,v)$.
Note that
\[
\begin{split}
 &|\mathcal{E}_p(u_n,v)-\mathcal{E}_p(u,v)|\\& \leq \frac{C_{N,p}}{2}\inta \frac{|(h(u_n(x)-u_n(y))-h(u(x)-u(y)))||(v(x)-v(y))|}{|x-y|^N} ~dx ~dy.
\end{split}
\]
Now, by Lemmas \ref{Lemma1-Lindgren}, \ref{section2.2-Mosconi} and H\"older inequality, we obtain,
if $p \in (1,2]$,
\[
\begin{split}
    &\frac{C_{N,p}}{2} \inta \frac{|(h(u_n(x)-u_n(y))-h(u(x)-u(y)))||(v(x)-v(y))|}{|x-y|^N} ~dx ~dy \\&\leq k_1(p) \frac{C_{N,p}}{2} \inta \frac{|((u_n-u)(x)-(u_n-u)(y))|^{p-1}|v(x)-v(y)|}{|x-y|^N} ~dx ~dy\\
    &\leq k_1(p) \mathcal{E}_p(u_n-u,u_n-u)\mathcal{E}_p(v,v),
\end{split}
\]
where $k_1(p) = (3^{p-1} + 2^{p-1})$ and if $p > 2$, we obtain
\[
\begin{split}
   & \frac{C_{N,p}}{2} \inta \frac{|(h(u_n(x)-u_n(y))-h(u(x)-u(y)))||(v(x)-v(y))|}{|x-y|^N} ~dx ~dy \\
  & \leq k_2(p) \frac{C_{N,p}}{2} \bigg(\inta \frac{|v(x)-v(y)||(u_n-u)(x)-(u_n-u)(y)|^{p-1}}{|x-y|^N} ~dx ~dy\\
  &+ \inta \frac{|v(x)-v(y)||u(x)-u(y)|^{p-2}|(u_n-u)(x)-(u_n-u)(y)|}{|x-y|^N} ~dx ~dy\bigg)\\
  & \leq k_2(p) \mathcal{E}_p(u_n-u,u_n-u)\mathcal{E}_p(v,v) \l(\mathcal{E}_p(u,u)+1\r).
\end{split}
\]
where $k_2(p)= 2^{p-2} (p-1).$ Finally, using the above calculations and strong convergence of $(u_n)_{n \in \N}$ in $\XO$, we get the claim above. \\
Collecting the claim and Lemma \ref{ln-limit} (ii), we get
\begin{equation}
\label{quadlim}
    \lim_{n \to \infty} \qform(u_n,v)= \qform(u,v).
    \end{equation}
Finally, putting together the above estimates and Lemma \ref{ln-limit}, we have proved \eqref{C1-prop}.
\end{proof}

\subsection{The Brezis-Nirenberg type problem}
In this part, we consider the problem \eqref{main_problem} with $\mu \in (0,\frac{p^2}{N})$, having critical growth nonlinearities.
To find solutions, we minimize the energy functional $J_{L_{\Delta_p}}$ over $\XO$.  Let $\Phi:\R^+ \to \R$ be defined as
\begin{equation}
\label{def:fibermap}
\Phi(t)=J_{L_{\Delta_p}}(tu)=\frac{t^p}{p}\qform(u,u)-\into G(x,tu) ~dx-\frac{\mu t^p}{p^2}\into |u|^p(\ln |tu|^p-1) ~dx.
\end{equation} However, by \eqref{assumption G(x,t)} and choosing $\eps<\frac{\mu}{p}$, for $t>0,$ we have
  \begin{equation}
  \label{J not bdd below}
  \begin{split}
  \Phi(t)& \leq \frac{t^p}{p}\qform(u,u)+a_2t^p\into |u|^p ~dx+ \eps t^p\into |u|^p |\ln|tu|| ~dx \\
  & \qquad +\frac{\mu t^p}{p^2}\into |u|^p ~dx - \frac{\mu t^p}{p^2}\into |u|^p\ln |tu|^p ~dx\\
  & \leq \frac{t^p}{p}\qform(u,u) + \eps t^p\into |u|^p |\ln|u|| ~dx + \eps t^p |\ln t|\into |u|^p ~dx\\
  & \qquad + t^p \l(\frac{\mu}{p^2} - \frac{\mu \ln t}{p} + a_2 \r)\into |u|^p ~dx -\frac{\mu t^p}{p}\into |u|^p\ln |u| ~dx \\
  &\leq \frac{t^p}{p}\bigg(\qform(u,u)+\l(\frac{\mu}{p}-\mu \ln t+\eps p |\ln t| + p a_2\r) \into |u|^p ~dx \\& \qquad \qquad-\mu \into |u|^p \ln|u| ~dx +\eps p \into |u|^p |\ln|u|| ~dx\bigg).
  \end{split}
  \end{equation}
This implies $J_{L_{\Delta_p}}(tu) \to -\infty$ as $t\to \infty.$ Thus, $J_{L_{\Delta_p}}$ is not bounded below and not coercive across the entire energy space $\XO$.
Consequently, we seek solutions under the natural constraint given by the Nehari manifold $\mathcal{N}_{L_{\Delta_p}}$. We begin by demonstrating that the energy functional $J_{L_{\Delta_p}}$ is bounded from below in $\mathcal{N}_{L_{\Delta_p}}$ and $\mathcal{N}_{L_{\Delta_p}}$ is away from origin in $\XO$.
\begin{lemma}\label{Nehrai:prop}
Let $g$ satisfy \eqref{assump g1}-\eqref{assump g4}. The following statements hold true:
\begin{enumerate}
  \item  [\textnormal{(i)}]$J_{L_{\Delta_p}}$ is bounded below in $\mathcal{N}_{L_{\Delta_p}}$,
  \item [\textnormal{(ii)}] The set $\mathcal{N}_{L_{\Delta_p}}$ is away from origin in $\XO$, {\it i.e.}, there exist constants $\zeta, \eta>0$ such that
  \begin{equation*}
      \nX \geq \zeta \quad \text{and} \quad \nP\geq \eta \quad \text{for all} \ u \in \mathcal{N}_{L_{\Delta_p}}.
  \end{equation*}
    \end{enumerate}
\end{lemma}
\begin{proof}
    Let $u\in \mathcal{N}_{L_{\Delta_p}}.$ Using \eqref{assump integ g4}, we have
    \[
    \begin{split}
    J_{L_{\Delta_p}}(u)
    &=\frac{1}{p} \into g(x,u)u ~dx -\into G(x,u) ~dx+\frac{\mu}{p^2}\into |u|^p ~dx \\
    & \geq \frac{(\mu-\delta)}{p}\|u\|_{L^p(\Omega)}^p>0.
    \end{split}
    \]
Hence, we get the claim in \text{(i)}. To show \text{(ii)}, set $\lambda:= \frac{(\mu+\eps) N}{p^2}$ and choose $\eps<\l(\frac{p^2}{N}-\mu\r)$.
Note that
\begin{align}
\label{lower-est-reminder}
\begin{split}
    \int_{\{|u|<1\}} |u|^p \ln |u|^p ~dx&= \ln (\|u\|_{L^p(\Omega)}^p) \int_{\{|u|<1\}} |u|^p ~dx \\ & \quad+\|u\|_{L^p(\Omega)}^p \int_{\{|u|<1\}} \l(\frac{|u|^p}{\|u\|_{L^p(\Omega)}^p}\r) \ln  \l(\frac{|u|^p}{\|u\|_{L^p(\Omega)}^p} \r) ~dx\\
    & \geq \ln (\|u\|_{L^p(\Omega)}^p) \int_{\{|u|<1\}} |u|^p ~dx-\frac{|\Omega|}{e}\|u\|_{L^p(\Omega)}^p.
\end{split}
\end{align}
Let $u \in \mathcal{N}_{L_{\Delta_p}}$ such that $\nP< \min\{1,\eta\}$, where \[\eta:=\exp\l(\frac{1}{\mu-\eps} \l((1-\lambda)\lambda_{1,L} -\lambda \bigg(k_0+\frac{2\eps|\Omega|}{ep\lambda}+ \frac{a_1}{\lambda}\bigg)\r)\r)\] and $\lambda_{1,L}$ is the first eigenvalue of $\logP$ in $\Omega$ (see \cite[Section 7]{Dyda-Jarohs-Sk}). Now, using \eqref{assumption g(x,t)}, \eqref{lower-est-reminder} and Theorem \ref{log-Sob-ineq} in \eqref{energy:deriv}, we obtain
    \[
    \begin{split}
        (J'_{L_{\Delta_p}}(u),u)_{\XO} &\geq \qform(u,u)-\l(a_1+\frac{2\eps|\Omega|}{ep}\r) \into |u|^p ~dx \\ & \quad- (\mu+\eps) \into \ln|u||u|^p ~dx+2\eps \int_{\{|u|<1\}} |u|^p \ln |u| ~dx\\
        &\geq  (1-\lambda)\qform(u,u) -\lambda \bigg(k_0+\frac{2\eps|\Omega|}{ep\lambda}+ \frac{a_1}{\lambda}\bigg)\nP^p\\
        & \qquad -(\mu+\eps)\nP^p\ln (\nP)\\ &\qquad \qquad+2\eps \ln(\nP) \int_{\{|u|<1\}} |u|^p ~dx\\
        &\geq \bigg[(1-\lambda)\lambda_{1,L} -\lambda \bigg(k_0+\frac{2\eps|\Omega|}{ep\lambda}+ \frac{a_1}{\lambda}\bigg)\\& \qquad\qquad \qquad \qquad-(\mu-\eps)\ln(\nP)\bigg]\nP^p >0.
    \end{split}
    \]
However, for every $u \in \mathcal{N}_{L_{\Delta_p}},$ we have $(J'_{L_{\Delta_p}}(u),u)_X=0$, which in turn implies that \begin{equation}
    \label{low_bd_norm-Lp}
    \nP\geq \eta \quad \text{for all} \ u \in \mathcal{N}_{L_{\Delta_p}}.
    \end{equation}
Moreover, by Lemma \ref{compct-embeding} (with $S$ being the embedding constant) and \eqref{low_bd_norm-Lp}, we get
   \[\eta^p \leq \nP^p \leq S^p \nX^p.\] Finally, the claim in \textnormal{(ii)} follows by taking $\zeta=\frac{\eta}{S}.$
\end{proof}
Now, we consider the fibering map, introduced by Drabek and Pohozaev in \cite{Drabek-Pohozaev}. Let $\Phi: \R^+ \to \R$  be the fibering map for the functional $J_{L_{\Delta_p}}$, defined in \eqref{def:fibermap}.

\begin{lemma}
\label{t-ast}
Let $g$ satisfy \eqref{assump g1}-\eqref{assump g3}. Then, for every $u \in \XO \setminus \{0\}$,
   there exists a point $t^\ast_u$ such that $\Phi'(t^\ast_u) =0$ and $\Phi'(t) >0$ for $t \in (0, t^\ast_u).$ Moreover, $t^\ast_u u\in \mathcal{N}_{L_{\Delta_p}}.$
\end{lemma}
\begin{proof}
By \eqref{assumption G(x,t)} and choosing $\eps<\frac{\mu}{p}$, for $0<t<1,$ we have
  \[
  \begin{split}
  \Phi(t)&\geq\frac{t^p}{p}\qform(u,u)-a_2 t^p\into |u|^p ~dx- \eps t^p\into |u|^p |\ln|tu|| ~dx+\frac{\mu t^p}{p^2}\into |u|^p ~dx\\
  & \qquad \qquad-\frac{\mu t^p}{p^2}\into |u|^p\ln |tu|^p ~dx\\
  & \geq\frac{t^p}{p}\qform(u,u)-a_2t^p\into |u|^p ~dx-\eps t^p \into |u|^p |\ln|u|| ~dx-\eps t^p |\ln t| \into |u|^p ~dx \\
  & \qquad\qquad +\frac{\mu t^p}{p^2}\into |u|^p ~dx -\frac{\mu t^p}{p}\into |u|^p\ln |u| ~dx -\frac{\mu t^p \ln t}{p}\into |u|^p ~dx\\
 &\geq \frac{t^p}{p}\bigg(\qform(u,u)+\l(\frac{\mu}{p}-\mu\ln t-\eps p |\ln t| - p a_2\r) \into |u|^p ~dx \\& \qquad \qquad-{\mu}\into |u|^p \ln|u| ~dx -\eps p \into |u|^p |\ln |u|| ~dx\bigg).
  \end{split}\]
The above  along with \eqref{J not bdd below}, implies that there exists $t^\ast_u>0$ small enough such that
\[ \lim\limits_{t \to 0} \Phi(t)=0, \quad \lim\limits_{t \to \infty} \Phi(t)=-\infty \quad \text{and} \quad \Phi(t)>0 \quad \text{for} \ t \in (0, t^\ast_u).
\] Thus, there exists a critical point $t^\ast_u$ of $\Phi$ such that $\Phi'(t) >0$  for $t \in (0, t^\ast_u)$. It is easy to verify that $t^\ast_u u \in \mathcal{N}_{L_{\Delta_p}}.$
\end{proof}

\begin{lemma}
\label{bdd_seq_XO}
    Let $g$ satisfy \eqref{assump g1}-\eqref{assump g4} and $(u_n)_{n\in \N} \subseteq \mathcal{N}_{L_{\Delta_p}}$ be a sequence such that $\sup\limits_{n\in \N} J_{L_{\Delta_p}}(u_n)\leq C$ for some constant $C>0.$ Then, $\|u_n\|_{\XO}\leq C'$ for some constant $C'>0.$
\end{lemma}
\begin{proof}
Since $u_n \in \mathcal{N}_{L_{\Delta_p}}$, for each $n \in \mathbb{N}$, we have \[J_{L_{\Delta_p}}(u_n)= \frac{1}{p} \into g(x,u_n)u_n ~dx -\into G(x,u_n)~dx+\frac{\mu}{p^2} \into |u_n|^p ~dx.\]
     This alongwith \eqref{assump integ g4} yields
    $ \sup_{n\in \N} \|u_n\|^p_{L^p(\Omega)} \leq \frac{Cp^2}{(\mu-\delta)}:= C_1. $
    Using \eqref{assumption G(x,t)}, choosing $\eps<\l(\frac{p}{N}-\frac{\mu}{p}\r)$ and Theorem \ref{log-Sob-ineq}, we get
    \[
    \begin{split}
        C\geq &J_{L_{\Delta_p}}(u_n) 
        \geq  \frac{1}{p}\qform(u_n,u_n)-a_2\into |u_n|^p ~dx- \eps \into |u_n|^p|\ln|u_n|| ~dx\\
        &\qquad\qquad \qquad\qquad-\frac{\mu}{p^2}\into |u_n|^p(\ln|u_n|^p -1) ~dx\\
        & \qquad \qquad\geq  \frac{1}{p}\qform(u_n,u_n)-a_2\into |u_n|^p ~dx- \l(\eps+\frac{\mu}{p}\r) \into |u_n|^p\ln|u_n| ~dx\\
        &\qquad\qquad \qquad\qquad+2\eps \int_{\{|u_n| < 1\}} |u_n|^p \ln|u_n| ~dx+\frac{\mu}{p^2}\into |u_n|^p ~dx\\
        & \qquad\qquad \geq \l(\frac{1}{p}- \frac{\eps N}{p^2}-\frac{\mu N}{p^3}\r)\qform(u_n,u_n)-\bigg[\l(a_2-\frac{\mu}{p^2}\r)\\
        &\qquad \qquad \quad \quad+\l(\frac{\eps N}{p^2}+\frac{\mu N}{p^3}\r)k_0 +\l(\eps+\frac{\mu}{p}\r)\ln(\|u_n\|_{L^p(\Omega)})\bigg]\|u_n\|^p_{L^p(\Omega)}\\&\qquad \qquad \qquad \qquad \qquad-\frac{2\eps|\Omega|}{ep}.
    \end{split}
    \]
Since $\mu \in (0, \frac{p^2}{N})$, we further obtain
    \[
    \begin{split}
    \sup_{n\in \N}\qform(u_n,u_n) \leq \frac{p^3(C+C_2+\frac{2\eps|\Omega|}{ep})}{p^2-p\eps N-\mu N}=:C_3,
    \end{split}
    \]
    where
    \[
    C_2:= \sup\limits_{z \in [0, C_1]} \bigg[\l(a_2-\frac{\mu}{p^2}\r)+\l(\frac{\eps N}{p^2}+\frac{\mu N}{p^3}\r)|k_0|+\l(\eps+\frac{\mu}{p}\r)|\ln z|\bigg]|z|^p.
    \]
Finally, using the definition of $\qform$ in \eqref{quad:form} and Lemma \ref{diff-F_p} in the above estimates, we obtain
    \[ \|u_n\|_{\XO} \leq  \l(C_3-\l(\rho_N-b\r)C_1\r)^{\frac{1}{p}}. \]
\end{proof}
\noindent
\textit{Proof of Theorem \ref{Existence L1}:}
Let $\ph:\XO\to \R$ be defined as \[\ph(u)= \qform(u,u)-\into g(x,u)u ~dx-\mu\into  |u|^p \ln|u| ~dx.\]
By following the same kind of arguments used in Lemma \ref{differentiability of J} and \eqref{assump g5}, it is easy to establish that $\ph$ is a $C^1$ functional on $\XO\setminus\{0\}$ and for $v \in \XO$,
\[
\begin{split}
(\ph'(u),v)_{\XO} = p\qform(u,v)&-\into g(x,u)v ~dx - \into g'(x,u)uv ~dx\\&-p\mu\into  \ln|u| h(u)v ~dx-\mu\into h(u)v ~dx.
\end{split}
\]
Note that by assumption \eqref{assump g4}, if $u\in \mathcal{N}_{L_{\Delta_p}}$,
\[
(\ph'(u),u)_X = \l[(p-1)\into g(x,u)u~dx -\into g'(x,u)u^2 ~dx -\mu \into |u|^p~dx\r]< 0,
\]
which in turn implies $0$ is a regular value of $\ph$. Therefore, $\mathcal{N}_{L_{\Delta_p}} = \ph^{-1} (0)$ is a $C^1$ manifold. It is easy to see that any minimizer $u$ of $J_{L_{\Delta_p}}$ restricted to $\mathcal{N}_{L_{\Delta_p}}$ satisfies $\ph(u)=0.$ In view of this property, we can apply the critical point theory on $\mathcal{N}_{L_{\Delta_p}}$ to get critical points of $J_{L_{\Delta_p}}.$ By Ekeland's variational principle (see \cite{Ekeland}), there exists a minimizing sequence $(u_n)_{n\in\N} \subseteq\mathcal{N}_{L_{\Delta_p}}$ and $(d_n)_{n\in\N} \subseteq \R$, such that
    \begin{equation}
    \label{ekeland-var-pri}
    0\leq J_{L_{\Delta_p}}(u_n)-\inf_{\mathcal{N}_{L_{\Delta_p}}} J_{L_{\Delta_p}} \leq \frac{1}{n^2} \quad \text{and} \quad \l\|J'_{L_{\Delta_p}}(u_n) -d_n \ph'(u_n)\r\|_{\mathcal{L}(\XO,\R)}\leq \frac{1}{n}.
    \end{equation}
This and Lemma \ref{Nehrai:prop} (i) imply that $J_{L_{\Delta_p}}(u_n)$ is bounded. Applying Lemma \ref{bdd_seq_XO} and using the reflexivity of $\XO$ (see \cite[Section 4]{Dyda-Jarohs-Sk}), we get a weakly convergent subsequence of $(u_n)_{n\in\N}$ (still denoted the same) such that $u_n \rightharpoonup u_0$ in $\XO$ for some $u_0 \in \XO.$ From \eqref{ekeland-var-pri}, we get
\[
    \begin{split}
    o(1) &=\frac{(J'_{L_{\Delta_p}}(u_n),u_n)_X-d_n(\ph'(u_n),u_n)_X}{\|u_n\|^p_{\XO}}\\
    &= -d_n \l[\frac{(p-1)\into g(x,u_n)u_n~dx -\into g'(x,u_n)u_n^2 ~dx -\mu \into |u_n|^p~dx}{\|u_n\|^p_{\XO}}\r].
    \end{split}
    \]
By Lemma \ref{Nehrai:prop} (ii) and \eqref{assump g4}, we get $d_n \to 0$ as $n\to \infty.$  Note that by \eqref{assump g5}, for a $\eps_1>0$, there exists a constant $C>0$ such that
\[|tg'(x,t)| \leq \eps_1|t|^{p-1} |\ln|t||+C \quad \text{for all} \ t \in \R, \ \text{a.e.} \  x \in \Omega.\]
Using this, the fact that $(u_n)_{n \in \N}$ is bounded and \eqref{assumption g(x,t)}, we get
    \[
    \begin{split}
|(\ph'(u_n),v)_{\XO}|
& \leq |p\qform(u_n,v)|+ \l|\into g(x,u_n)v ~dx\r|+ \l|\into g'(x,u_n)u_nv ~dx\r|\\ & \qquad+\l|p\mu\into  \ln|u_n| h(u_n)v ~dx\r|+ \l|\mu\into h(u_n)v ~dx\r| \\ 
& \leq |p \qform(u_n,v)|+(p\mu +\eps+\eps_1) \into  |\ln|u_n|| |u_n|^{p-1} |v| ~dx\\ & \qquad+ (\mu +a_1)\into |u_n|^{p-1}|v| ~dx+C \into |v| ~dx.
    \end{split}
    \]
This further gives in light of Theorem \ref{log-Sob-ineq},    
        \[
    \begin{split}
|(\ph'(u_n),v)_{\XO}|
& \leq |p \qform(u_n,v)|+(p\mu +\eps+\eps_1) \into  \ln|u_n| |u_n|^{p-1} |v| ~dx\\ & \qquad-2(p\mu +\eps+\eps_1) \int_{\{|u_n|<1\}}  \ln|u_n| |u_n|^{p-1} |v| ~dx\\ &\qquad+ (\mu +a_1)\into |u_n|^{p-1}|v| ~dx+C\into |v| ~dx\\
&\leq C'' \quad \text{for} \  v \in \XO, \ \|v\|_{\XO}=1,
    \end{split}
    \]
    where $C''>0$ is a constant. Using this, \eqref{ekeland-var-pri} and triangle inequality, we get
    \[
    \begin{split}
    &\|J'_{L_{\Delta_p}}(u_n)\|_{\mathcal{L}(\XO,\R)}-d_n\|\ph'(u_n)\|_{\mathcal{L}(\XO,\R)} \\
    & \qquad  \qquad  \qquad  \qquad \qquad \leq \l\|J'_{L_{\Delta_p}}(u_n) -d_n \ph'(u_n)\r\|_{\mathcal{L}(\XO,\R)}\leq \frac{1}{n}.
    \end{split}
    \]
    Therefore,
    \[
    \|J'_{L_{\Delta_p}}(u_n)\|_{\mathcal{L}(\XO,\R)}\leq \l(d_nC''+\frac{1}{n}\r) \to 0 \ \text{as} \ n\to \infty.
    \]
From the above, for all $v\in \XO$, we have
   \begin{equation}
\label{weak-soln}
     \begin{split}
         0 & = \lim_{n\to \infty} (J'_{L_{\Delta_p}}(u_n),v)_X\\
        &=\lim_{n \to \infty} \bigg(\mathcal{E}_p(u_n,v)+\mathcal{F}_p(u_n,v)+\rho_N \int_{\R^N} h(u_n)v ~dx\\& \qquad\qquad \qquad \qquad \qquad - \into g(x,u_n) v ~dx - \mu\into \ln|u_n|h(u_n) v ~dx\bigg).
        \end{split}
        \end{equation}
       Since $u_n \rightharpoonup u_0 $ in $\XO,$ by Lemma \ref{compct-embeding}, $u_n \to u_0$ in $L^p(\Omega).$ Now, by Lemma \ref{ln-limit} (i), (ii), (iii), for all $v\in \XO$, we obtain
        \[\lim_{n\to \infty} \mathcal{F}_p(u_n,v)=\mathcal{F}_p(u_0,v), \quad  \lim_{n \to \infty} \int_{\R^N} h(u_n)v ~dx = \int_{\R^N} h(u_0)v ~dx,\]
       \[\lim_{n \to \infty} \into g(x,u_n) v ~dx= \into g(x,u_0) v ~dx \] and \[\lim_{n \to \infty} \into \ln|u_n|h(u_n) v ~dx = \into \ln|u_0|h(u_0) v ~dx .\]
        In order to conclude that $u_0$ is a weak solution to \eqref{main_problem}, it only remains to show that
        \[\lim_{n \to \infty} \mathcal{E}_p(u_n,v) =\mathcal{E}_p(u_0,v) \quad \text{for all} \ v \in \XO.\]
         Denote
        \[\mathcal{D}:=\R^{2N}\cap\{x,y \in \R^N \,:\, |x-y|\leq 1\}.\]
        By the definition of $\mathcal{E}_p(\cdot, \cdot)$, we set
        \[
        \begin{split}
        \mathcal{E}_p(u_n,v):= \iint_{\mathcal{D}} D_{0,p'}(u_n) D_{0,p}(v) ~dx ~dy,
        \end{split}
        \]
        where $p'$ is the H\"older conjugate of $p$,
        \[D_{0,p'}(w):=\frac{h(w(x)-h(w(y)))}{|x-y|^{N/p'}} \quad \text{and} \quad  D_{0,p}(w):=\frac{(w(x)-w(y))}{|x-y|^{N/p}}.\]
        For a given $w \in \XO$, it is easy to observe that $D_{0,p'} (w) \in L^{p'}(\mathcal{D})$ and $D_{0,p} (w) \in L^{p}(\mathcal{D}).$
        Since $(u_n)_{n\in \N}$ is a bounded sequence in $\XO$, $(D_{0,{p'}}(u_n))_{n \in \mathbb{N}}$ is a bounded sequence in $L^{p'}(\mathcal{D})$. Let $v\in\XO$ and define a functional $T_v: L^{p'}(\mathcal{D}) \to \R $ such that
        \[T_v(\omega)= \iint_{\mathcal{D}} \omega D_{0,p}(v) ~dx ~dy.\]
        It is easy to deduce that $T_v$ is a linear and bounded functional on $L^{p'}(\mathcal{D})$. Furthermore, using $u_n \to u$ a.e. in $\Omega$ and the boundedness of $(D_{0,{p'}}(u_n))_{n\in \N}$ in $L^{p'}(\mathcal{D})$, we have $D_{0,{p'}}(u_n) \rightharpoonup D_{0,{p'}}(u_0)$ in $L^{p'}(\mathcal{D})$ (up to a subsequence). Moreover,
        \[\lim_{n \to \infty} \mathcal{E}_p(u_n,v)= \lim_{n \to \infty} T_v(D_{0,{p'}}(u_n))= T_v(D_{0,{p'}}(u_0))= \mathcal{E}_p(u_0,v).\]
        Combining all the above, we get that $u_0$ is a weak solution to \eqref{main_problem}. In particular, taking $v=u_0, $ we get  $u_0 \in \mathcal{N}_{L_{\Delta_p}}.$ By Lemma \ref{ln-limit}, we get
    \[
    \begin{split}
    \inf_{\mathcal{N}_{L_{\Delta_p}}} J_{L_{\Delta_p}} &= \lim_{n \to \infty} J_{L_{\Delta_p}}(u_n) \\
    &= \lim_{n\to \infty} \l(\frac{1}{p} \into g(x,u_n)u_n ~dx-\into G(x,u_n) ~dx +\frac{\mu}{p^2} \into |u_n|^p ~dx\r) \\
    &= \frac{1}{p} \into g(x,u_0)u_0 ~dx-\into G(x,u_0) ~dx +\frac{\mu}{p^2} \into |u_0|^p ~dx=J_{L_{\Delta_p}}(u_0).
    \end{split}
    \]
    Hence, $u_0$ has least energy. We now claim that $u_0$ does not change sign if $G(x,t)\leq G(x,|t|)$ for all $t \in \mathbb{R},$ a.e. $x \in \Omega.$ Indeed,
\[
\begin{split}
    J(|u_0|)&=\frac{1}{p}\qform(|u_0|,|u_0|)- \into G(x,|u_0|) ~dx-\frac{\mu}{p^2}\into |u_0|^p(\ln(|u_0|^p)-1) ~dx\\
    &\leq\frac{1}{p}\qform(u_0,u_0)- \into G(x,u_0) ~dx-\frac{\mu}{p^2}\into |u_0|^p(\ln(|u_0|^p)-1) ~dx\\
   &=J(u_0)=\inf_{\mathcal{N}_{\logP}} J.
\end{split}
\]
Thus, $J(|u_0|)=J(u_0)$.  From this and \cite[Lemma 4.7]{Dyda-Jarohs-Sk}, we get \[\qform(u_0,u_0) = \qform(|u_0|,|u_0|).\]
   Again, by \cite[Lemma 4.7]{Dyda-Jarohs-Sk}, we get that $u_0$ does not change sign in $\Omega.$
 This completes the proof of the theorem.
\qed

\subsection{The Logistic type problem}
Throughout this subsection, we consider the problem \eqref{main_problem} with $\mu \in (-\infty,0).$ \vspace{0.1cm}\\
\textit{Proof of Theorem \ref{Existence L2}:}
    Using the definitions of $J_{L_{\Delta_p}}$, $\qform$, \eqref{assumption G(x,t)}, choosing $\eps<\frac{-\mu}{p}$ and Lemma \ref{diff-F_p}, we get
    \[
    \begin{split}
    J_{L_{\Delta_p}}(u)&\geq \frac{1}{2p}\nX^p-\frac{1}{p}\l(b-\rho_N+pa_2-\frac{\mu}{p}\r)\nP^p\\
    &\qquad -\l(\frac{\mu}{p}+\eps\r)\into |u|^p \ln |u| ~dx+2\eps \int_{\{|u|<1\}} |u|^p \ln |u| ~dx,
    \end{split}
    \]
    where $b$ is a constant depending on $N$ and $p.$
    Denote
    \[ \Omega_1=\left\{x \in \Omega \,:\, \ln |u|^p >\frac{-p}{(\mu+p\eps)}\l(b-\rho_N+pa_2-\frac{\mu}{p}\r)\right\}.\]
    Then, we have
    \[
    \begin{split}
        J_{L_{\Delta_p}}(u)&\geq \frac{1}{2p}\nX^p-\frac{1}{p}\l(b-\rho_N+pa_2-\frac{\mu}{p}\r)\int_{\Omega\setminus\Omega_1} |u|^p ~dx.\\
    &\qquad -\l(\frac{\mu}{p}+\eps\r)\int_{\Omega\setminus\Omega_1} |u|^p \ln |u| ~dx-\frac{2\eps |\Omega|}{ep}.
    \end{split}
    \]
    Since $ |u|^p <\exp\l(\frac{-p}{(\mu+\eps)}\l(b-\rho_N+pa_2-\frac{\mu}{p}\r)\r)$ in $\Omega\setminus\Omega_1,$ there exists a constant $c>0$ such that
    \[-\frac{1}{p}\l(b-\rho_N+pa_2-\frac{\mu}{p}\r) \int_{\Omega\setminus\Omega_1} |u|^p ~dx-\l(\frac{\mu}{p} +\eps\r)\int_{\Omega\setminus\Omega_1} |u|^p \ln |u| ~dx>-c.\]
     From above, we get that \[J_{L_{\Delta_p}}(u) \geq \frac{1}{2p}\nX^p-c,\] and thus, we can conclude that $J_{L_{\Delta_p}}$ is coercive.
     Let $(u_k)_{k\in\N} \subseteq \XO$ be a minimizing sequence such that\[\lim_{k \to \infty} J_{L_{\Delta_p}}(u_k)= \inf_{ u\in \XO} J_{L_{\Delta_p}}(u)=:m.\]
     By the coercivity of the energy functional $J_{L_{\Delta_p}}$, $(u_k)_{k\in \N}$ is bounded in $\XO$. Moreover, by the reflexivity of $\XO$ and Lemma \ref{compct-embeding}, for some $u_1 \in \XO,$ we have (up to a subsequence)
     \[u_k \rightharpoonup u_1 \ \text{in} \ \XO, \quad  u_k \to u_1 \ \text{in} \  L^p(\Omega)\ \text{and a.e. in} \ \Omega. \]
     \textbf{Claim:} $J_{L_{\Delta_p}}(u_1) \leq \liminf_{k \to \infty} J_{L_{\Delta_p}}(u_k) = m.$\\
     Since the function $|t|^p \ln |t|$ is bounded below by its minima $-\frac{1}{ep}$, by Fatou's lemma, we have
     \begin{equation}
     \label{2}
     \into |u_1|^p \ln |u_1| ~dx\leq \liminf_{k\to \infty} \into |u_k|^p \ln |u_k| ~dx.
     \end{equation}
     By the definition of $\qform(\cdot,\cdot)$, we have
     \[
     \begin{split}
    \qform(u_k,u_k)= \mathcal{E}_p(u_k,u_k)+ \mathcal{F}_p(u_k,u_k)+ \rho_N \int_{\R^N} |u_k(x)|^p ~dx.
     \end{split}
     \]
   Since $u_k \rightharpoonup u_1$ in $\XO,$ using the weak lower semicontinuity of the norm $\|\cdot\|$, we get
     \begin{align}
     \label{3}
     \begin{split}
     \mathcal{E}_p(u_1,u_1) \leq \lim_{k \to \infty} \mathcal{E}_p(u_k,u_k).
     \end{split}
     \end{align}
By the estimates in \eqref{2}-\eqref{3} and using Lemma \ref{ln-limit}, we get the required claim. Hence $u_1$ is a least energy solution to \eqref{main_problem}. Now, for $\beta>0$ sufficiently small, $v \in C_c^\infty(\Omega),$ using \eqref{assumption G(x,t)} with $\eps< \frac{-\mu}{p}$, we have $\inf\limits_{\XO} J_{L_{\Delta_p}} =J_{L_{\Delta_p}}(u_1) \leq J_{L_{\Delta_p}}(\beta v)$ and
\begin{align}
\label{u_1 non trivial}
\begin{split}
J_{L_{\Delta_p}}(\beta v) & =\frac{\beta^p}{p}\qform(v,v) -\frac{\mu \beta^p}{p^2}\into |v|^p(\ln |\beta v|^p-1) ~dx- \into G(x,\beta v) ~dx\\
&\leq \beta^p\bigg[\frac{\qform(u,u)}{p}+\l(\frac{\mu}{p^2} - \l(\frac{\mu}{p} + \eps \r)\ln \beta+a_2\r)\into |v|^p ~dx \\
& \qquad \qquad -\frac{\mu}{p} \into |v|^p \ln |v|~dx  +\eps \into |v|^p |\ln|v|| ~dx\bigg]<0.
\end{split}
\end{align}
 Thus, $u_1$ is nontrivial. Now, we show that $u \in L^\infty(\Omega)$. Note that for $\eps<-\mu,$ using \eqref{assumption g(x,t)}, we have
 \[
 \begin{split}
 \qform(u,v)&=\into g(x,u)v ~dx+\mu \into \ln|u|h(u)v ~dx\\
 &\leq a_1\into |u|^{p-1}v ~dx + (\eps+\mu)\into |u|^{p-1} \ln |u| v ~dx \\
 & \qquad \qquad - 2 \eps \int_{\{|u|<1\}} |u|^{p-1} \ln |u| v ~dx\\
 &\leq a_1\into |u|^{p-1}v ~dx+\frac{(\eps-\mu)}{e(p-1)} \into v ~dx.
 \end{split}
 \]
 Then, by \cite[Theorem 6.7]{Dyda-Jarohs-Sk}, we get $u \in L^\infty(\Omega).$ Furthermore, if $G(x,t) \leq G(x,|t|)$ for all $t \in \R$, a.e. $x \in \Omega$, the claim that $u_1$ does not change sign in $\Omega$ can be inferred with the same arguments as in Theorem \ref{Existence L1}. Now, to prove the uniqueness, let $u$ and $v$ be nontrivial, nonnegative and distinct solutions to \eqref{main_problem}. Let $\gamma_t(x):=\gamma(u(x),v(x),t):= \l((1-t)u^p+tv^p\r)^{\frac{1}{p}}, \ t \in [0,1].$
 Define a map $J_{L_{\Delta_p}}(\gamma_t):[0,1] \to \R$ as
 \[
 \begin{split}
      J_{L_{\Delta_p}}(\gamma_{t}) & = \frac{1}{p}\qform(\gamma_{t}, \gamma_{t}) -\into G(x, \gamma_{t}) ~dx-\frac{\mu}{p^2}\into |\gamma_{t}|^p(\ln |\gamma_{t}|^p-1) ~dx\\
     &=: \theta_1(t)+\theta_2(t).
 \end{split}
 \]

 Note that
\[\frac{d\theta_2}{dt} = -\into \frac{(v^p-u^p)}{p}\l[\gamma_{t}^{1-p} g(x,\gamma_{t})+\mu\ln(\gamma_{t})\r] ~dx,\]

 \[\frac{d^2\theta_2}{dt^2}= -\into \frac{\gamma_{t}^{-p}(v^p-u^p)^2}{p^2}\l[g'(x,\gamma_{t})\gamma_{t}^{2-p}+(1-p)g(x,\gamma_t)\gamma_{t}^{-p}+\mu\r] ~dx.\]
Now, by \eqref{assump g6}, we have $\frac{d^2\theta_2}{dt^2}\geq 0,$ which implies, the map $t\mapsto \theta_2(t)$ is convex. Let $t_1,t_2,\eta\in [0,1].$ For $\xi =(\xi_1, \xi_2) $ and $\Theta=(\Theta_1,\Theta_2)$, using the following Minkowski's inequality
 \begin{equation}
 \label{Minkowski-ineq}
 \l| \l(|\xi_1|^p+|\xi_2|^p\r)^{\frac{1}{p}}-\l( |\Theta_1|^p+|\Theta_2|^p\r)^{\frac{1}{p}}\r|^p \leq |\xi_1-\Theta_1|^p+|\xi_2-\Theta_2|^p,
 \end{equation}
 with \[\xi=((1-\eta)^{\frac{1}{p}}\gamma_{t_1}(x), \ \eta^{\frac{1}{p}}\gamma_{t_2}(x)) \quad  \text{and} \quad \Theta  =((1-\eta)^{\frac{1}{p}}\gamma_{t_1}(y), \ \eta^{\frac{1}{p}}\gamma_{t_2}(y)),\] we get
 \begin{align}
 \label{ineq}
 \begin{split}
 &\l|\gamma_{(1-\eta)t_1+\eta t_2}(x)-\gamma_{(1-\eta)t_1+\eta t_2}(y)\r|^p\\
&= \l[\l((1-\eta)\gamma_{t_1}^p(x)+\eta \gamma_{t_2}^p(x) \r)^{\frac{1}{p}} -\l((1-\eta)\gamma_{t_1}^p(y)+\eta \gamma_{t_2}^p(y) \r)^{\frac{1}{p}}\r]^p\\
     &\leq  (1-\eta) \l|\gamma_{t_1}(x)-\gamma_{t_1}(y)\r|^{p} + \eta \l|\gamma_{t_2}(x)-\gamma_{t_2}(y)\r|^{p}.
 \end{split}
 \end{align}
Using \eqref{ineq}, we get
\begin{equation}\label{convex-1}
    \mathcal{E}_p(\gamma_{(1-\eta)t_1+\eta t_2},\gamma_{(1-\eta)t_1+\eta t_2}) \leq (1-\eta)\mathcal{E}_p(\gamma_{t_1}, \gamma_{t_1})+\eta \mathcal{E}_p(\gamma_{t_2}, \gamma_{t_2}).
\end{equation}
Note that \[|\gamma_{(1-\eta)t_1+\eta t_2}|^p= (1-\eta)|\gamma_{t_1}|^p+\eta |\gamma_{t_2}|^p.\]
Using this and \eqref{ineq}, we get
\begin{equation}\label{convex-2}
    \mathcal{F}_p(\gamma_{(1-\eta)t_1+\eta t_2},\gamma_{(1-\eta)t_1+\eta t_2}) \leq (1-\eta)\mathcal{F}_p(\gamma_{t_1}, \gamma_{t_1})+\eta \mathcal{F}_p(\gamma_{t_2}, \gamma_{t_2}).
\end{equation}
\[\]
Finally, combining the estimates \eqref{convex-1}-\eqref{convex-2} in the definition of $\qform(\cdot,\cdot)$, we deduce that $\theta_1(t)$ is convex, which together with the convexity of $\theta_2$ yields, $J_{L_{\Delta_p}}(\gamma_t)$ is convex. Therefore, applying \cite[Theorem 1.1]{Foldes-Saldana}, we get that the solution to \eqref{main_problem} is unique.
\qed
\section{Small order asymptotics of the \texorpdfstring{$p$}{p}-fractional Laplacian problem}
\label{{p}-fractional Laplacian problem}
In this section, we prove that a sequence of least energy solutions to the fractional problem \eqref{frac_pblm} converges to a solution to the limiting problem \eqref{lim_pblm}. We consider two cases of a subcritical growth exponent function such that the growth of the nonlinearity in the given problem is superlinear and sublinear respectively. Note that the problem \eqref{lim_pblm} is equivalent to \eqref{main_problem} with $g(x,u)= a'(0,x)h(u)$. Therefore, for this case \[J_{\logP}(u):= \frac{1}{p} \qform(u,u)-\frac{1}{p}\into a'(0,x)|u|^p ~dx-\frac{\mu}{p^2}\into |u|^p(\ln|u|^p-1) ~dx.\]
Recall the following result:
\begin{lemma}
\cite[Lemma 3.1]{Santamaria-Saldana}
    \label{lemma3.1}
 If $b\neq 0,$ then
 \[\lim_{s \to 0^+} \l(1+sb+o(s)\r)^{\frac{1}{s}}=e^b= \lim_{s \to 0^+} (1+sb)^{\frac{1}{s}}.\]
\end{lemma}
\subsection{The superlinear case}
In this subsection, we consider the function $q$ satisfying \eqref{assump q} and
\[
q(s)\in \l(p,\frac{Np}{N-sp}\r) \quad \text{and} \quad \mu \in \l(0,\frac{p^2}{N}\r).\]
\begin{lemma}
\label{t-hash}
For any $\ph\in \WO\setminus\{0\}$ satisfying $\into a(s,x)|\ph|^{q(s)}~dx>0$, \[t^{\#}_{\ph,s}:= {\l(\frac{\nWp^p}{\into a(s,x)|\ph|^{q(s)}~dx}\r)}^{\frac{1}{q(s)-p}}\] is the unique point of maximum of the map $\mathbb{R}^+ \ni t \mapsto \psi(t):=J_{s,p}(t\phi)$ such that $t^{\#}_{\ph,s} \ph \in \mathcal{N}_{s,p}$, where $J_{s,p}$ is defined in \eqref{energy-frac}.
\end{lemma}
\begin{proof}
   Differentiating twice with respect to $t$, we obtain
    \[
    \psi'(t)=t^{p-1}\nWp^p-t^{q(s)-1}\into a(s,x)|\ph|^{q(s)}~dx
    \]
    and
    \[
    \psi''(t)=(p-1)t^{p-2}\nWp^p-(q(s)-1)t^{q(s)-2}\into a(s,x)|\ph|^{q(s)}~dx.
    \]
    Now,
    \[\psi'(t)=0 \quad \text{implies} \quad  t= t^{\#}_{\ph,s}:={\l(\frac{\nWp^p}{\into a(s,x)|\ph|^{q(s)}~dx}\r)}^{\frac{1}{q(s)-p}}.\]
    Since $p< q(s)$ and $t^{\#}_{\ph,s}$ is a critical point of $\psi$, we have $\psi''(t^{\#}_{\ph,s})<0$ and  $t^{\#}_{\ph,s}\ph \in \mathcal{N}_{s,p}.$ Therefore, we can conclude that $ t^{\#}_{\ph,s}$ is the unique point of maxima of $\psi.$
\end{proof}

\begin{lemma}
\label{u-0-nontriv}
Let $a$ satisfy \eqref{assump2-a}, \eqref{assump4-a} and $u_s \in \mathcal{N}_{s,p}$. Then,
\begin{equation}
\label{lowbd_u}
\nWs\geq {[k(N,s,p)]}^{\frac{1}{p-q(s)}},
\end{equation}
where
$k(N,s,p):=\|a\|_{L^{\nu}(\Omega)} |\Omega|^{\frac{1}{\nu'}-\frac{q(s)}{p^\ast_s}} {[A(N,p,s)]}^{\frac{q(s)}{p}}$, $\nu':=\nu'(s)$ denotes the H\"older conjugate of $\nu.$
\end{lemma}
\begin{proof}
 By H\"older inequality, \eqref{assump2-a} and \eqref{sob_ineq}, we have
 \begin{align}
 \label{Holder-sob}
 \begin{split}
    \into a(s,x)|u_s|^{q(s)}~dx & \leq \|a\|_{L^{\nu}(\Omega)}\l(\into |u|^{q(s)\nu'}~dx\r)^{\frac{1}{\nu'}}\\
    & \leq \|a\|_{L^{\nu}(\Omega)} |\Omega|^{\frac{1}{\nu'}-\frac{q(s)}{p^\ast_s}}\|u_s\|^{q(s)}_{L^{p^\ast_s}(\Omega)}\\
     & \leq \|a\|_{L^{\nu}(\Omega)} |\Omega|^{\frac{1}{\nu'}-\frac{q(s)}{p^\ast_s}} {[A(N,p,s)]}^{\frac{q(s)}{p}} \nWs^{q(s)}\\
     &=:k(N,s,p) \nWs^{q(s)}.
 \end{split}
 \end{align}
 Since $u_s \in \mathcal{N}_{s,p},$ we have
 \[
    0 = J'_{s,p}(u_s) \geq \nWs^p \l(1- k(N,s,p)\nWs^{q(s)-p}\r),
 \]
which implies
\[\nWs\geq {[k(N,s,p)]}^{\frac{1}{p-q(s)}}.\]
\end{proof}
\begin{remark}
\label{k>M}
Note that by \eqref{nu}, $\sup_{s \in (0, \frac{1}{2})}|\Omega|^{\frac{1}{s\nu}}<\infty$. Furthermore, by the Taylor's expansion of $A(N,p, s)$ about $s=0$ and Lemma \ref{lemma3.1}, we get
    \[
    \begin{split}
        \lim_{s \to 0^+} \l({[A(N,p,s)]}^{\frac{q(s)}{p}}\r)^{\frac{1}{p-q(s)}}
        &= \exp\bigg[\lim_{s \to 0^+} \l(\frac{q(s)}{p}\r)\lim_{s \to 0^+} \l(\frac{s}{p-q(s)}\r) \\
        & \qquad \qquad\lim_{s \to 0^+} \l(\ln[1+sk_0+o(s)]^{\frac{1}{s}}\r)\bigg] =  e^{\frac{-k_0}{\mu}}
    \end{split}
    \]
and using \eqref{assump4-a}, we get ${[k(N,s,p)]}^{\frac{1}{p-q(s)}}\geq M$
for some constant $M>0$ (independent of $s$).
\end{remark}
\begin{remark}
\label{inf-exists}
Note that on using Lemma \ref{u-0-nontriv}, for $u\in \mathcal{N}_{s,p},$ we get $J_{s,p}(u)=
 \l(\frac{1}{p}-\frac{1}{q(s)}\r)\nW^p$ is bounded below. Thus, $\inf_{\mathcal{N}_{s,p}} J_{s,p}$ exists.
 \end{remark}
\begin{theorem}
    \label{exist-frac}
Let $a: [0, \frac{1}{2}] \times \Omega \to \R$ satisfies \eqref{assump2-a}, \eqref{assump4-a}. Then, there exists a nontrivial least energy weak solution $u_s$ to the problem
\[ (-\Delta_p)^s u_s = a(s,x)|u_s|^{q(s)-2}u_s  \ \text{in}  \ \Omega, \qquad u_s=0 \ \text{in} \ \R^N \setminus \Omega.\]
 Moreover, if $a(s, \cdot)$ is non-negative for all $s \in [0, \frac{1}{2}]$, all least energy solutions are of constant sign in $\Omega.$
\end{theorem}
\begin{proof}
It is easy to see that $\mathcal{N}_{s,p}$ is a $C^1$ manifold, so we can apply the critical point theory on $\mathcal{N}_{s,p}$ to get critical points of $J_{s,p}.$
    Let $T : \WO\setminus\{0\} \to \R$ be given by
    \[T(u) = \nW^p\ - \into a(s,x)|u|^{q(s)} ~dx
    \]
 and
 \[(T'(u), u)_{\WO}= p\nW^p-{q(s)}\into a(s,x)|u|^{q(s)} ~dx.\]
By Ekeland's variational principle and Remark \ref{inf-exists}, there exists a minimizing sequence $(u_n)_{n\in \N} \subset \mathcal{N}_{s,p}$, $(\zeta_n)_{n\in \N}\subset \R$ such that  for all $n\in \N$,
\begin{equation}\label{Ekland:conseq}
    0\leq J_{s,p}(u_n)-\inf_{\mathcal{N}_{s,p}} J_{s,p} \leq \frac{1}{n^2}, \quad \|J_{s,p}'(u_n)-\zeta_n T'(u_n)\|_{\mathcal{L}(\WO,\R)}\leq \frac{1}{n}.
\end{equation}
From above, it is easy to see that \[0\leq \l[J_{s,p}(u_n)-\inf_{\mathcal{N}_{s,p}} J_{s,p}\r]= \l(\frac{1}{p}-\frac{1}{q(s)}\r)\|u_n\|_{\WO}^p-\inf_{\mathcal{N}_{s,p}} J_{s,p}\leq \frac{1}{n^2}.\]
Thus, by the above estimate and Lemma \ref{u-0-nontriv}, there exists a constant $C>1$ (independent of $n$) such that $\frac{1}{C}\leq \|u_n\|_{\WO} \leq C $ for all $n \in \mathbb{N}.$ Therefore, there exists $u_s\in \WO\setminus\{0\}$ such that $u_n \rightharpoonup u_s$ in $\WO$ and $u_n \to u_s$ in $L^{q(s)}(\Omega)$ and a.e. in $\R^N$. By \eqref{Ekland:conseq}, we have
\[
\begin{split}
 o(1) &= \frac{1}{\|u_n\|_{\WO}}\l((J'_{s,p}(u_n),u_n)_{\WO}-\zeta_n (T'(u_n),u_n)_{\WO} \r)\\
& = \zeta_n(q(s)-p)\|u_n\|_{\WO}^{p-1}.
\end{split}
\]
Therefore, $\zeta_n \to 0$ as $n \to \infty$ and by triangle inequality, we get
\[\|J'_{s,p}(u_n)\|_{\mathcal{L}(\WO,\R)} \to 0 \quad \text{as} \quad  n\to \infty.\]
Let $\ph \in \WO.$ Then, we have
\[
\begin{split}
    & 0 = \lim_{n\to \infty} (J'_{s,p}(u_n), \ph)_{\WO} \\&= \lim_{n\to \infty} \bigg(\frac{C_{N,s,p}}{2}\intr \frac{h(u_n(x)-u_n(y))(\ph(x)-\ph(y))}{|x-y|^{N+sp}} ~dx ~dy\\& \qquad \qquad \qquad-\into a(s,x)\abs{u_n}^{q(s)-2}u_n \ph ~dx \bigg)
    =:\lim_{n \to \infty} (I_1+I_2).
\end{split}
\]
For a given $w\in \WO,$ we denote
\[D_{s,p'}(w):= \frac{h(w(x)-w(y))}{|x-y|^{(N+sp)/p'}} \in L^{p'}(\R^{2N}),\]
\[\text{and} \quad D_{s,p}(w):= \frac{w(x)-w(y)}{|x-y|^{(N+sp)/p}} \in L^p(\R^{2N}). \]
By the boundedness of the sequence $(u_n)_{n\in\N}$ in $\WO,$ we get that the sequence $(D_{s,p'}(u_n))_{n\in \N}$ is bounded in $L^{p'}(\R^{2N}).$ For $\ph \in \WO$, we define $T_\ph:L^{p'}(\R^{2N}) \to \R$ as
\[T_\ph(w)= \frac{C_{N,s,p}}{2}\intr w D_{s,p}(\ph) ~dx ~dy. \] Then, $T_\ph$ is a linear and bounded functional in $L^{p'}(\R^{2N}).$ Since $(D_{s,p'}(u_n))_{n\in \N}$ is bounded in $L^{p'}(\R^{2N})$ and $u_n \to u_s$ a.e. in $\mathbb{R}^N$, $D_{s,p'}(u_n) \rightharpoonup D_{s,p'}(u_s)$ in $L^{p'}(\R^{2N})$, upto a subsequence, {\it i.e.},
\[
\begin{split}
\lim_{n \to \infty} I_1 & = \lim_{n \to \infty} T_\ph(D_{s,p'}(u_n)) = T_\ph(D_{s,p'}(u_s)) \\
& = \frac{C_{N,s,p}}{2}\intr \frac{h(u_s(x)-u_s(y))(\ph(x)-\ph(y))}{|x-y|^{N+sp}} ~dx ~dy.
\end{split}
\] Using \eqref{nu}, we get $q(s)\nu'<p^\ast_s$. Now, applying Young's inequality and \cite[Lemma A.1]{Willem} with $U \in L^{q(s)\nu'}(\Omega)$ as a majorant, upto a subsequence, we obtain
\[
\begin{split}
|a(s,x)|u_n|^{q(s)-2}u_n \ph| & \leq C(q(s))|a(s,x)| \l(|u_n|^{q(s)} + |\ph|^{q(s)} \r)\\
&\leq C_1(q(s),\nu)\l(|a(s,x)|^{\nu}+|u_n|^{q(s)\nu'}+ |\ph|^{q(s)\nu'}\r)\\
&\leq C_1(q(s),\nu)\l(|a(s,x)|^{\nu}+|U|^{q(s)\nu'}+ |\ph|^{q(s)\nu'}\r) \in L^1(\Omega).
\end{split}
\]
Using these calculations, the fact that $u_n \to u_s$ in $L^{r}(\Omega)$, for $r \in (1,p^\ast_s)$, $u_n \to u_s$ a.e. in $\R^N$ and applying the Lebesgue dominated convergence theorem, we get
\[\lim_{n \to \infty} I_2=\into a(s,x)\abs{u_s}^{q(s)-2}u_s \ph ~dx.\]
Hence,
$u_s$ is a weak solution to the given problem. In particular, taking $\ph=u_s$ as a test function, we get $u_s \in \mathcal{N}_{s,p}$.
Using the fact that $u_n , \ u_s \in \mathcal{N}_{s,p}$ and a similar Young's inequality argument as above, we have
\[
\begin{split}
\inf_{\mathcal{N}_{s,p}} J_{s,p} = \lim_{n\to \infty} J_{s,p}(u_n) &= \lim_{n \to \infty} \l(\frac{1}{p}- \frac{1}{q(s)}\r) \into a(s,x)|u_n|^{q(s)} ~dx\\
&=\l(\frac{1}{p}- \frac{1}{q(s)}\r) \into a(s,x)|u_s|^{q(s)} ~dx = J_{s,p}(u_s)  \geq   \inf_{\mathcal{N}_{s,p}} J_{s,p}.
 \end{split}
 \]
 Thus, $u_s \in \mathcal{N}_{s,p}$ is a least energy weak solution to the problem \eqref{frac_pblm}. Furthermore, let $u_s$ be any least energy solution to the problem \eqref{frac_pblm}. Using \eqref{Minkowski-ineq}, with $\xi=(u(x),0)$ and $\Theta=(u(y),0)$, it is easy to see that  $\||u_s|\|_{\WO} \leq \nWs$. From this, the fact that $u_s \in \mathcal{N}_{s,p}$ and the definition of $t^{\#}_{|u_s|,s}$ in Lemma \ref{t-hash}, we get $t^{\#}_{|u_s|,s}\leq 1$. Using these calculations, we get
 \[
 \begin{split}
 J_{s,p}(u_s) &\leq J_{s,p}(t^{\#}_{|u_s|,s}|u_s|) \\
 &\leq (t^{\#}_{|u_s|,s})^p \l(\frac{1}{p} \nWs^p  -\frac{1}{q(s)} \into a(s,x)|u_s|^{q(s)}~dx\r)  \leq J_{s,p}(u_s),
 \end{split}\]
 yielding $t^{\#}_{|u_s|,s} =1$. Thus, $\abs{u_s} \in \mathcal{N}_{s,p}$ is a nonnegative least energy solution to \eqref{frac_pblm}. Now, by the strong maximum principle (see \cite[Theorem 1.2]{Pezzo-Quaas}), we get $|u_s|>0$ in $\Omega$. Thus, we can
conclude that $u_s$ is either strictly positive or strictly negative in $\Omega$.
 \end{proof}
\begin{lemma}
\label{ln-lim-integral}
    Let $(s_n)_{n \in \N}\subseteq (0,\frac{1}{2})$ be a sequence such that $\lim\limits_{n \to \infty} s_n = 0$ and \[\nu:=\nu(s) > \max\l\{\frac{p^\ast_s}{p^\ast_s-q(s)}, 1+\frac{(N-sp)p}{s(p^2-N\mu-\delta p^2)}\r\},\]
for all $s \in \l[0, \frac{1}{2}\r]$. Let $q$ satisfy \eqref{assump q} and $g: [0,\frac{1}{2}]\times\Omega \to \R$ be such that \[\sup_{n\in\N} \|g(s_n,\cdot)\|_{L^{\nu(s_n)}(\Omega)}<\infty , \quad \lim_{n \to \infty} g(s_n,\cdot)=g(0,\cdot) \ \text{ a.e. in} \  \Omega.\] Let $(u_n)_{n \in \N}$ be a sequence in $L^p(\Omega)$ such that $u_n \to u$ in $L^p(\Omega)$. For all $v \in C_c^\infty(\Omega)$, it holds that
\[\lim_{n \to \infty} \into  g(s_n,x) |u_n|^{q(s_n)-2} \ln |u_n| u_n v ~dx = \into g(0,x) |u|^{p-2} \ln |u| uv ~dx.\]
\end{lemma}

\begin{proof}
Note that $\sup_{t \in [0,1)} |t|^{q(s_n)-1}|\ln|t||<\infty$. Thus, for $|u_n| \leq 1,$ we have
    \[
\begin{split}
    |g(s_n,x) |u_n|^{q(s_n)-2} \ln |u_n| u_n v|
    & \leq c\|v\|_{L^\infty(\Omega)}|g(s_n,x)|,
\end{split}
    \]
    where $c$ is a positive constant. Moreover, in view of \eqref{assump q}, we choose $\tau$ such that $0< \tau < 1+p-q(s_n)$. Using Lemma \ref{log t>1} with the above choice of $\tau \in (0,1)$ and applying Young's inequality for $|u_n| >1$, we get
\[
\begin{split}
    |g(s_n,x) |u_n|^{q(s_n)-1} \ln |u_n| |v| & \leq \frac{\|v\|_{L^\infty(\Omega)}}{e \tau}|g(s_n,x)||u_n|^{(q(s_n)-1+\tau)}\\
    & \leq \frac{\|v\|_{L^\infty(\Omega)}}{e \tau}\bigg[|g(s_n,x)|^{\frac{p}{p-q(s_n)+1-\tau}} + |u_n|^{p}\bigg].
\end{split}
\]
Note that $\nu(s_n) \to +\infty$ and $\frac{p}{p-q(s_n)+1-\tau} \to \frac{p}{1-\tau}$ as $n \to \infty$. For a given $\eta>0$, there exists $n_0 \in \mathbb{N}$ and $M>0$ (independent of $n$) such that
\begin{equation}\label{limit-upper-est-1}
    \frac{p}{p-q(s_n)+1-\tau} < \frac{p}{1-\tau} + \eta < \tilde{p}:= \frac{p}{1-\tau} + 2\eta < \nu(s_n), \quad \|g(s_n,\cdot)\|_{L^{\tilde{p}}(\Omega)}^{\frac{p}{1-\tau} + \eta} \leq M
\end{equation}
and
\[
|g(s_n,x) |u_n|^{q(s_n)-1} \ln |u_n| |v| \leq \frac{\|v\|_{L^\infty(\Omega)}}{e \tau}\bigg[1+ |g(s_n,x)|^{\frac{p}{1-\tau} +\eta} + |u_n|^{p}\bigg] \ \text{for} \ n \geq n_0.
\]
Let $A \subseteq \Omega$ such that $|A| \leq \min\{\frac{\eps}{4}, \l(\frac{\eps}{2M}\r)^\frac{\tilde{p}}{\tilde{p} - \eta-\frac{p}{1-\tau}}\}.$ By Young's and H\"older inequality, we get
\[
\begin{split}
   \int_{A} |g(s_n,x)| ~dx &\leq |A| + \int_{A} |g(s_n,x)|^{\frac{p}{p-q(s_n)+1-\tau}} ~dx \leq 2 |A| + \int_A |g(s_n,x)|^{\frac{p}{1-\tau} +\eta} ~dx \\
   & \leq 2 |A| +  |A|^{1-\frac{\frac{p}{1-\tau} + \eta}{\tilde{p}}}\|g(s_n,\cdot)\|_{L^{\tilde{p}}(\Omega)}^{\frac{p}{1-\tau} + \eta} \leq \frac{\eps}{2} + |A|^{\frac{\tilde{p} - \eta - \frac{p}{1-\tau}}{\tilde{p}}} M < \eps.
\end{split}
\]
Finally, using the above calculations and Vitali convergence theorem, we get the result.
\end{proof}
\begin{lemma}
    \label{lim-t_hash}
 Let $a$ satisfy \eqref{assump1-a}. For every $\ph \in \WO\setminus\{0\},$ it holds that
\[\lim_{s \to 0^+} t^{\#}_{\ph,s}= t^{\ast}_{\ph},\]  where  $t^{\#}_{\ph,s}$ is defined in Lemma \ref{t-hash} and \[t^\ast_{\ph}:=\exp{\l(\frac{ \qform(\ph,\ph)-\into a'(0,x)|\ph|^p ~dx-\mu \into \ln |\ph|^p ~dx}{\mu \|\ph\|^p_{L^p(\Omega)}}\r)}.\] Moreover, $\sup_{s\in[0,\frac{1}{2}]} (t^{\#}_{\ph,s})^p <\infty.$
\end{lemma}

\begin{proof}
Using \eqref{assump1-a}, Lemma \ref{lemma3.1} and \cite[Lemma 7.2]{Dyda-Jarohs-Sk}, we get
 \[
\begin{split}
   &\lim_{s \to 0^+} t^{\#}_{\ph,s}=\lim_{s \to 0^+} {\l(\frac{\nWp^p}{\into a(s,x)|\ph|^{q(s)}~dx}\r)}^{\frac{1}{q(s)-p}}\\
   &= \lim_{s \to 0^+} {\l(\frac{\|\ph\|^p_{L^p(\Omega)}+s\qform(\ph,\ph)+o(s)}{\|\ph\|^p_{L^p(\Omega)}+s(\into a'(0,x)|\ph|^p ~dx+\mu \into \ln |\ph||\ph|^p ~dx)+o(s)}\r)}^{\frac{1}{q(s)-p}}\\
   &= {\l(\frac{\lim_{s\to 0^+} {\l(1+s\frac{\qform(\ph,\ph)}{\|\ph\|^p_{L^p(\Omega)}}+o(s)\r)}^{\frac{1}{s}}}{\lim_{s \to 0^+} {\l(1+s\l(\frac{\into a'(0,x)|\ph|^p ~dx+\mu \into  \ln |\ph||\ph|^p ~dx}{\|\ph\|^p_{L^p(\Omega)}} \r)+o(s)\r)}^{\frac{1}{s}}}\r)}^{\frac{1}{\mu}}\\
   &= \exp{\l(\frac{1}{\mu \|\ph\|^p_{L^p(\Omega)}}. \l(\qform(\ph,\ph)-\into a'(0,x)|\ph|^p ~dx-\mu \into \ln |\ph|^p ~dx\r)\r)}=t^\ast_{\ph}.
\end{split}
 \]
 Thus, we can continuously extend the map $s \mapsto t^{\#}_{\ph,s}$ over the entire interval $[0,\frac{1}{2}].$ Hence, $\sup_{s\in[0,\frac{1}{2}]} (t^{\#}_{\ph,s})^p$ is finite.
\end{proof}

\begin{lemma}
\label{u_s-bdd-W}
  Let $a$ satisfy \eqref{assump1-a}, \eqref{assump2-a} and $u_s \in \mathcal{N}_{s,p}$ be a least energy solution to the problem \eqref{frac_pblm}. Then, there exists a $C=C(p,\Omega,N)$ such that \[\nWs^p \leq C\quad  \text{for all} \  s\in(0,\frac{1}{2}).\]
\end{lemma}
\begin{proof}
Let $\ph \in C_c^\infty(\Omega)\setminus\{0\}.$ Using the fact that $u_s$ is a least energy solution, Lemmas \ref{t-hash} and \ref{lim-t_hash}, we have
    \begin{equation}
     \label{bd-nW-sup}
    \nWs^p= \inf_{v\in \mathcal{N}_{s,p}} \|v\|_{\WO}^p \leq (t^{\#}_{\ph,s})^p\nWp^p \leq \sup_{s\in (0,\frac{1}{2})} (t^{\#}_{\ph,s})^p\nWp^p < C.
    \end{equation}

\end{proof}
\begin{lemma}
\label{u_s-bdd-X}
    Let $s \in (0,\frac{1}{2})$, $a$ satisfies \eqref{assump1-a}, \eqref{assump3-a} and $u_s\in \mathcal{N}_{s,p}$ be such that $\nWs^p \leq C$ for some constant $C>0$ independent of $s$. Then, there exists a $C'>0$ such that
    \[
    \|u_s\|_{\XO}^p \leq C'.
    \]
\end{lemma}

\begin{proof}
Assume first that $\ph \in C^\infty_c(\Omega)\cap \mathcal{N}_{s,p}.$ By the embedding $\WO \hookrightarrow L^p(\Omega)$, we get
\begin{equation}
    \label{Lp}
     \|\ph\|_{L^p(\Omega)}^p\leq C_1
\end{equation}
where $C_1>0$ independent of $s$. Now, using \eqref{assump1-a}, $\ph \in C^\infty_c(\Omega)\cap \mathcal{N}_{s,p}$ and \cite[Lemma 7.2]{Dyda-Jarohs-Sk}, we obtain
\[
\begin{split}
   \qform(\ph,\ph)&= \lim_{s \to 0^+} \frac{\nWp^p-\|\ph\|_{L^p(\Omega)}^p}{s} =  \lim_{s \to 0^+} \frac{\into a(s,x)|\ph|^{q(s)} ~dx - \into |\ph|^{p} ~dx }{s}\\
   & =  \into a'(0,x)|\ph|^p ~dx+\mu \into |\ph|^p \ln|\ph| ~dx.
\end{split}
\]
Furthermore, by Theorem \ref{log-Sob-ineq}, \eqref{assump3-a} and \eqref{Lp}, we get
\begin{equation}
\label{deriv}
 \begin{split}
 & \qform(\ph,\ph) = \into a'(0,x)|\ph|^p ~dx +\mu \into |\ph|^p \ln|\ph| ~dx \\
    & \quad \leq \frac{ \mu N}{p^2}\qform(\ph,\ph))+ \mu \l(\ln (\|\ph\|_{L^p(\Omega)} + \frac{Nk_0}{p^2}+\frac{\|a'(0,\cdot)\|_{L^\infty(\Omega)}}{\mu}\r) \|\ph\|_{L^p(\Omega)}^p \\
   & \quad \leq  \frac{\mu N}{p^2}\qform(\ph,\ph)+\tilde{C},
 \end{split}
\end{equation}
 where $\tilde{C}>0$ is a constant independent of $s$. Now, using \eqref{Lp}, \eqref{deriv}, $\mu<\frac{p^2}{N}$ and Lemma \ref{diff-F_p}, we obtain
 \[\|\ph\|_{\XO}^p = \mathcal{E}_p(\ph,\ph) \leq C'.\]
This establishes our claim for $\ph \in C^\infty_c(\Omega) \cap \mathcal{N}_{s,p}.$ Note that for any $u \in \WO$ and $s \in (0,\frac{1}{2}),$ we have
   \[\|u\|_{\XO}^p=2 \mathcal{E}_p(u,u) \leq \frac{4C_{N,p}}{C_{N,s,p}}\nW^p = \l(\frac{2^{2-2s}\Gamma\l(\frac{N}{2}\r)\Gamma(1-s)}{s\Gamma\l(\frac{N+sp}{2}\r)}\r)\nW^p.\]
Finally, due to the above estimate and the density of $C^\infty_c(\Omega)$ in $W^{s,p}_0(\Omega),$ we get the required claim for any $u_s \in \mathcal{N}_{s,p}.$
\end{proof}
\noindent
\textit{Proof of Theorem \ref{asym-sup}: }
    The existence of the sequence of solutions $(u_s)$ to the problem \eqref{frac_pblm} is guaranteed by Theorem \ref{exist-frac}.
 Using Lemma \ref{u_s-bdd-X}, we deduce that the sequence $(u_s)$ is bounded in $\XO$. Now, by the reflexivity of $\XO$ and Lemma \ref{compct-embeding}, there exists a $u_0 \in \XO$ such that $u_s \rightharpoonup u_0$ in $\XO$ and $u_s \to u_0$ in $L^p(\Omega)$ (up to a subsequence).\\
 \textbf{Claim 1:} $u_0$ is a weak solution to \eqref{lim_pblm}. \\
 Let $v \in C_c^\infty(\Omega).$ Then,
 \[
 \begin{split}
 &\into u_s(|v|^{p-2} v + s \logP v+ o(s)) ~dx = \into u_s (-\Delta_p)^s v ~dx=\into a(s,x)|u_s|^{q(s)-2} u_s v ~dx  \\&= \into v \bigg(|u_s|^{p-2}u_s+ s \bigg(\int_0^1 \bigg[q'(s\theta)a(s\theta,x)|u_s|^{q(s\theta)-2}\ln(|u_s|) u_s\\& \qquad \qquad \qquad \qquad \qquad \qquad  \qquad \qquad \qquad + a'(s\theta,x)|u_s|^{q(s\theta)-2} u_s\bigg]~d\theta\bigg)+o(s)\bigg) ~dx.
 \end{split}
 \]
Since $u_s \rightharpoonup u_0$ in $\XO$, by \cite[Lemma 4.4]{Dyda-Jarohs-Sk} and arguments as in Theorem \ref{Existence L1}, we obtain
 \[\lim_{s \to 0^+} \into u_s(\logP v) ~dx = \lim_{s \to 0^+} \qform(u_s,v) = \qform(u_0,v). \]
By \eqref{assump3-a}, we have
 \[
\begin{split}
|a'(s\theta,x)|u_s|^{q(s\theta)-2}u_s v|&\leq \|a'(s\theta,\cdot)\|_{L^\infty(\Omega)}\|v\|_{L^\infty(\Omega)}|u_s|^{q(s\theta)-1}\\
& \leq c_0\|v\|_{L^\infty(\Omega)}|u_s|^{q(s\theta)-1} \leq c_0\|v\|_{L^\infty(\Omega)}(1+|u_s|^p),
\end{split}
\]
for $s$ small enough where $c_0:=\sup_{s\in [0,\frac{1}{2}]} \|a'(s,\cdot)\|_{L^\infty(\Omega)}<\infty$.
Thus, by $u_s \to u$ in $L^p(\Omega)$ and Lebesgue dominated convergence theorem, upto a subsequence, we get \[\lim_{s \to 0^+} \into a'(s\theta,x)|u_s|^{q(s\theta)-2}u_s v ~dx=\into a'(0,x)|u_0|^{p-2}u_0 v ~dx.\]
Now, by Lemma \ref{ln-lim-integral} and \eqref{assump1-a}, upto a subsequence, we have
\[
\lim_{s \to 0^+} \into \int_0^1 q'(s\theta) a(s\theta, x) |u_s|^{q(s\theta)-2} \ln|u_s| u_s v ~dx ~d\theta = \mu \into |u_0|^{p-2} \ln|u_0|u_0 v ~dx.
\]
 Using the density of $C_c^\infty(\Omega)$ in $\XO$ and collecting all the above estimates, we get
    \[\qform(u_0,\ph)= \into (a'(0,x)+\mu\ln|u_0|)|u_0|^{p-2}u_0 \ph ~dx \quad \text{for all} \ \ph \in \XO. \]
 Thus, $u_0$ is a weak solution to \eqref{lim_pblm}.\\
 \textbf{Claim 2:} $u_0$ is nontrivial.\\
Using \eqref{Holder-sob} and \eqref{lowbd_u}, we get
 \[{[k(N,s,p)]}^{\frac{p}{p-q(s)}} \leq \nWs^p \leq \|a\|_{L^{\nu}(\Omega)}\l(\into |u_s|^{q(s)\nu'}~dx\r)^{\frac{1}{\nu'}}.\]
Let $\lambda_s:= \frac{q(s)\nu'-p}{p^\ast_s-p}.$ Since $q(s)\nu'>p$ and $p^\ast_s>p$, for all $s\in (0,\frac{1}{2}),$ we get $0<\lambda_s<1.$
  Using \eqref{sob_ineq}, H\"older inequality and \eqref{bd-nW-sup}, we get
 \[
 \begin{split}
   \into & |u_s|^{q(s)\nu'} ~dx = \into |u_s|^{\bigg[\frac{p(p^\ast_s-q(s)\nu')}{p^\ast_s-p}+\frac{p^\ast_s(q(s)\nu'-p)}{p^\ast_s-p}\bigg]} ~dx =\|u_s\|_{L^p(\Omega)}^{p(1-\lambda_s)} \|u_s\|_{L^{p^\ast_s}(\Omega)}^{p^\ast_s \lambda_s}\\
   &\leq\|u_s\|_{L^p(\Omega)}^{p(1-\lambda_s)} [A(N,p,s)]^{\frac{p^\ast_s \lambda_s}{p}} \|u_s\|^{p^\ast_s \lambda_s}_{\WO} \leq\|u_s\|_{L^p(\Omega)}^{p(1-\lambda_s)} [A(N,p,s)]^{\frac{p^\ast_s \lambda_s}{p}} C^{p^\ast_s \lambda_s}.
 \end{split}
 \]
 Using the above, we get
 \begin{align}
 \label{lim-non-triv}
 \begin{split}
\frac{[k(N,s,p)]^{\frac{p\nu'}{p-q(s)}}}{\|a\|_{L^\nu(\Omega)}^{\nu'}} \leq \|u_s\|_{L^p(\Omega)}^{p(1-\lambda_s)} [A(N,p,s)]^{\frac{p^\ast_s \lambda_s}{p}} C^{p^\ast_s \lambda_s}.
\end{split}
\end{align}
Using the fact that $\nu'< \l[1+\frac{s(p^2-N\mu-\delta)}{(N-sp)p}\r],$ we obtain
\[
\begin{split}
    \lim_{s \to 0^+} \lambda_s&= \lim_{s \to 0^+} \frac{q(s)\nu'-p}{p^\ast_s-p} \\
    & <\lim_{s \to 0^+} \l[\l(\frac{q(s)-p}{s}\r)+\l(\frac{q(s)(p^2-N\mu-\delta p^2)}{(N-sp)p}\r)\r] \lim_{s \to 0^+} \l(\frac{s}{p^\ast_s-p}\r) \\
    & = \l[\mu+\l(\frac{p^2-N\mu-\delta p^2}{N}\r)\r]\frac{N}{p^2} =(1-\delta).
\end{split}
\]
Thus, $\lim_{s \to 0^+} \lambda_s \in (0,1).$
Taking limit $s \to 0^+$ both sides in \eqref{lim-non-triv}, using the above calculations, \eqref{assump1-a}, \eqref{assump2-a}, Remark \ref{k>M} and $\lim_{s \to 0} [A(N,p,s)]=1$, it follows
\[0< \l[\frac{M^p}{ C^{p(1-\delta)}}\r]^{\frac{1}{p\delta}}\leq \|u_0\|_{L^p(\Omega)}.\]
Thus, $u_0$ is nontrivial. \\
 \textbf{Claim 3:} $u_0$ has least energy.\\
In view of Lemma \ref{lim-t_hash}, it can be concluded that $\lim_{s \to 0^+} t^{\#}_{u_s,s}= t^{\ast}_{u_0}.$
 Using these calculations, Fatou's Lemma and the fact that $u_s \in \mathcal{N}_{s,p}$, $u_0 \in \mathcal{N}_{L_{\Delta_p}}$ and definitions of $t^\#_{u_s}, \ t^\ast_{u_0},$ we get
 \[
 \begin{split}
     \inf_{\mathcal{N}_{L_{\Delta_p}}} J_{L_{\Delta_p}} & = \frac{\mu}{p^2} \|u_0\|_{L^p(\Omega)} \leq \frac{\mu}{p^2} \liminf_{s \to 0^+}(\nWs^p)\\
     & \leq \lim_{s \to 0^+} \frac{1}{s}\l(\frac{1}{p}-\frac{1}{q(s)}\r)\nWs^p\\
     & = \lim_{s \to 0^+} \frac{1}{s}\l(\frac{1}{p}-\frac{1}{q(s)}\r) \| \ t^{\#}_{u_s,s}u_s\ \|^p_{\WO}= \frac{\mu}{p^2}\|t^{\ast}_{u_0}u_0\|^p_{L^p(\Omega)}\\& = \frac{\mu}{p^2}\|u_0\|^p_{L^p(\Omega)}= J_{L_{\Delta_p}}(u_0).
 \end{split}
 \]
 From above, we conclude that $u_0$ is a least energy solution to \eqref{lim_pblm}. This completes the proof of the theorem.
\qed
\subsection{The sublinear case}
Throughout this subsection, we consider the function $q$ satisfying \eqref{assump q} and \[q(s)\in (1,p) \quad \text{and} \quad \mu\in(-\infty,0).\]
 Let $\ph_{s,p}$ represent the first normalized eigenfunction of the fractional $p$-Laplacian corresponding to its first eigenvalue $\lambda_{s,p}.$
\begin{lemma}
\label{low-upp-bd-sub}
    Let $a$ satisfy \eqref{assump2-a} and $u_s$ be a positive least energy solution to \eqref{frac_pblm}. Then,
    \begin{equation}\label{upper-lower-est}
    \Theta_1(s) \leq \nWs^p \leq \Theta_2(s),
    \end{equation}
    where
     \[\Theta_1(s):=\frac{pq(s)}{(q(s)-p)} \  J_{s,p}\l(\frac{t_0 \ph_{s,p}}{2}\r)\]
     and
     \[
     \Theta_2(s):= \l({\|a\|_{L^\nu(\Omega)}|\Omega|^{\frac{1}{\nu'}-\frac{q(s)}{p^\ast_s}}[A(N,p,s)]^{\frac{q(s)}{p}}}\r)^{\frac{p}{p-q(s)}}.\]
\end{lemma}
\begin{proof}
    By the definition of $J_{s,p}$ in \eqref{energy-frac}, we get
    \[J_{s,p}(t\ph_{s,p})=0 \quad \text{implies} \quad t=t_0:= {\l(\frac{p\into a(s,x)|\ph_{s,p}|^{q(s)} ~dx}{q(s)\|\ph_{s,p}\|^p_{\WO}}\r)}^{\frac{1}{p-q(s)}} .\]
    Notice that $J_{s,p}(t\ph_{s,p})<0,$ if $t<t_0.$ Since $u_s$ is the least energy solution to \eqref{frac_pblm}, we have
    \[
    \begin{split}
    J_{s,p}(u_s) &= \l(\frac{1}{p}-\frac{1}{q(s)}\r) \nWs^p\leq J_{s,p}\l(\frac{t_0 \ph_{s,p}}{2}\r) < 0.
    \end{split}
    \]
Thus, by the fact that $q \in (1,p),$ we obtain
    \[\nWs^p \geq \frac{pq(s)}{(q(s)-p)} \  J_{s,p}\l(\frac{t_0 \ph_{s,p}}{2}\r) >0. \]
    Using H\"older inequality and \eqref{sob_ineq}, we have
    \[
    \begin{split}
        J_{s,p}(u_s)&= \frac{1}{p}\nWs^p- \frac{1}{q} \into a(s,x)|u_s|^{q(s)} ~dx\\
        & \geq \frac{1}{p} \nWs^p - \frac{1}{q }\|a\|_{L^{\nu}(\Omega)}\l(\into |u|^{q(s)\nu'}~dx\r)^{\frac{1}{\nu'}}\\
        &  \geq \frac{1}{p} \nWs^p - \frac{1}{q}\|a\|_{L^{\nu}(\Omega)} |\Omega|^{\frac{1}{\nu'}-\frac{q(s)}{p^\ast_s}} {[A(N,p,s)]}^{\frac{q(s)}{p}} \nWs^{q(s)}.
    \end{split}
    \]
A simple calculation yields \[t_1:= \l(\|a\|_{L^{\nu}(\Omega)} |\Omega|^{\frac{1}{\nu'}-\frac{q(s)}{p^\ast_s}} {[A(N,p,s)]}^{\frac{q(s)}{p}}\r)^{\frac{1}{p-q(s)}}\]is the point of minima of the function \[h(t):= \frac{1}{p}t^p-\l(\frac{\|a\|_{L^{\nu}(\Omega)} |\Omega|^{\frac{1}{\nu'}-\frac{q(s)}{p^\ast_s}} {[A(N,p,s)]}^{\frac{q(s)}{p}}}{q(s)}\r)t^{q(s)}.\] Hence,
  \[
  \begin{split}
  J_{s,p}(u_s)&= \l(\frac{1}{p}-\frac{1}{q(s)}\r) \nWs^p \\&\geq \l(\frac{1}{p}-\frac{1}{q(s)}\r)\l(\|a\|_{L^{\nu}(\Omega)} |\Omega|^{\frac{1}{\nu'}-\frac{q(s)}{p^\ast_s}} {[A(N,p,s)]}^{\frac{q(s)}{p}}\r)^{\frac{p}{p-q(s)}},
  \end{split}
  \]
along with the fact that $1<q(s)<p$, gives us the desired upper bound $\Theta_2(s).$
\end{proof}
\begin{lemma}
\label{lim-lowbd}
Let $a$ satisfy \eqref{assump1-a}, \eqref{assump2-a} and \eqref{assump4-a}, it holds that
\begin{enumerate}
 \item[\textnormal{(i)}]    $\lim\limits_{s\to 0^+} \Theta_1(s) =\kappa_1:=\frac{ B^p (p\ln 2)}{2^p},$
    where \[B:= \exp{\bigg[\frac{1}{p}+\frac{1}{\mu}+\frac{\lambda_{1,L}}{\mu}-\frac{1}{\mu}\l(\into a'(0,x) |\ph_{1,L}|^p ~dx+\mu \into |\ph_{1,L}|^p \ln |\ph_{1,L}| ~dx\r)\bigg]},\]
   \item[\textnormal{(ii)}] $\lim\limits_{s \to 0^+} \Theta_2(s)<\infty,$
    \end{enumerate}
     where $\Theta_1(s)$, $\Theta_2(s)$ are defined in Lemma \ref{low-upp-bd-sub} and $\ph_{1,L}$ represents the first normalized eigenfunction of $\logP$ corresponding to its first eigenvalue $\lambda_{1,L}.$
\end{lemma}
\begin{proof}
    We first calculate $\lim\limits_{s \to 0^+} t_0,$ where $t_0$ is defined in Lemma \ref{low-upp-bd-sub}. Using \eqref{assump1-a} and Lemma \ref{lemma3.1}, we get
    \[
    \begin{split}
        &\lim_{s\to 0^+} t_0= \lim_{s \to 0^+} {\l(\frac{p\into a(s,x)|\ph_{s,p}|^{q(s)} ~dx}{q(s)\|\ph_{s,p}\|^p_{\WO}}\r)}^{\frac{1}{p-q(s)}}\\
        &= \lim_{s \to 0^+} {\l(\frac{p\into a(s,x)|\ph_{s,p}|^{q(s)} ~dx}{q(s) \lambda_{s,p}\|\ph_{s,p}\|^p_{L^p(\Omega)}}\r)}^{\frac{1}{p-q(s)}}\\
        &=  {\frac{\lim\limits_{s\to 0^+}{\bigg[1+s\l(\into a'(0,x) |\ph_{1,L}|^p ~dx+\mu \into |\ph_{1,L}|^p \ln |\ph_{1,L}| ~dx\r)+o(s)\bigg]}^{\frac{-1}{s\mu}}}{\lim\limits_{s\to 0^+}{\bigg[\l(1+s\frac{\mu}{p}+o(s)\r)\l(1+s\lambda_{1,L}+o(s)\r)\l(1+s\|\ph_{1,L}\|^p_{L^p(\Omega)}+o(s)\r)\bigg]}^{\frac{-1}{s\mu}}}}\\
        &= \exp{\bigg[\frac{1}{p}+\frac{1}{\mu}+\frac{\lambda_{1,L}}{\mu}-\frac{1}{\mu}\l(\into a'(0,x) |\ph_{1,L}|^p ~dx +\mu  \into |\ph_{1,L}|^p \ln |\ph_{1,L}| ~dx\r)\bigg]}.
    \end{split}
    \]
    Using above calculations and \cite[Lemma 7.4]{Dyda-Jarohs-Sk}, we get
    \[
    \begin{split}
    \lim_{s\to 0^+} \Theta_1(s) &= \lim_{s\to 0^+} \frac{t_0^p \l(\frac{1}{p 2^p}-\frac{t_0^{q(s)-p}}{q 2^{q(s)}}\r)}{\l(\frac{1}{p}-\frac{1}{q(s)}\r)}\lambda_{s,p} \|\ph_{s,p}\|^p_{L^p(\Omega)}= \frac{B^p (p\ln 2)}{2^p}.
    \end{split}
    \]
    This proves (i). Now, we proceed to prove (ii). Note that in view of Lemma \ref{lemma3.1} and Taylor's series expansion about $s=0,$ we have
   \[\lim_{s \to 0^+} [A(N,p,s)]^\frac{q(s)}{p-q(s)} = \lim_{s \to 0^+} \l[\l(1+sk_0+o(s)\r)^\frac{1}{s}\r]^{\frac{sq(s)}{p-q(s)}} = \exp\l(\frac{-pk_0 }{\mu}\r)\]
where $k_0 = \partial_s A(N,p,0).$  By \eqref{assump4-a}, we obtain
 \[c_1^p \leq \lim_{s \to 0^+} \|a(s,\cdot)\|_{L^\nu(\Omega)}^\frac{p}{p-q(s)}\leq c_2^p.\]
  When $|\Omega|\geq 1$, using Taylor's series expansion about $s=0$ and Lemma \ref{lemma3.1}, we infer
   \[
   \begin{split}
   \lim_{s\to 0^+} \l[|\Omega|^{\frac{1}{\nu'}-\frac{q(s)}{p^\ast_s}}\r]^{\frac{p}{p-q(s)}}&=\lim_{s\to 0^+} \l(|\Omega|^{\frac{p^\ast_s-q(s)}{p}-\frac{1}{\nu}}\r)^{\frac{-p}{s\mu}}\leq \lim_{s\to 0^+} \l(|\Omega|^{\frac{p^\ast_s-q(s)}{p}}\r)^{\frac{-p}{s\mu}}\\
   & = \lim_{s \to 0^+} \l[1+s\l(\frac{p}{N}-\frac{\mu}{p}\r)\ln|\Omega|+o(s)\r]^{\frac{-p}{s\mu}}=|\Omega|^{1-\frac{p^2}{N\mu}}.
   \end{split}
   \]
 When $|\Omega|<1$, using $\nu>1+\frac{(N-sp)p}{s(p^2-N\mu-\delta p^2)}$, Taylor's series expansion about $s=0$ and Lemma \ref{lemma3.1}, we deduce
    \[
    \begin{split}
    \lim_{s \to 0^+} \l[|\Omega|^{\frac{1}{\nu'}-\frac{q(s)}{p^\ast_s}}\r]^{\frac{p}{p-q(s)}} &=\lim_{s\to 0^+} \l(|\Omega|^{\frac{p^\ast_s-q(s)}{p}-\frac{1}{\nu}}\r)^{\frac{-p}{s\mu}}\\
    &=\lim_{s \to 0^+} \l[1+s\l(\frac{\delta p \ln|\Omega|}{N}\r)+o(s)\r]^{\frac{-p}{s\mu}}=|\Omega|^{-\frac{\delta p^2}{N\mu}}.
    \end{split}
    \]
    From above, we get
    \[
    \begin{split}
    \lim_{s \to 0^+} \l[|\Omega|^{\frac{1}{\nu'}-\frac{q(s)}{p^\ast_s}}\r]^{\frac{p}{p-q(s)}}\leq \begin{cases}
       |\Omega|^{1-\frac{p^2}{N\mu}} \quad &\text{if} \ |\Omega|\geq 1,\\
       |\Omega|^{\frac{-\delta p^2}{N\mu}} \quad &\text{if} \ |\Omega|<1.
    \end{cases}
    \end{split}\]
    Thus,
    \[\lim_{s \to 0^+} \l({\|a\|_{L^\nu(\Omega)}|\Omega|^{\frac{1}{\nu'}-\frac{q(s)}{p^\ast_s}}[A(N,p,s)]^{\frac{q(s)}{p}}}\r)^{\frac{p}{p-q(s)}}<\infty.\] This completes the proof of (ii).
\end{proof}

\begin{corollary}
\label{L^p u_s}
    Let $a$ satisfy \eqref{assump1-a}, \eqref{assump2-a} and \eqref{assump4-a}. Let $u_s$ be a positive least energy solution to the problem \eqref{frac_pblm}. Then, there exists constants, $\kappa_1>0$ (defined in Lemma \ref{lim-lowbd}) and $\kappa_2>0$ (independent of $s$) such that
    \[\kappa_1 \leq \|u_s\|_{L^p(\Omega)}^p \leq \kappa_2. \]
\end{corollary}

\begin{proof}
The proof follows by passing limits $s \to 0^+$ in \eqref{upper-lower-est} and using Lemmas \ref{low-upp-bd-sub}, \ref{lim-lowbd}.
\end{proof}
\begin{lemma}
\label{u_s-bdd-sub}
Let $a$ satisfy \eqref{assump1-a}, \eqref{assump2-a}, \eqref{assump3-a} and \eqref{assump4-a}. Let $u_s$ be a positive least energy solution to the problem
\eqref{frac_pblm}. Then, there exists a constant $C=C(N,p,\Omega)$ such that
\[\|u_s\|_{\XO}^p\leq C.\]
\end{lemma}
\begin{proof}
Let $\ph \in C_c^\infty(\Omega)$ such that $\nWp^p=\into a(s,x)|\ph|^{q(s)} ~dx$. By the mean value theorem, we get
   \[
    \begin{split}
        &\frac{\nWp^p-\|\ph\|_{L^p(\Omega)}^p}{s}
        = \frac{\into a(s,x)|\ph|^{q(s)} ~dx-\|\ph\|^p_{L^p(\Omega)}}{s}\\
        &= \into \int_0^1 a'(s\theta,x)|\ph|^{q(s\theta)} ~d\theta ~dx + \into \int_0^1 q'(s\theta)a(s\theta,x)|\ph|^{q(s\theta)} \ln|\ph| ~d\theta ~dx\\
        &=:I_1+I_2.
    \end{split}
    \]
    Using \eqref{assump3-a} and the fact that $\WO \hookrightarrow L^{q(s)}(\Omega)$, we have
    \[\lim_{s \to 0^+} I_1 = \into a'(0,x)|\ph|^p ~dx.\]
    Now, by the fact that $q(s) \in (1,p)$, $q'(s)<0$ in a neighborhood of $0$ and \eqref{assump4-a}, we get
    \[
    \begin{split}
        I_2&= \into \int_0^1 q'(s\theta)a(s\theta,x)|\ph|^{q(s\theta)} \ln|\ph| ~d\theta ~dx\\
       &= \l(\int_{\{|\ph|<1\}}+\int_{\{|\ph|\geq 1\}}\r) \int_0^1 q'(s\theta)a(s\theta,x)|\ph|^{q(s\theta)} \ln|\ph| ~d\theta ~dx\\
       &\leq \int_{\{|\ph|<1\}} \int_0^1 q'(s\theta)a(s\theta,x)|\ph|^{q(s\theta)} \ln|\ph| ~d\theta ~dx \\
       &\leq \frac{-|\Omega|^{\frac{1}{\nu'}}}{e} \int_0^1 q'(s \theta)\|a(s\theta,\cdot)\|_{L^{\nu}(\Omega)} ~d\theta \leq  \frac{-|\Omega|^{\frac{1}{\nu'}}}{e} \int_0^1 q'(s \theta) c_2^{q(s\theta)-p} ~d\theta.
    \end{split}
    \]
    Thus, by \eqref{nu}, we have
    \[ \lim_{s \to 0^+} I_2 \leq \frac{-\mu |\Omega|}{e}  .\]
    Using the above calculations, \eqref{assump3-a}, Corollary \ref{L^p u_s} and \cite[Lemma 7.2]{Dyda-Jarohs-Sk}, we obtain
    \[
    \begin{split}
    \qform(\ph,\ph)&=\lim_{s \to 0^+}  \l(\frac{\nWp^p-\|\ph\|_{L^p(\Omega)}^p}{s}\r) \leq \into a'(0,x)|\ph|^p ~dx-\frac{\mu |\Omega|}{e}\\
    & \leq \|a'(0,x)\|_{L^\infty(\Omega)}\|\ph\|_{L^p(\Omega)}^p -\frac{\mu |\Omega|}{e} \leq c-\frac{\mu |\Omega|}{e}.
    \end{split}
    \]
    Using Lemma \ref{diff-F_p}, for a constant $\tilde{c}>0$, we get
\[
\begin{split}
\qform(\ph,\ph)&\geq \|\ph\|_{\XO}^p-\tilde{c}\|\ph\|_{L^p(\Omega)}^p.
\end{split}
\]
Using the above and Corollary \ref{L^p u_s}, we infer
\[\|\ph\|_{\XO}^p \leq c-\frac{\mu|\Omega|}{e}+\tilde{c}\|\ph\|_{L^p(\Omega)} \leq \frac{\l(ec-\mu|\Omega|+e\kappa_2\tilde{c}\r)}{e}=:C.\]
Finally, by density of $C_c^\infty(\Omega)$ in $\WO$, we obtain the required assertion.
\end{proof}
\noindent \textit{Proof of Theorem \ref{asym-sub}:}
By Lemma \ref{u_s-bdd-sub}, we get that $(u_s)$ is bounded in $\XO$. Further, by the reflexivity of $\XO$ and Lemma \ref{compct-embeding}, there exists a $u_1 \in \XO$ such that (up to a subsequence) $u_s \rightharpoonup u_1$ in $\XO,$ $u_s \to u_1$ in $L^p(\Omega),$ as $s\to 0^+$.\\
 Now, we show that $u_1$ is a weak solution to \eqref{lim_pblm}.
 Let $v \in C_c^\infty(\Omega).$ Then,
  \[
 \begin{split}
 &\into u_s(|v|^{p-2} v + s \logP v+ o(s)) ~dx = \into u_s (-\Delta_p)^s v ~dx=\into a(s,x)|u_s|^{q(s)-2} u_s v ~dx  \\&= \into v \bigg(|u_s|^{p-2}u_s+ s \bigg(\int_0^1 \bigg[q'(s\theta)a(s\theta,x)|u_s|^{q(s\theta)-2}\ln(|u_s|) u_s\\& \qquad \qquad \qquad \qquad \qquad \qquad  \qquad \qquad \qquad + a'(s\theta,x)|u_s|^{q(s\theta)-2} u_s\bigg]~d\theta\bigg)+o(s)\bigg) ~dx.
 \end{split}
 \]
Repeating arguments similar to \textbf{Claim 1} in Theorem \ref{asym-sup}, we can conclude that
 \[\qform(u_1,\ph) = \into (a'(0,x)+\mu\ln|u_1|)|u_1|^{p-2}u_1 \ph ~dx \quad \text{for all} \ \ph \in \XO.\]
 Hence, $u_1$ is a weak solution to \eqref{lim_pblm}. By Corollary \ref{L^p u_s},
  we have
 \[\kappa_1\leq \lim_{s\to 0^+} \nWs^p=\|u_1\|_{L^p(\Omega)}^p.\]
 Therefore, $u_1$ is nontrivial.
 It remains to show that $u_1$ has least energy.
At first, note that for any $\ph \in C^\infty_c(\Omega)\setminus\{0\},$ using \cite[Lemma 7.2]{Dyda-Jarohs-Sk}, we have
 \begin{equation}
 \label{J/s}
 \begin{split}
     \lim_{s \to 0^+} \frac{J_{s,p}(\ph)}{s}&= \lim_{s\to 0^+} \frac{1}{s}\l( \frac{\nWp^p}{p}-\frac{\into a(s,x)|\ph|^{q(s)} ~dx}{q(s)}\r)\\
     &= \lim_{s\to 0^+} \bigg[\l(\frac{\nWp^p\l(\frac{1}{p}-\frac{1}{q(s)}\r)}{s}\r)+\l(\frac{\l(\nWp^p-\|\ph\|_{L^p(\Omega)}^p\r)}{sq(s)}\r)\\
     & \qquad \qquad + \l(\frac{\|\ph\|_{L^p(\Omega)}^p-\into a(s,x)|\ph|^{q(s)} ~dx}{sq(s)}\r)\bigg]\\
     &= \frac{\mu}{p^2}\|\ph\|_{L^p(\Omega)}^p + \frac{1}{p}\qform(\ph,\ph)\\& \qquad \qquad \qquad-\frac{1}{p}\l(\into a'(0,x)|\ph|^p ~dx+ \mu\into |\ph|^p \ln |\ph| ~dx\r).
 \end{split}
 \end{equation}
It is easy to deduce that
  \[
  \begin{split}
      \frac{\mu}{p^2} \lim_{s \to 0^+} \nWs^p&= \frac{\mu}{p^2}\lim_{s \to 0^+} \into a(s,x)|u_s|^{q(s)} ~dx\\
      &= \frac{\mu}{p^2} \|u_1\|_{L^p(\Omega)}^p = J_{L_{\Delta_p}}(u_1).
  \end{split}
  \]
  By Theorem \ref{Existence L2}, there exists, $v_0 \in \XO$ such that $$J_{L_{\Delta_p}}(v_0)= \inf\limits_{\XO} J_{L_{\Delta_p}}.$$ By density, let $(v_k)_{k\in \N}\subset C^\infty_c(\Omega)$ such that $v_k \to v_0$ in $\XO$, as $k\to \infty$. Using the fact that $u_s$ is a least energy solution to \eqref{frac_pblm}, for every $k \in \N$, we have
  \[
  \frac{\mu}{p^2} \lim_{s \to 0^+} \nWs^p= \lim_{s \to 0^+} \l(\frac{1}{p}-\frac{1}{q(s)}\r) \frac{\nWs^p}{s} = \lim_{s \to 0^+} \frac{J_{s,p}(u_s)}{s}\leq \lim_{s \to 0^+} \frac{J_{s,p}(v_k)}{s}.\]
  Applying \eqref{J/s} for $v_k$ and using the above observations, we get
  \[
  \begin{split}
  \inf_{\XO}  J_{L_{\Delta_p}} \leq J_{L_{\Delta_p}}(u_1) = \frac{\mu}{p^2} \|u_1\|^p_{L^p(\Omega)} &= \frac{\mu}{p^2} \lim_{s \to 0^+} \nWs^p\\&\leq \lim_{s \to 0^+} \frac{J_{s,p}(v_k)}{s}=J_{L_{\Delta_p}}(v_k).
  \end{split}\]
  Now, by passing limits $k \to \infty$, we obtain
  \[\lim_{k \to \infty} J_{L_{\Delta_p}}(v_k)= J_{L_{\Delta_p}}(v_0)= \inf_{\XO}  J_{L_{\Delta_p}}.\]
Combining all the above, we get that $u_1$ is  a least energy solution to the problem \eqref{lim_pblm}. Moreover, by Theorem \ref{Existence L2}, we obtain $u_1 \in L^\infty(\Omega)$ and that $u_1$ is unique.
This completes the proof of the theorem.
\qed
\appendix
\section{}
\renewcommand{\thelemma}{\thesection.\arabic{lemma}}
\renewcommand{\thetheorem}{\thesection.\arabic{theorem}}
\label{Appendix}
In this section, we recall some known elementary inequalities and prove preliminary lemma useful in deriving the asymptotics of the constant $B_{N,s,p}$ involved in the fractional Sobolev inequality in \cite[Corollary 4.2]{Frank-Seiringer} and \cite[Theorem 1]{Maz'ya-Shaposhnikova} and normalization constant $C_{N,s,p}$ in \eqref{C-nsp}.
\begin{lemma}
    \label{Lemma1-Lindgren} \cite[Lemma 1 and 2]{Lindgren}
    For $a,b \in \R$, the following hold true:
    \[
    \begin{split}
    |h(a+b)-h(a)|\leq
        \begin{cases}
          (3^{p-1}+2^{p-1}) |b|^{p-1} \quad &\text{if} \  p \in (1,2),\\
          (p-1)|b|(|a|+|b|)^{p-2} \quad & \text{if} \  p \geq 2.
        \end{cases}
    \end{split}
    \]
\end{lemma}
\begin{lemma}
    \label{section2.2-Mosconi} \cite[Section 2.2]{Iannizzotto-Mosconi-Papageorgiou}
    For $a,b \geq 0$, the following hold true:
    \[
    \begin{split}
        (a+b)^q \leq
    \begin{cases}
        a^q+b^q \quad &\text{if} \ q \in (0,1),\\
        2^{q-1} \l(a^q+b^q\r) \quad &\text{if} \ q\geq 1.
    \end{cases}
    \end{split}
    \]
\end{lemma}

\begin{theorem}
    \label{thm3-MS}\cite[Theorem 3]{Maz'ya-Shaposhnikova}
    For any $u \in C_c^\infty(\mathbb{R}^N),$
    \[\lim\limits_{s \to 0} s\int_{\mathbb{R}^N} \int_{\mathbb{R}^N} \pfrac ~dx ~dy= \frac{2|\mathbb{S}^{N-1}|}{p}\|u\|^p_{L^{p}(\mathbb{R^N)}}.\]
\end{theorem}
\begin{lemma}\label{prelim-lemma}
    Let $h:[0,1] \to \mathbb{R}^+$ be a continuous function. Then, we have
    \[
    \lim_{s \to 0^+} s \int_0^1 r^{sp-1} h(r) ~dr = \frac{h(0)}{p}.
    \]
\end{lemma}
\begin{proof}
Let $\eps>0$. Choose $\delta>0$ such that
\[
|h(r) - h(0)| < \eps \quad \text{for all}  \quad 0<r< \delta.
\]
Then, we have
\[
\begin{split}
s \int_0^1 r^{sp-1} h(r) ~dr &= s \int_0^\delta r^{sp-1} h(r) ~dr + s \int_\delta^1 r^{sp-1} h(r) ~dr \\
& \leq (h(0) + \eps) \frac{\delta^{sp}}{p} + \frac{\|h\|_{L^\infty([\delta,1])}}{p} (1-\delta^{sp})
\end{split}
\]
and
\[
\begin{split}
s \int_0^1 r^{sp-1} h(r) ~dr & \geq s \int_0^\delta r^{sp-1} h(r) ~dr \geq (h(0) - \eps) \frac{\delta^{sp}}{p}.
\end{split}
\]
Combining the above estimates, we obtain
\[
\begin{split}
\frac{h(0)}{p} (\delta^{sp}-1) - \eps \frac{\delta^{sp}}{p} &\leq s \int_0^1 r^{sp-1} h(r) ~dr - \frac{h(0)}{p} \\
& \leq \frac{h(0)}{p} (\delta^{sp}-1) + \eps \frac{\delta^{sp}}{p}  + \frac{\|h\|_{L^\infty([\delta,1])}}{p} (1-\delta^{sp}).
\end{split}
\]
Since $\eps>0$ is arbitrary, passing limits $s \to 0^+$, we obtain the required claim.
\end{proof}

\section*{Acknowledgments}

The first author is funded by ANRF Research Grant SRG/2023/000308, India. The second author is partially funded by IFCAM (Indo-French Centre for Applied Mathematics) IRL CNRS 3494. The fourth author is funded by the UGC Junior Research Fellowship with reference no. 221610015405. 



\begin{thebibliography}{99}

 \bibitem{Angeles-Saldana}
F. Angeles and A. Salda\~na, {\it Small order limit of fractional Dirichlet sublinear-type problems}, Fract. Calc. Appl. Anal. {\bf 26} (2023), no. 4, 1594-1631.

\bibitem{Antil-Bartels} H. Antil and S. Bartels, {\it Spectral approximation of fractional PDEs in image processing and phase field modeling}, Comput. Methods Appl. Math. {\bf 17} (2017), no.~4, 661--678; MR3709055

\bibitem{Arora-Crespo-Blanco-Winkert}
R. Arora, \'{A}. Crespo-Blanco and P. Winkert, {\it On logarithmic double phase problems}, J. Differential Equations {\bf 433} (2025), Paper No. 113247, 60 pp.

\bibitem{Arora-Giacomoni-Vaishnavi} R. Arora, J. Giacomoni and A. Vaishnavi, {\it The Brezis-Nirenberg and logistic problem for the Logarithmic Laplacian}, arXiv preprint arXiv:2504.18907 (2025). 

\bibitem{Beckner} W. Beckner, {\it Pitt’s inequality and the uncertainty principle}, Proc. Amer. Math. Soc. {\bf 123} (1995), no. 6, 1897–1905.

\bibitem{Biagi et. al} S. Biagi, S. Dipierro, E. Valdinoci and E. Vecchi, {\it A Brezis-Nirenberg type result for mixed local and nonlocal operators}, NoDEA, Nonlinear Differ. Equ. Appl. {\bf32} (2025), no. 4, Paper No. 62, 28 pp.



\bibitem{Foldes-Saldana} 
 D. Bonheure, J. F\"oldes, E. Moreira dos Santos, A. Salda\~na and H. Tavares, {\it Paths to uniqueness of critical points and applications to partial differential equations}, Trans. Amer. Math. Soc. {\bf 370} (2018), no. 10, 7081–7127. 

 \bibitem{Brasco-Parini}
 L. Brasco and E. Parini, {\it The second eigenvalue of the fractional p-Laplacian}, Adv. Calc. Var. {\bf 9} (2016), no. 4, 323–355.

\bibitem{Brezis-Nirenberg} H. Brezis and L. Nirenberg, {\it Positive solutions of nonlinear elliptic equations involving critical Sobolev exponents}, Comm. Pure Appl. Math. {\bf 36} (1983), no. 4, 437–477.


\bibitem{Caffarelli} L.~\'A. Caffarelli, {\it Non-local diffusions, drifts and games}, Nonlinear partial differential equations, 37--52, Abel Symp., 7, Springer, Heidelberg.

 \bibitem{Caffarelli-Dipierro-Valdinoci}  L.~\'A. Caffarelli, S. Dipierro and E. Valdinoci, {\it A logistic equation with nonlocal interactions}, Kinet. Relat. Models {\bf  10} (2017), no. 1, 141–170.

\bibitem{Caffarelli-Silvestre} L.~\'A. Caffarelli and L.~E. Silvestre, {\it An extension problem related to the fractional Laplacian}, Comm. Partial Differential Equations {\bf 32} (2007), no. 7-9, 1245-1260.

 \bibitem{Castro-Kuusi-Palatucci}
 A. Di Castro, T. Kuusi and G. Palatucci, {\it Local behavior of fractional p-minimizers}, Ann. Inst. H. Poincar\`e, Anal. Non Lin\`eaire {\bf 33} (2016), no. 5, 1279–1299.


\bibitem{Chen-Hauer-Weth} H. Chen, D. Hauer and T. Weth, {\it An extension problem for the logarithmic Laplacian}, arXiv preprint, 	arXiv:2312.15689  (2023).

\bibitem{Chen-Veron} H. Chen and L. V\'eron, { \it Bounds for eigenvalues of the Dirichlet problem for the logarithmic Laplacian}, Adv. Calc. Var. {\bf 16} (2023), no. 3, 541-558.

\bibitem{Chen-Veron-1} H. Chen and L. V\'eron, {\it The Cauchy problem associated to the logarithmic Laplacian with an application to the fundamental solution}, J. Funct. Anal. {\bf 287} (2024), no.~3, Paper No. 110470, 72 pp.


\bibitem{Chen-Weth} H. Chen and T. Weth, {\it The Dirichlet problem for the logarithmic Laplacian}, Comm. Partial Differential Equations {\bf 44} (2019), no. 11, 1100-1139. 


\bibitem{Chen-Zhou} H. Chen and F. Zhou, {\it On positive solutions of critical semilinear equations involving the logarithmic Laplacian}, arXiv preprint arXiv:2409.04797 (2024).



\bibitem{Correa-DePablo} E. Correa and A. de Pablo, {\it Nonlocal operators of order near zero}, J. Math. Anal. Appl. {\bf 461} (2018), no. 1, 837–867.

\bibitem{Costa-Drabek-Tehrani} D. G. Costa, P. Dr\'abek and H. T. Tehrani, {\it Positive solutions to semilinear elliptic equations with logistic type nonlinearities and constant yield
 harvesting in $\R^N$}, Comm. Partial Differential Equations {\bf33} (2008), no. 7-9, 1597–1610.

\bibitem{Crismale et. al}V. Crismale, L. De Luca, A. Kubin, A. Ninno and M. Ponsiglione, {\it The variational approach to $s$-fractional heat flows and the limit cases $s\to0^+$ and $s\to1^-$}, J. Funct. Anal. {\bf 284} (2023), no.~8, Paper No. 109851, 38 pp.

\bibitem{Nezza-Palatucci-Valdinoci}
 E. Di Nezza, G. Palatucci and E. Valdinoci, {\it Hitchhiker’s guide to the fractional Sobolev spaces}, Bull. Sci. Math. {\bf136} (2012), no. 5, 521-573.
 
 \bibitem{Drabek-Pohozaev} P. Drabek and S. I. Pohozaev, {\it Positive solutions for the p-Laplacian: application of the fibering method}, Proc. Roy. Soc. of Edinburgh Sec. A {\bf 127} (1997), 703-726.
 
\bibitem{Dyda-Jarohs-Sk} B. Dyda, S. Jarohs and F. SK, {\it The Dirichlet problem for the logarithmic $p$-Laplacian}, accepted in Trans. Amer. Math. Soc. (2025).

 \bibitem{Ekeland} I. Ekeland,  {\it On the variational principle,} J. Math. Anal. Appl. {\bf47} (1974), 324–353.


\bibitem{Felsinger-Kassmann-Voigt} M. Felsinger, M. Kassmann and P. Voigt, {\it The Dirichlet problem for nonlocal operators}, Math. Z. {\bf 279} (2015), 779–809.

\bibitem{Fernandez-Saldana} J. C. Fern\`andez and A. Salda\~na, {\it The conformal logarithmic Laplacian on the sphere: Yamabe-type problems and Sobolev spaces}, arXiv preprint 	arXiv:2507.21779 (2025)

\bibitem{Feulefack} P. A. Feulefack, {\it The logarithmic Schr\"odinger operator and associated Dirichlet problems}, J. Math. Anal. Appl. {\bf 517} (2023), no. 2, 126656.

\bibitem{Feulefack-Jarohs} P.A. Feulefack and S. Jarohs, {\it Nonlocal operators of small order}, Ann. Mat. Pura Appl.(4) {\bf 202} (2023), no. 4, 1501-1529.

\bibitem{Feulefack-Jarohs-Weth} P.A. Feulefack, S. Jarohs and T. Weth, {\it Small order asymptotics of the Dirichlet eigenvalue problem for the fractional Laplacian}, J. Fourier Anal. Appl. {\bf 28} (2022), no. 2, Paper No. 18, 44 pp.

\bibitem{Foghem} G. Foghem, {\it Stability of complement value problems for p-L\`evy operators}, NODEA Nonlinear Differential Equations Appl. {\bf 32} (2025), no. 1, Paper No. 1, 106 pp.

\bibitem{Frank-Konig-Tang} R.L. Frank, T. K\"onig and H. Tang, {\it Classification of solutions of an equation related to a conformal log Sobolev inequality}, Adv. Math. {\bf 375} (2020), 107395, 27 pp.

\bibitem{Frank-Seiringer} R. L. Frank and R. Seiringer, {\it Non-linear ground state representations and sharp Hardy inequalities,} J. Funct. Anal. {\bf 255} (2008), no. 12, 3407–3430.

 \bibitem{Franzina-Palatucci} G. Franzina and G. Palatucci, {\it Fractional p-eigenvalues}, Riv. Mat. Univ. Parma (N.S.) {\bf 5} (2014), no. 2, 373-386.



 \bibitem{Harjulehto-Hasto}
P. Harjulehto and P. H\"{a}st\"{o}, Orlicz Spaces and generalized Orlicz Spaces, Lecture Notes in Math., 2236, Springer, Cham, 2019. x+167 pp.

\bibitem{Harrach-Lin-Weth} B. Harrach, Y.-H. Lin and T. Weth, {\it The Calder\'on problem for the logarithmic Schr\"odinger equation}, J. Differential Equations {\bf 444} (2025), Paper No. 113665, 25 pp.


\bibitem{Santamaria et. al}V. Hern\'andez-Santamar\'ia, S. Jarohs, A. Salda\~na and L. Sinsch, {\it FEM for 1D-problems involving the logarithmic Laplacian: error estimates and numerical implementation}, Comput. Math. Appl. {\bf 192} (2025), 189-211.


\bibitem{Santamaria-Rios-Saldana} V. Hern\'andez Santamar\'ia, L. F. L\'opez R\'ios and A. Salda\~na, {\it Optimal boundary regularity and a Hopf-type lemma for Dirichlet problems involving the logarithmic Laplacian}, Discrete Contin. Dyn. Syst. {\bf 45} (2025), no. 1, 1-36.

\bibitem{Santamaria-Saldana}
V. Hern\'andez Santamar\'ia and A. Salda\~na, 
{\it Small order asymptotics for nonlinear fractional problems}, Calc. Var. Partial Differential Equations {\bf 61} (2022), no. 3, Paper No. 92, 26 pp.

\bibitem{Ho et.al} K. Ho, K. Perera and I. Sim, {\it On the Brezis-Nirenberg problem for the $(p, q)$ Laplacian}, Ann. Mat. Pura Appl. {\bf202} (2023), no. 4, 1991-2005.

\bibitem{Iannizzotto-Mosconi-Papageorgiou} A. Iannizzotto, S. Mosconi and N. S. Papageorgiou, {\it  On the logistic equation for the fractional p-Laplacian},  Math. Nachr. {\bf 296} (2023), no. 4, 1451–1468.

\bibitem{Jarohs-Saldana-Weth} S. Jarohs, A. Salda\~na and T. Weth, {\it A new look at the fractional Poisson problem via the logarithmic Laplacian}, J. Funct. Anal. {\bf 279} (2020), no. 11, 108732, 50 pp.


\bibitem{Kassmann-Mimica} M. Kassmann and A. Mimica, {\it Intrinsic scaling properties for nonlocal operators}, J. Eur. Math. Soc. (JEMS) {\bf 19} (2017), no. 4, 983–1011.


\bibitem{Jarohs-Saldana-Weth-1} S. Jarohs, A. Salda\~na and T. Weth, {\it Differentiability of the nonlocal-to-local transition in fractional Poisson problems}, Potential Anal. {\bf 63} (2025), no. 1, 77-99.

\bibitem{Lara-Saldana} H. A. Chang-Lara and A. Salda\~na, {\it Classical solutions to integral equations with zero order kernels},  Math. Ann. {\bf 389} (2024), no. 2, 1463–1515.

\bibitem{Lindgren}
 E. Lindgren, {\it H\"older estimates for viscosity solutions of equations of fractional p-Laplace type}, NoDEA Nonlinear Differential Equations Appl. {\bf23} (2016), no. 5, Art. 55, 18 pp. 


\bibitem{Luca et. al} L. De Luca, M. Novaga and M. Ponsiglione, {\it The 0-fractional perimeter between fractional perimeters and Riesz potentials}, Ann. Sc. Norm. Super. Pisa Cl. Sci. (5) {\bf 22} (2021), no. 4, 1559-1596.

\bibitem{Maz'ya-Shaposhnikova} V. Maz’ya and T. Shaposhnikova, {\it On the Bourgain, Brezis, and Mironescu theorem concerning limiting embeddings of fractional Sobolev spaces}, J. Funct. Anal. {\bf195} (2002), no. 2, 230–238.


\bibitem{Pellacci-Verzini}  B. Pellacci and G. Verzini, {\it Best dispersal strategies in spatially heterogeneous environments: optimization of the principal eigenvalue for indefinite fractional Neumann problems}, J. Math. Biol. {\bf 76} (2018), no. 6, 1357–1386.

\bibitem{Pezzo-Quaas} L. M. del Pezzo and A. Quaas, {\it A Hopf’s lemma and a strong minimum principle for the fractional p-Laplacian}, J. Differential Equations {\bf 263} (2017), no. 1, 765-778.

\bibitem{Pollastro-Soave} L. Pollastro and N. Soave, {\it Antisymmetric maximum principles and Hopf's lemmas for the logarithmic Laplacian, with applications to symmetry results}, Ann. Mat. Pura Appl. {\bf 204} (2025), no. 4, 1827-1845.


\bibitem{Servadei-Valdinoci}  R. Servadei and E. Valdinoci, {\it The Brezis-Nirenberg result for the fractional Laplacian}, Trans.
Amer. Math. Soc. 367 (2015), no. 1, 67-102.


\bibitem{Sikic-Song-Vondracek} H. \u Siki\'c, R. Song and Z. Vondrac\u ek, {\it Potential theory of geometric stable processes}, Probab. Theory Related Fields {\bf 135} (2006), no. 4, 547–575.

\bibitem{Sprekels-Valdinoci} J. Sprekels and E. Valdinoci, {\it A new type of identification problems: optimizing the fractional order in a nonlocal evolution equation}, SIAM J. Control Optim. {\bf 55} (2017), no. 1, 70–93.

\bibitem{Szulkin-Weth} A. Szulkin and T. Weth,
The method of Nehari manifold, Handbook of nonconvex analysis and applications, 597–632, International Press, Somerville, MA, 2010.

 \bibitem{Temgoua-Weth} R.~Y. Temgoua and T. Weth, {\it The eigenvalue problem for the regional fractional Laplacian in the small order limit}, Potential Anal. {\bf 60} (2024), no. 1, 285-306.

\bibitem{Vazquez} J.~L. V\'azquez, {\it Recent progress in the theory of nonlinear diffusion with fractional Laplacian operators}, Discrete Contin. Dyn. Syst. Ser. S {\bf 7} (2014), no. 4, 857-885.

 \bibitem{Willem} M. Willem, Minimax theorems, Progr. Nonlinear Differential Equations Appl., {\bf24}, Birkhäuser Boston, Inc., Boston, MA, 1996. x+162 pp.


\end{thebibliography}
\end{document}